\documentclass[a4paper,10pt]{article}
\usepackage[english]{babel}
\usepackage[latin1]{inputenc}
\usepackage[T1]{fontenc}
\usepackage{textcomp} 
\usepackage{amsmath}
\usepackage{amsfonts}
\usepackage{amsthm}
\usepackage{mathrsfs}
\usepackage{amssymb,dsfont}
\usepackage{mathabx}
\usepackage{fnpct}
\usepackage{csquotes}
\usepackage{graphicx}
\usepackage{color}
\usepackage[mathscr]{euscript}
\DeclareSymbolFont{rsfs}{U}{rsfs}{m}{n}
\DeclareSymbolFontAlphabet{\mathscrsfs}{rsfs}

\newcommand{\footremember}[2]{%
   \footnote{#2}
    \newcounter{#1}
    \setcounter{#1}{\value{footnote}}%
}
\newcommand{\footrecall}[1]{%
    \footnotemark[\value{#1}]%
}

\theoremstyle{definition}
\newtheorem{Def}{Definition}[section]

\newtheorem{Rmk}[Def]{Remark}

\theoremstyle{plain}
\newtheorem{Prop}[Def]{Proposition}
\newtheorem{Thm}[Def]{Theorem}
\newtheorem{Lemma}[Def]{Lemma}

\newcommand{\R}{\mathbb{R}}
\newcommand{\E}{\mathbb E}

\newcommand{\Z}{\mathbb{Z}}
\newcommand{\N}{\mathbb{N}}

\DeclareMathOperator{\supp}{supp}

\newcommand{\ep}{\varepsilon}

\newcommand{\dist}{\text{\rm dist}}

\newcommand{\PP}{\mathbb P}

\newcommand{\lp}{\left(}
\newcommand{\rp}{\right)}
\newcommand{\lc}{\left[}
\newcommand{\rc}{\right]}
\newcommand{\lcl}{\left\{}
\newcommand{\rcl}{\right\}}
\newcommand{\lln}{\left|}
\newcommand{\rrn}{\right|}

\newcommand{\wt}[1]{\widetilde{#1}}
\newcommand{\wc}[1]{\widecheck{#1}}

\newcommand{\ros}[1]{R_{H_1,H_2} \left( #1 \right)}

\newcommand{\norm}[1]{\left\lVert#1\right\rVert_{L^2(\Omega)}}
\newcommand{\norminf}[1]{\left\lVert#1\right\rVert_\infty}

\renewcommand{\epsilon}{\varepsilon}

\DeclareFontFamily{U} {cmsy}{}

\DeclareFontShape{U}{cmsy}{m}{n}{
  <-8> cmsy7
  <8-9> cmsy8
  <9-10> cmsy9
  <10-> cmsy10}{}

\DeclareSymbolFont{Xcmsy} {U} {cmsy}{m}{n}

\DeclareMathSymbol{\cmemptyset}{\mathord}{Xcmsy}{59}

\def\dist{\mbox{\rm dist}}

\title{ Wavelet methods to study the pointwise regularity of the generalized Rosenblatt process }
\author{L. Daw\footremember{Luxembourg}{University of Luxembourg,Department of Mathematics (DMATH), Maison du Nombre, 6, avenue de la Fonte, L-4364 Esch-sur-Alzette, Grand Duchy Of Luxembourg.}, L. Loosveldt\footrecall{Luxembourg} \footnote{Corresponding author at laurent.loosveldt@uni.lu.}}

\begin{document}

\maketitle

\begin{abstract}
We prove that we can identify three types of pointwise behaviour in the regularity of the (generalized) Rosenblatt process. This extends to a non Gaussian setting previous results known for the (fractional) Brownian motion. On this purpose, fine bounds on the increments of the Rosenblatt process are needed. Our analysis is essentially based on various wavelet methods.
\end{abstract}

\section{Introduction}

Precise study of path behaviour, and in particular regularity, of stochastic processes is a classical research field, initiated in the 1920s by the works of Wiener \cite{wiener76}. It lies in between probability and (harmonic) analysis and a common strategy is to mix probabilistic arguments with analytical tools. Pioneer works concerned Brownian motion. Among them, one can cite Paley and Wiener's expansion \cite{paleywiener} using Fourier series, L\'evy's representation \cite{levy48} obtained with some techniques of interpolation theory or, more recently, Kahane's expansion \cite{kahane85} in the Schauder basis. 

In the last decades, the emergence of wavelet analysis allowed to obtain series expansions for many stochastic processes. Let $\psi \, : \, \R \to \R$ be a smooth function satisfying the admissibility condition \cite{Meyer:95}
\begin{equation}\label{intro:admcond}
\int_{\R} \frac{|\widehat{\psi}(\xi)|}{|\xi|} \, d\xi < \infty,
\end{equation}
where $\widehat{\psi}$ is the Fourier transform of $\psi$. As such it generates an orthonormal basis of $L^{2}(\R)$. More precisely, any function $f\in L^2(\R)$ can be decomposed as
\begin{equation}\label{intro:waveletexp}
f= \sum_{j \in \Z}\sum_{k \in \Z} c_{j,k} \psi(2^j
\cdot-k),
\end{equation}
where
\[
c_{j,k}=2^{j}\int_{\R}f(x) \psi(2^jx-k)\, dx.
\]
It is noteworthy that the expansion \eqref{intro:waveletexp} holds true in many function spaces. We refer to the seminal books \cite{Daubechies:92,Meyer:95,Mallat:99} for more details and proofs of these facts. Multifractal analysis has demonstrated the efficiency of wavelet methods to study uniform and pointwise H\"older regularity of functions both from a theoretical \cite{barralseuret,Bastin:14,Clausel:11,jaff97siam,Jaffard:04b,jaffardmeyer96} and a practical point of view \cite{arneodoadn,Arneodo:95,brundekropp,deklsn:17,flandrin,jaffardabryroux,Nicolay:07,wendtabryjaf,wendtbis}. 

Now, let us consider a probability space $(\Omega,\mathcal{A},\mathbb{P})$ and a real-valued stochastic process $X$ defined on it. If $X$ is smooth enough, for all $\omega \in \Omega$, one can apply expansion \eqref{intro:waveletexp} to the simple path $t \mapsto X(t,\omega)$. This way, one defines a sequence of random wavelet coefficients $(c_{j,k}(\omega))_{j,k \in \Z}$. For instance, if $X=B_H$ is the fractional Brownian motion of Hurst index $H \in (0,1)$ and if $\psi$ is a sufficiently regular wavelet, one has \cite{meyersellantaqqu99,jaffardabryroux}
\begin{equation}\label{intrpo:bh}
B_{H} = \sum_{j \in \N}\sum_{k \in \Z} 2^{-Hj} \xi_{j,k} \psi_{H+1/2} (2^{j}\cdot-k)
 + R,
\end{equation}
where $(R(t,\cdot))_{t \in \R^+}$ is a  process with alsmost surely $C^\infty$ sample paths, $(\xi_{j,k})_{j\in \N, k \in \Z}$ is a sequence of independent $\mathcal{N}(0,1)$ random variables and $\psi_{H+1/2}$ is a fractional antiderivative of $\psi$, see Section \ref{sec:wav} for a precise definition. 

In \cite{esserloosveldt}, Esser and Loosveldt undertook  a systematic study of Gaussian wavelet series. Thanks to \eqref{intrpo:bh}, it applies in particular to the fractional Brownian motion and leads to the following theorem.
\begin{Thm}
\label{thm:FBM}
For all $H \in (0,1)$, there exists an event $\Omega_H$ of probability $1$ satisfying the following assertions for all $\omega \in \Omega_H$ and every non-empty interval $I$ of $\R$.
\begin{itemize}
\item For almost every $t\in I$,
\[
0< \limsup_{s \to t} \dfrac{|B_H(t,\omega)-B_H(s,\omega) |}{|t-s|^{H} \sqrt{\log \log |t-s|^{-1}}} < + \infty.
\]
Such points  are called \emph{ordinary points}.
\item There exists a dense set of points $t \in I$  such that
\[
0< \limsup_{s \to t} \dfrac{|B_H(t,\omega)-B_H(s,\omega) |}{|t-s|^{H}  \sqrt{\log |t-s|^{-1}}} < + \infty.
\]
Such points  are called \emph{rapid points}.
\item There exists a dense set of points $t \in I$  such that
\[
0<\limsup_{s \to t} \dfrac{|B_H(t,\omega)-B_H(s,\omega) |}{|t-s|^{H} } < + \infty.
\]
Such points  are called \emph{slow points}.
\end{itemize}
\end{Thm}
Note that Theorem \ref{thm:FBM} extends some well-known results of Kahane concerning the Brownian motion \cite{kahane85}. The ``ordinary'', ``rapid'' and ``slow'' terminology is inspired by them. Let us justify it. In a measure-theoretical point of view, the modulus of continuity $x \mapsto |x|^H \sqrt{\log \log |x|^{-1}}$ is the most frequent among the points of singles paths. Thus, it is natural to refer it as ordinary. Now, $|x|^H \sqrt{\log \log |x|^{-1}} = o(|x|^H \sqrt{\log |x|^{-1}})$ if $x \to 0^+$ and thus points for which $x \mapsto |x|^H \sqrt{\log |x|^{-1}}$ is the pointwise modulus of continuity are refereed as rapid. On the other side, points for which $x \mapsto |x|^H $ is the pointwise modulus of continuity are referred as slow because $|x|^H = o(|x|^H \sqrt{\log \log |x|^{-1}})$ if $x \to 0^+$.  

\bigskip

Now, let us turn to the stochastic process we will deal with in this paper. The Rosenblatt process appears naturally as a limit of normalized sums of long-range dependent random variables \cite{dobrushin}. Like the fractional Brownian motion, it belongs to the class of Hermite processes, fractional Brownian motion being of order $1$ while Rosenblatt process is of order $2$. Both are selfsimilar stochastic processes with stationary increments and are characterized by a parameter $H$, called the Hurst exponent. However, unlike the fractional Brownian motion, the Rosenblatt process is not Gaussian. Does it make a big difference regarding ordinary, rapid and slow points? In other words, can Theorem \ref{thm:FBM} be extended to cover the non Gaussian Rosenblatt process?

For the last fifteen years the Rosenblatt process has received a significantly increasing interest in both theoretical and practical lines of research. Due to its self-similarity, its applications are numerous across a multitude of fields, including internet traffic \cite{chaurasia2019performance} and turbulence \cite{sakthivel2018retarded, lakhel2019existence}. From a statistical point of view, estimating the value of the Hurst index $H$ is important for practical applications and various estimators exist, see \cite{bardet-tudor-2010, tudor-viens-2009}. Also, from a mathematical point of view the Rosenblatt process has received a lot of interest since its inception in~\cite{rosenblatt-1956}. Its distribution, still not known in explicit form, was studied first in~\cite{albin-1998} and more recently in~\cite{maejima-tudor-2013} and~\cite{veillette-taqqu-2013}.

In this paper, we even consider a generalization of the Rosenblatt process, as defined and studied in \cite{tudor12}. It depends on two parameters $H_1,H_2 \in (\frac{1}{2},1)$ which are such that $H_1+H_2> \frac{3}{2}$. The generalized Rosenblatt process $\{R_{H_1,H_2}(t,\cdot)\}_{t \in \R_+}$ is defined as a double Wiener-It\^o integral of a kernel function $K_{H_1,H_2}$ with respect to a given Brownian motion. More precisely, consider a standard two-sided Brownian motion $B$, and set
\begin{equation}\label{eqn:introdefrosen}
R_{H_1,H_2}(t,\cdot) = \int_{\R^2}^{'}  K_{H_1,H_2}(t,x_1,x_2) \, dB(x_1) dB(x_2),
\end{equation}
where $\int_{\R^2}^{'}$ denotes integration over $\R^2$ excluding the diagonal. The kernel function in \eqref{eqn:introdefrosen} is expressed, for all $(t,x_1,x_2)$ on $\R_+ \times \R^2$, by
\[ K_{H_1,H_2}(t,x_1,x_2) = \frac{1}{\Gamma \left(H_1-\frac{1}{2}\right) \Gamma \left(H_2-\frac{1}{2}\right)} \int_0^t (s-x_1)_+^{H_1-\frac{3}{2}} (s-x_2)_+^{H_2-\frac{3}{2}} \, ds,\]
where $\Gamma$ stands for the usual Gamma Euler function, and where for $(x,\alpha) \in \R^2$
\begin{align*}
x_+^\alpha=\left\{ \begin{array}{cl}
                       x^\alpha & \text{if } x \geq 0\\
                       0 & \text{otherwise.}
                     \end{array} \right.
\end{align*}
 Note that the (standard) Rosenblatt process is the process $\{R_{H,H}(t,\cdot)\}_{t \in \R_+}$ for $H \in (3/4,1)$. The generalized Rosenblatt process $\{R_{H_1,H_2}(t,\cdot)\}_{t \in \R_+}$ is non-Gaussian, belongs to the second Wiener chaos, and has the following basic properties:
\begin{enumerate}
	\item[\textbf{(1)}] \textbf{Continuity:} the trajectories of the Rosenblatt process $R_{H_1,H_2}$ are continuous. 
	\item[\textbf{(2)}] \textbf{Stationary increments: } $R_{H_1,H_2}$ has stationary increments; that is, the distribution of the process $\left\{ R_{H_1,H_2}(t+s,\cdot)- R_{H_1,H_2}(s,\cdot)\right\}_{t \in \R_+}$ does not depend on $s \geq 0$.
	\item[\textbf{(3)}] \textbf{Self-similarity: } $R_{H_1,H_2}$ is self-similar with exponent $H_1+H_2-1$; that is, the processes $\{R_{H_1,H_2}(ct,\cdot)\}_{t \in \R_+}$ and $\left\{c^{H_1+H_2-1}R_{H_1,H_2}(t,\cdot)\right\}_{t \in \R_+}$ have the same distribution for all $c>0$.
\end{enumerate}

In \cite{ayacheesmili}, Ayache and Esmili presented a wavelet-type representation of the generalized Rosenblatt process, very similar to the one given in \cite{meyersellantaqqu99} for fractional Brownian motion, excepted for the use of integrals of two-dimensional wavelet bases. This representation is the starting point of this paper. It is one of our key tools to prove the following Theorem \ref{thm:main} which is the main result of this paper.

\begin{Thm}\label{thm:main}
For all $H_1,H_2 \in (\frac{1}{2},1)$ such that $H_1+H_2> \frac{3}{2}$, there exists an event $\Omega_{H_1,H_2}$ of probability $1$ satisfying the following assertions for all $\omega \in \Omega_{H_1,H_2}$ and every non-empty interval $I$ of $\R$.
\begin{itemize}
\item For almost every $t\in I$,
\begin{equation}\label{eqn:mainthm1}
0< \limsup_{s \to t} \dfrac{|R_{H_1,H_2}(t,\omega)-R_{H_1,H_2}(s,\omega) |}{|t-s|^{H_1+H_2-1} \log \log |t-s|^{-1}} < + \infty.
\end{equation}
Such points  are called \emph{ordinary points}.
\item There exists a dense set of points $t \in I$  such that
\begin{equation}\label{eqn:mainthm2}
0< \limsup_{s \to t} \dfrac{|R_{H_1,H_2}(t,\omega)-R_{H_1,H_2}(s,\omega) |}{|t-s|^{H_1+H_2-1}  \log |t-s|^{-1}} < + \infty.
\end{equation}
Such points  are called \emph{rapid points}.
\item There exists a dense set of points $t \in I$  such that
\begin{equation}\label{eqn:mainthm3}
\limsup_{s \to t} \dfrac{|R_{H_1,H_2}(t,\omega)-R_{H_1,H_2}(s,\omega) |}{|t-s|^{H_1+H_2-1} } < + \infty.
\end{equation}
Such points  are called \emph{slow points}.
\end{itemize}
\end{Thm}

Theorem \ref{thm:main} shows in particular that slow, ordinary and rapid points are not specific to Gaussian processes.

\begin{Rmk}
Let us compare Theorems \ref{thm:FBM} and \ref{thm:main}. Each type of points is defined in the same way when considering their pointwise moduli of continuity. Indeed, if $X$ denotes both the fractional Brownian motion or the generalized Rosenblatt process, we see that the asymptotic behaviour of \linebreak $|X(t,\omega)-X(s,\omega)|$ is always compared to a modulus of continuity of the form $|t-s|^\alpha \theta(|t-s|)$, with $\alpha$ corresponding to the self-similarity exponent of $X$ and $\theta$ a potential logarithmic correction. For the ordinary points, $\theta$ is an iterated logarithm. More precisely, for the fractional Brownian motion, we have \linebreak $\theta(|t-s|)= \sqrt{\log \log |t-s|^{-1}}$ while, for the generalized Rosenblatt process, $\theta(|t-s|)=\log \log |t-s|^{-1}$. The same feature appears for the rapid points: in the case of the fractional Brownian motion we have $\theta(|t-s|)= \sqrt{\log |t-s|^{-1}}$ and for the generalized Rosenblatt process we have $\theta(|t-s|)= \log |t-s|^{-1}$. Therefore, the only difference between the corresponding logarithmic corrections is the square root that is used for the fractional Brownian motion and not for the generalized Rosenblatt process. It comes from the estimates that can be done on the tails of the distribution of random variables in the first order Wiener chaos, for the fractional Brownian motion, or the second order, for the generalized Rosenblatt process, see Theorems \ref{thm:bound} and \ref{thm:bound2} below. Concerning the slow points, there is no logarithmic correction, $\theta=1$ in both case. Unfortunately, contrary to the fractional Brownian motion, we did not manage to show the positiveness of the limit in \eqref{eqn:mainthm3}. In fact, for that, we would need to find an almost-sure uniform lower modulus of continuity for the generalized Rosenblatt process and to be able to judge its optimality, which seems to be a difficult task. This is discussed in details in Remark \ref{rmk:last} below, where we give an almost-sure uniform lower modulus of continuity using the techniques we use to prove the positiveness of the limits in \eqref{eqn:mainthm1} and \eqref{eqn:mainthm2}.
\end{Rmk}

Our strategy to prove Theorem \ref{thm:main} is as follows. First, in Section \ref{sec:upper} we derive upper-bounds for the oscillations $|R_{H_1,H_2}(t,\omega)-R_{H_1,H_2}(s,\omega) |$ that are sharp enough to imply the finiteness of the limits \eqref{eqn:mainthm1}, \eqref{eqn:mainthm2} and \eqref{eqn:mainthm3}. This is done by means of the wavelet-type expansion given in \cite{ayacheesmili}, see Theorem \ref{thm:decompo} below. Then, in Section \ref{sec:lower}, we give lower bounds for the so-called wavelet-leaders, see Section \ref{sec:wav}, of the generalized Rosenblatt process on a given compactly supported wavelet basis. This will prove the positiveness of the limits \eqref{eqn:mainthm1}, \eqref{eqn:mainthm2}. In particular, we use different bases depending on whether we deal with the finiteness of the limits in Theorem \ref{thm:main} or with their strict positiveness. This is very different from \cite{esserloosveldt} where the authors always work with the same wavelet. The reason is that the expression \eqref{eqn:waveletserie} in Theorem \ref{thm:decompo} below is not a wavelet series: it involves additional quantities. Therefore, standard arguments linking wavelet coefficients and regularity of the associated functions can no longer be used.

There is a priori no obstacles to extend our results in Section \ref{sec:lower} to any Hermite process. On the contrary, extending the results of Section \ref{sec:upper} does not seem obvious at all. This is because a wavelet-type expansion of arbitrary Hermite process is still missing but also because our strategy relies on arguments which are specific to the two-dimensional feature of the Rosenblatt process, see Lemma \ref{Lemma:intR} for instance. 

Notations used through this paper are rather standard except, maybe, that if $s,t$ are two real numbers,$ \int_{[s,t]}$ stands for $\int_s^t$ if $s \leq t$ and $-\int_s^t = \int_t^s$ otherwise.

\section{Some important facts involving wavelets} \label{sec:wav}

In this section, we gather all the facts concerning wavelets that we will strongly use all along this article. First, an immediate but important consequence of the admissibility condition \eqref{intro:admcond} is that, if the wavelet $\psi \in L^1(\R)$, its first moment always vanishes, i.e.
$$
\int_{\R} \psi(x) dx=0. 
$$
This condition is met for all the wavelets we consider in this paper.

First, while dealing with the upper bounds for the limits in Theorem \ref{thm:main}, we will use a wavelet-type expansion of the generalized Rosenblatt process. It is given in \cite{ayacheesmili} by the mean of the Meyer's wavelet: $\psi$ belongs to the Schwartz class $\mathcal{S}(\R)$, and its Fourier transform is compactly supported, see \cite{lemeyer}. In particular, for all $H \in (1/2,1)$, $\psi_H$, the fractional antiderivative $\widehat{\psi_H}$ of order $H-1/2$ of $\psi$ is well-defined by means of its Fourier transform as
\begin{equation}\label{eqn:fracprim}
 \widehat{\psi_H}(0)=0 \quad \text{and} \quad \widehat{\psi_H}(\xi)=(i\xi)^{-(H-\frac{1}{2})} \widehat{\psi}(\xi), \, \forall \, \xi \neq 0.
\end{equation}
It also belongs to the Schwartz class $\mathcal{S}(\R)$, see \cite{ayache,ayacheesmili,schwartz:78} for instance. Moreover, some standard facts from distribution theory \cite{schwartz:78,ayache} give us the explicit formula
\[ \psi_H(t) = \frac{1}{\Gamma \left(H-\frac{1}{2}\right)} \int_{\R} (t-x)_+^{H-\frac{3}{2}} \psi(x) \, dx.\]

From \eqref{eqn:fracprim}, we see that $\supp(\widehat{\psi_H})= \supp(\widehat{\psi})$ which is the key fact to establish the following lemma, gathering facts already proved in \cite{ayacheesmili}.

\begin{Lemma}\label{Lemma:intR}
Let $H_1,H_2 \in (\frac{1}{2},1)$. If $(j_1,j_2,k_1,k_2) \in \Z^4$ are such that $|j_1-j_2|>1$, then the integral
\[ I_{j_1,j_2}^{k_1,k_2} := \int_\R \psi_{H_1}(2^{j_1}x-k_1) \psi_{H_2}(2^{j_2}x-k_2)\]
vanishes. Moreover, for all $(j,k_1,k_2) \in \Z^3$, we have
\begin{align}
I_{j+1,j}^{k_1,k_2}&=2^{-j} \int_\R e^{-i(k_1-2k_2)\xi} \widehat{\psi_{H_1}}(\xi) \overline{\widehat{\psi_{H_2}}(2\xi)} \, d\xi, \label{eqn:intR1}\\
I_{j,j}^{k_1,k_2}&=2^{-j} \int_\R e^{-i(k_1-k_2)\xi} \widehat{\psi_{H_1}}(\xi) \overline{\widehat{\psi_{H_2}}(\xi)} \, d\xi, \label{eqn:intR2}\\
I_{j,j+1}^{k_1,k_2}&=2^{-j} \int_\R e^{-i(2k_1-k_2)\xi} \widehat{\psi_{H_1}}(2\xi) \overline{\widehat{\psi_{H_2}}(\xi)} \, d\xi.\label{eqn:intR3}
\end{align}

In addition, for all $L>0$, there exists a constant $C_L>0$ such that for all $(j,k_1,k_2) \in \Z^3$,
\begin{align*}
|I_{j+1,j}^{k_1,k_2}| &\leq C_L  \frac{2^{-j}}{(3+|k_1-2k_2|)^L},\\
|I_{j,j}^{k_1,k_2}| &\leq C_L  \frac{2^{-j}}{(3+|k_1-k_2|)^L},\\
|I_{j,j+1}^{k_1,k_2}| &\leq C_L  \frac{2^{-j}}{(3+|2k_1-k_2|)^L}.
\end{align*}
\end{Lemma}

When dealing with the the lower bounds for the limits in Theorem \ref{thm:main}, we use Daubechies compactly supported wavelets \cite{Daubechies:88}. Note that, if $\supp(\Psi) \subseteq [-N,N]$, for a positive integer $N$, then, using the first vanishing moment, for all $(j,k) \in \N \times \Z$ and $t \in \R$, one can write

\begin{equation}\label{eqn:coef}
c_{j,k}= \int_{-N}^{N} \left( f \left(\frac{x+k}{2^j} \right) - f \left(t \right) \right) \Psi(x) \, dx
\end{equation}

Since $\Psi$ is compactly supported, $\Psi(2^j \cdot-k)$ is localized around the dyadic interval
\[\lambda_{j,k} :=   \left[\frac{k}{2^j}, \frac{k+1}{2^j}\right)\]
and it is therefore common to index wavelets these intervals. For simplicity, we sometimes omit any references to the indices $j$ and $k$ for such intervals by writing $\lambda = \lambda_{j,k}$, and $k=s(\lambda)$. Similarly, $c_\lambda$ refers to the quantity $c_{j,k}$. The notation $\Lambda_j$ stands for the set of dyadic intervals $\lambda$ of $\R$ with side length $2^{-j}$. The unique dyadic interval from $\Lambda_j$ containing the point $t \in \R$ is denoted $\lambda_j(t)$. The set of dyadic intervals is $\Lambda:= \cup_{j\in\N} \Lambda_j$. Two dyadic intervals $\lambda$ and $\lambda'$ are adjacent if there exist $j\in\N$ such that $\lambda, \lambda' \in \Lambda_j$ and $\dist(\lambda,\lambda')=0$. The set of dyadic intervals adjacent to $\lambda$ is denoted by $3\lambda$. In this setting, one defines the \textit{wavelet leader} \cite{Jaffard:04b} of $f$ at $t$ and of scale $j$ by
\begin{equation}\label{eqn:leader}
d_j(t_0) = \max_{\lambda \in 3 \lambda_j(t_0)} \sup_{\lambda' \subseteq \lambda} |c_\lambda'|.
\end{equation}
Then, if $\supp(\psi) \subseteq [-N,N]$, from \eqref{eqn:coef}, one can write
\begin{equation}\label{eqn:boundleader}
 d_j(t) \leq 2N \sup_{s \in (t_0-2^{-j}(N+2),t_0+2^{-j}(N+2))} |f(s)-f(t)| \|\psi \|_{L^\infty}.
\end{equation}
When we study stochastic processes, the wavelet leaders are random variables $ d_j(t,\omega)$. Inequality \eqref{eqn:boundleader} with some easy computations implies that in order to obtain the positiveness of the limit \eqref{eqn:mainthm1}, it suffices to show that for all $\omega \in \Omega_{H_1,H_2}$ and all open intervals $I \subseteq \R^+$, for almost every $t \in I$,
\begin{equation} \label{eq:leaderord}
0 < \limsup_{j \to + \infty} \frac{d_j(t,\omega)}{2^{-j(H_1+H_2-1)}\log(j)}.
\end{equation}
Similarly, to prove the positiveness of the limit \eqref{eqn:mainthm2}, we just have to show that for all $\omega \in \Omega_{H_1,H_2}$ and all open intervals $I \subseteq \R^+$, there exists a dense set of points $t \in I$ such that
\begin{equation}\label{eq:leaderrapid}
0 < \limsup_{j \to + \infty} \frac{d_j(t,\omega)}{2^{-j(H_1+H_2-1)} j}.
\end{equation}

\begin{Rmk}
Let us mention that wavelet leaders can not be used to prove the finiteness of the limits in Theorem \ref{thm:main} because they do not precisely characterize the pointwise regularity, see for instance \cite{kreit18,loonic21} for more details.
\end{Rmk}

\section{Upper bounds for oscillations} \label{sec:upper}

Starting from now and until the end of the paper, we fix $H_1,H_2 \in (\frac{1}{2},1)$ such that $H_1+H_2 > \frac{3}{2}$. In this section, we show the finiteness of the limits \eqref{eqn:mainthm1}, \eqref{eqn:mainthm2} and \eqref{eqn:mainthm3}. Concerning the rapid points, we will in fact show a stronger result, obtaining an almost sure \textit{uniform} modulus of continuity for the generalized Rosenblatt process.

We use a wavelet-type expansion of the generalized Rosenblatt process. It relies on the following random variables.

\begin{Def}
For all  $(j_1,j_2,k_1,k_2) \in \Z^4$, let $\varepsilon_{j_1,j_2}^{k_1,k_2}$ be the second order Wiener chaos random variable defined by
\[ 2^{\frac{j_1+j_2}{2}} \int_{\R^2}' \psi(2^{j_1}x_1-k_1) \psi(2^{j_2}x_2-k_2) \, dB(x_1)dB(x_2).\]
\end{Def}

\begin{Rmk}
For all $(j_1,j_2,k_1,k_2) \in \Z^4$, we have (\cite[Proposition 2.3]{ayacheesmili})
\begin{equation}\label{eqn:randomvar1}
\varepsilon_{j_1,j_2}^{k_1,k_2} = \left( 2^\frac{j_1}{2} \int_\R \psi(2^{j_1}x-k_1) dB(x) \right) \left( 2^\frac{j_2}{2} \int_\R \psi(2^{j_2}x-k_2) dB(x) \right)
\end{equation}
for $j_1 \neq j_2$ or $k_1 \neq k_2$, and 
\begin{equation}\label{eqn:randomvar2}
\varepsilon_{j_1,j_1}^{k_1,k_1} = \left( 2^\frac{j_1}{2} \int_\R \psi(2^{j_1}x-k_1) dB(x) \right)^2-1 
\end{equation}
for $j_1=j_2$ and $k_1=k_2$. Using the fact that $(2^{j/2} \psi(2^j \cdot-k))_{(j,k) \in \Z^2}$ forms an orthonormal basis of $L^2(\R)$, and elementary properties of Wiener integral, we know that $(2^{j/2} \int_\R \psi(2^j x-k) \, dB(x))_{(j,k) \in \Z^2}$ is a family of iid $\mathcal{N}(0,1)$ random variables. So the random variables $\varepsilon_{j_1,j_2}^{k_1,k_2}$ and $\varepsilon_{j_1',j_2'}^{k_1',k_2'}$ are independent as soon as $$\{(j_1,k_1),(j_2,k_2) \} \cap \{(j_1',k_1'),(j_2',k_2') \} = \cmemptyset.$$
\end{Rmk}

The following theorem, proved in \cite{ayacheesmili}, gives the wavelet-type expansion we use in this section.

\begin{Thm}\label{thm:decompo}
Let $\psi$ be the Meyer wavelet and $I$ be any compact interval of $\R_+$. Almost surely, the random series
\begin{equation}\label{eqn:waveletserie}
 \sum_{(j_1,j_2,k_1,k_2) \in \Z^4} 2^{j_1(1-H_1) + j_2(1-H_2)} \varepsilon_{j_1,j_2}^{k_1,k_2} \int_0^{t} \psi_{H_1}(2^{j_1} x-k_1) \psi_{H_2}(2^{j_2} x-k_2) \, dx
\end{equation}
converges uniformly to $R_{H_1,H_2}$ on the interval $I$.
\end{Thm}

\begin{Rmk}\label{rmk:mainthm}
Any open interval in $\R$ can be written as a countable union of dyadic intervals $( \lambda_{j,k} )_{j \in \N, k \in \Z}$. Then, to prove Theorem \ref{thm:main}, it is sufficient to show that, for all $j \in \N, k \in \Z$, there exist an event $\Omega_{j,k}$ of probability $1$ such that, for all $\omega \in \Omega_{j,k}$, almost every $t \in \lambda_{j,k}$ is ordinary and there exist $t_r \in \lambda_{j,k}$ which is rapid and $t_s \in \lambda_{j,k}$ which is slow. For the sake of simpleness in notation, we will only do the proofs in full details for $\lambda_{0,0}=[0,1)$. In fact, after dilatation and translation, our proofs hold true for any arbitrary dyadic interval. 
\end{Rmk}

\subsection{Rapid points}

Let us first focus on rapid points. we prove that $x \mapsto |x|^{H_1+H_2-1} \log |x|^{-1}$ is almost surely a uniform modulus of continuity for $R_{H_1,H_2}$.

\begin{Prop} \label{prop:rapid}
There exists an event $\Omega_{\text{rap}}$ of probability $1$ such that for all $\omega \in \Omega_{\text{rap}}$ there exists $C_R(\omega)>0$ such that, for all $t,s \in (0,1)$, we have
\begin{equation}\label{eqn:Prop:rapid}
|R_{H_1,H_2}(t,\omega)-R_{H_1,H_2}(s,\omega)| \leq  C_R(\omega) |t-s|^{H_1+H_2-1} \log |t-s|^{-1}.
\end{equation}
\end{Prop}

Let us set, for all $s,t \in (0,1)$ and $(j_1,j_2,k_1,k_2) \in \Z^4$,
\[ I_{j_1,j_2}^{k_1,k_2}[t,s] = \int_{[t,s]} \psi_{H_1}(2^{j_1}x-k_1) \psi_{H_2}(2^{j_2}x-k_2) \, dx. \]

All along this section, if $s,t \in (0,1)$ are given, $n$ always refers to the unique positive integer such that
\begin{equation}\label{eqn:defofn}
2^{-n-1} < |t-s| \leq 2^{-n}.
\end{equation}
Our proof consists in writing
\begin{equation}\label{eqn:maindiff}
|R_{H_1,H_2}(t,\cdot))-R_{H_1,H_2}(s,\cdot)| = \left | \sum_{(j_1,j_2,k_1,k_2) \in \Z^4} 2^{j_1(1-H_1)} 2^{j_2(1-H_2)} \varepsilon_{j_1,j_2}^{k_1,k_2}  I_{j_1,j_2}^{k_1,k_2}[t,s] \right|
\end{equation}
and to split the sum in the right-hand side in subsums determined according to the position of $j_1$ and $j_2$ with respect to $n$. To bound from above some of these subsums the following lemma is key.
\begin{Lemma}\label{Lemma:ayacheesmili}\cite[Lemma 2.4.]{ayacheesmili} There exist an event $\Omega^*$ of probability $1$ and a positive random variable  $C_1$ with finite moment of any order, such that, for all $\omega \in \Omega^*$ and for each $(j_1,j_2,k_1,k_2) \in \Z^4$,
\begin{equation}\label{eqn:moustique}
|\varepsilon_{j_1,j_2}^{k_1,k_2}(\omega)| \leq C_1(\omega) \sqrt{\log(3+|j_1|+|k_1|)} \sqrt{\log(3+|j_2|+|k_2|)}.
\end{equation}
\end{Lemma}
In view of Lemma \ref{Lemma:ayacheesmili}, we set
\[ L_{j_1,j_2}^{k_1,k_2} =  \sqrt{\log(3+|j_1|+|k_1|)} \sqrt{\log(3+|j_2|+|k_2|)}.\]
As a first step, Lemmata \ref{lem:rapid1} to \ref{lem:rapid4} are devoted to bound some deterministic series whose general term is
\[ 2^{j_1(1-H_1)} 2^{j_2(1-H_2)}  L_{j_1,j_2}^{k_1,k_2} |I_{j_1,j_2}^{k_1,k_2}[t,s]| . \]

This first lemma will be useful to bound the subsums in the right-hand side of \eqref{eqn:maindiff} for $j_1<n$ and $j_2<n$.

\begin{Lemma}\label{lem:rapid1}
There exists a deterministic constant $C>0$ such that, for all $t,s \in (0,1)$, we have
\[
  \sum_{j_1<n} \sum_{j_2<n} \sum_{(k_1,k_2) \in \Z^2} 2^{j_1(1-H_1)} 2^{j_2(1-H_2)}  L_{j_1,j_2}^{k_1,k_2} \left| I_{j_1,j_2}^{k_1,k_2}[t,s] \right| \leq  C |t-s|^{H_1+H_2-1} \log |t-s|^{-1}.
\]
\end{Lemma}

\begin{proof}
Let us start by considering, for all $(j_1,j_2) \in \Z^2$, the series 
\begin{align*}
 R_{j_1,j_2} \, : \, t &\mapsto \sum_{(k_1,k_2) \in \Z^2}  L_{j_1,j_2}^{k_1,k_2}\int_{0}^t |\psi_{H_1}(2^{j_1}x-k_1) \psi_{H_2}(2^{j_2}x-k_2)| \, dx \text{ and}\\
 {R'}_{j_1,j_2} \, : \, t& \mapsto \sum_{(k_1,k_2)\in \Z^2} L_{j_1,j_2}^{k_1,k_2}  |\psi_{H_1}(2^{j_1} t -k_1) \psi_{H_2}(2^{j_2} t-k_2)|.
\end{align*}
The fast decay of the fractional antiderivatives of $\psi$ allows us to write, for all $H \in \{H_1,H_2\}$ and for all $x \in \R$
\begin{equation}\label{eqn:fastdecay}
|\psi_H(x)| \leq C (1+|x|)^{-4}.
\end{equation}
Moreover, according to \cite[Lemma 4.2]{ayacheesmili} for all $L>1$ there exists $C>0$ such that, for all $j \in \Z$ and $x \in \R$
\begin{equation}\label{eqn:sumlog}
\sum_{k \in \Z} \frac{\sqrt{\log(3+|j|+k)}}{(3+|2^j x -k|)^L} \leq C \sqrt{\log(3+|j|+2^j |x|)}.
\end{equation}
Therefore, if $K$ is any compact set of $\R_+$, if $s= \sup_K$, for all $t \in K$, we have
\begin{align*}
|R_{j_1,j_2} (t)| & \leq C \int_0^t  \sqrt{\log(3+|j_1|+2^{j_1} |x|)} \sqrt{\log(3+|j_2|+2^{j_2} |x|)} \,dx \\
& \leq C s \sqrt{\log(3+|j_1|+2^{j_1} s)} \sqrt{\log(3+|j_2|+2^{j_2} s)}. 
\end{align*}
The same arguments can be applied to ${R'}_{j_1,j_2}$, which means that both series converge uniformly on any compact set of $\R_+$. From this,  we can use mean value theorem: for all $(j_1,j_2) \in \Z^2$ there is $\xi(j_1,j_2) \in [s,t]$ such that
\begin{align}
\nonumber & \sum_{(k_1,k_2) \in \Z^2} 2^{j_1(1-H_1)} 2^{j_2(1-H_2)}  L_{j_1,j_2}^{k_1,k_2} \left| I_{j_1,j_2}^{k_1,k_2}[t,s] \right| \\ \quad &\leq |t-s| \sum_{(k_1,k_2) \in \Z^2} L_{j_1,j_2}^{k_1,k_2}   |\psi_{H_1}(2^{j_1} \xi-k_1) \psi_{H_2}(2^{j_2} \xi-k_2)|.\label{eqn:aprestaf}
\end{align}
Now, we use the fast decay of the fractional antiderivatives of $\psi$ \eqref{eqn:fastdecay} and inequality \eqref{eqn:sumlog} to bound \eqref{eqn:aprestaf} from above: for all $j_1,j_2<n$,
\begin{align*}
\sum_{(k_1,k_2) \in \Z^2} & L_{j_1,j_2}^{k_1,k_2}   |\psi_{H_1}(2^{j_1} \xi-k_1) \psi_{H_2}(2^{j_2} \xi-k_2)| \\ & \leq C\left(\sum_{k_1 \in \Z} \frac{\sqrt{\log(3+|j_1|+|k_1|)}}{(3+|2^{j_1}\xi-k_1|)^4} \right) \left( \sum_{k_2 \in \Z} \frac{\sqrt{\log(3+|j_2|+|k_2|)}}{(3+|2^{j_2}\xi-k_2|)^4} \right) \\
& \leq C  \sqrt{\log(3+|j_1|+2^{j_1} |\xi|)} \sqrt{\log(3+|j_2|+2^{j_2} |\xi|)} \\
& \leq C \sqrt{\log(3+|j_1|+2^{j_1} )} \sqrt{\log(3+|j_2|+2^{j_2})},
\end{align*}
as $\xi \in (0,1)$. Let us then remark that
\begin{align}
 \nonumber \sum_{j_1<n}& 2^{j_1(1-H_1)} \sqrt{\log(3+|j_1|+2^{j_1} )} \\
 \nonumber &= \sum_{j_1 \leq 0} 2^{j_1(1-H_1)} \sqrt{\log(3+|j_1|+2^{j_1} )} + \sum_{j_1=0}^{n-1} 2^{j_1(1-H_1)} \sqrt{\log(3+|j_1|+2^{j_1} )} \\
\nonumber & \leq C +   \sum_{j_1=0}^{n-1} 2^{j_1(1-H_1)} \sqrt{\log(3+|j_1|+2^{j_1} )} \\
& \leq C 2^{n(1-H_1)} \sqrt{n},\label{eqn:ineserie1}
\end{align}
as $1-H_1>0$. The same can be applied to the sum over $j_2$ and we finally get
\begin{align*}
&   \sum_{j_1<n} \sum_{j_2<n} \sum_{(k_1,k_2) \in \Z^2} 2^{j_1(1-H_1)} 2^{j_2(1-H_2)}  L_{j_1,j_2}^{k_1,k_2} \left| I_{j_1,j_2}^{k_1,k_2}[t,s] \right| \\  \quad &\leq C  |t-s |\sum_{j_1 < n} \sum_{j_2 < n} 2^{j_1(1-H_1)} 2^{j_2(1-H_2)} \sqrt{\log(3+|j_1|+2^{j_1} )} \sqrt{\log(3+|j_2|+2^{j_2})} \\
& \leq C  |t-s| 2^{n(2-H_1-H_2)}n\\
& \leq C  |t-s|^{H_1+H_2-1} \log |t-s|^{-1}.
\end{align*}
\end{proof}

Lemmata \ref{lem:rapid2} and \ref{lem:rapid3} will help finding an upper bound for the subsums in the right-hand side of \eqref{eqn:maindiff} with $j_1<n \leq j_2$ or $j_2 <n \leq j_1$ as well as the ones where $n \leq j_1 \leq j_2$ and $n \leq j_2 \leq j_1$. Let us define the following partition of $\Z$, which determines the relative positions of $[k_2 2^{-j_2},(k_2+1) 2^{-j_2})$ and $[s,t]$.

\begin{Def}\label{def:partitionofZ}
For all $j_2 \in \N$, we set
\begin{align*}
\Z_{j_2}^<(t,s) &= \{ k_2 \in \Z \, : \, k_2 2^{-j_2} <  \min\{t,s\} \},\\
\Z_{j_2}^>(t,s) &= \{ k_2 \in \Z \, : \, k_2 2^{-j_2}>  \max\{t,s\} \},\\
\text{and }\Z_{j_2}[t,s] &= \Z \setminus (\Z_{j_2}^<(t,s) \cup \Z_{j_2}^>(t,s)).
\end{align*}
\end{Def}
\begin{Rmk}
Note that we have $\# \Z_{j_2}[t,s] \leq 2^{j_2-n}+1$.
\end{Rmk}

Let us also observe that for all $a,b>0$,
\begin{equation}\label{eqn:inelog}
\log(3+a+b) \leq \log(3+a) \log(3+b).
\end{equation}

\begin{Lemma}\label{lem:rapid2}
There exists a deterministic constant $C>0$ such that, for all $t,s \in (0,1)$ and $j_1 \leq j_2$,
the quantities
\begin{equation}\label{eqn:r<>3}
\sum_{k_1 \in \Z} \sum_{k_2 \in \Z_{j_2}^<(t,s)}L_{j_1,j_2}^{k_1,k_2} \left|I_{j_1,j_2}^{k_1,k_2}[t,s]  \right|
\end{equation}
\begin{equation}\label{eqn:r<>2}
\sum_{k_1 \in \Z} \sum_{k_2 \in \Z_{j_2}^>(t,s)}L_{j_1,j_2}^{k_1,k_2} \left| I_{j_1,j_2}^{k_1,k_2}[t,s] \right|
\end{equation}
are bounded from above by
\[
 C\sqrt{\log(3+|j_1|+2^{j_1})} \sqrt{\log(3+|j_2|+2^{j_2})} 2^{-j_2}.
\]
\end{Lemma}

\begin{proof}
Let us bound \eqref{eqn:r<>3}, the proof for \eqref{eqn:r<>2} being similar. From the fast decay of the fractional antiderivatives of $\psi$ \eqref{eqn:fastdecay}, inequalities \eqref{eqn:sumlog} and \eqref{eqn:inelog}
for $j_1 \leq j_2$, we have
\begin{align*}
\eqref{eqn:r<>3} & \leq C  \int_{[s,t]} \left(\sum_{k_1 \in \Z} \frac{\sqrt{\log(3+|j_1|+|k_1|)}}{(3+|2^{j_1}x-k_1|)^4} \right) \left( \sum_{k_2 \in \Z^<_{j_2}(t,s)} \frac{\sqrt{\log(3+|j_2|+|k_2|)}}{(3+|2^{j_2}x-k_2|)^4} \right) \, dx \\
& \leq C \sqrt{\log(3+|j_1|+2^{j_1})} \sqrt{\log(3+|j_2|+2^{j_2})} \\ &	 \quad \int_{[s,t]} \sum_{k_2 \in \Z^<_{j_2}(t,s)} \frac{\sqrt{\log(3+|2^{j_2}x-k_2|)}}{(3+|2^{j_2}x-k_2|)^4}  \, dx.
\end{align*}
For all $x \in [s,t]$ the mapping $y \mapsto 2+ 2^{j_2}x-2^{j_2} \min \{s,t\}+y)^{-3}$ is decreasing and thus
\begin{align}
\nonumber\int_{[s,t]}& \sum_{k_2 \in \Z^<_{j_2}(t,s)} \frac{\sqrt{\log(3+|2^{j_2}x-k_2|)}}{(3+|2^{j_2}x-k_2|)^4}  \, dx \\ 
\nonumber & \leq \int_{[s,t]} \sum_{k_2 \in \Z^<_{j_2}(t,s)} \frac{dx}{(3+2^{j_2}x-k_2)^{3}} \\
\nonumber & \leq \int_{[s,t]} \sum_{m=0}^{+ \infty} \frac{dx}{(3+2^{j_2}x-2^{j_2} \min \{s,t\}+m)^{3}} \\
\nonumber & \leq \int_{[s,t]} \int_0^{+ \infty} \frac{dx dy}{(2+2^{j_2}x-2^{j_2} \min \{s,t\}+y)^{3}} \\
& \leq C 2^{-j_2}.\label{eqn:intavec2}
\end{align}
This bound leads to
\begin{equation}\label{eqn:boundfor18}
\eqref{eqn:r<>3} \leq C \sqrt{\log(3+|j_1|+2^{j_1})} \sqrt{\log(3+|j_2|+2^{j_2})} 2^{-j_2} .
\end{equation}

\end{proof}

\begin{Lemma}\label{lem:rapid3}
There exists a deterministic constant $C>0$ such that, for all $t,s \in (0,1)$ and $j_1 \leq j_2$, the quantities
\[  \sum_{k_1 \in \Z} \sum_{k_2 \in \Z_{j_2}[t,s]} L_{j_1,j_2}^{k_1,k_2} \left|\int_{- \infty}^{\min\{s,t\}} \psi_{H_1}(2^{j_1} x-k_1) \psi_{H_2}(2^{j_2} x-k_2) \, dx \right| \]
\[  \sum_{k_1 \in \Z} \sum_{k_2 \in \Z_{j_2}[t,s]}L_{j_1,j_2}^{k_1,k_2}\left| \int_{\max\{s,t\}}^{+\infty} \psi_{H_1}(2^{j_1} x-k_1) \psi_{H_2}(2^{j_2} x-k_2) \, dx\right| \]
are bounded from above by
\[
C \sqrt{\log(3+|j_1|+2^{j_2})} \sqrt{\log(3+|j_2|+2^{j_2})} 2^{-j_2}.
\]
\end{Lemma}

\begin{proof}
Let us assume that $s \leq t$, the argument for $t < s$ being similar. As $j_2 \geq j_1$, we have, by inequality \eqref{eqn:sumlog},
\begin{align*}
\int_{- \infty}^{s} &\left(\sum_{k_1 \in \Z} \frac{\sqrt{\log(3+|j_1|+|k_1|)}}{(3+|2^{j_1}x-k_1|)^4} \right) \left( \sum_{k_2 \in \Z_{j_2}[t,s]} \frac{\sqrt{\log(3+|j_2|+|k_2|)}}{(3+|2^{j_2}x-k_2|)^4} \right)\, dx \\
& \leq C_L\int_{- \infty}^{s} \sqrt{\log(3+|j_1|+2^{j_1} |x|)}  \left( \sum_{k_2 \in \Z_{j_2}[t,s]} \frac{\sqrt{\log(3+|j_2|+|k_2|)}}{(3+|2^{j_2}x-k_2|)^4} \right)\, dx \\
&\leq C_L\int_{- \infty}^{s} \sqrt{\log(3+|j_1|+2^{j_2} |x|)} \left( \sum_{k_2 \in \Z_{j_2}[t,s]} \frac{\sqrt{\log(3+|j_2|+|k_2|)}}{(3+|2^{j_2}x-k_2|)^4} \right)\, dx.
\end{align*}
For all $k_2 \in \Z_{j_2}[t,s]$, $|k_2| \leq 2^{j_2}$, we have, using \eqref{eqn:inelog},
\begin{align*}
\log&(3+|j_1|+2^{j_2} |x|)\leq  \log(3+|j_1|+2^{j_2}) \log(3+|2^{j_2}x-k_2|) \text{ and} \\
\log&(3+|j_2|+|k_2|) \leq \log(3+j_2+2^{j_2}) \log(3+|2^{j_2}x-k_2|) . 
\end{align*}
Thus, it only remains us to deal with
\begin{align*}
 \int_{- \infty}^{s}  \sum_{k_2 \in \Z_{j_2}[t,s]} \frac{dx}{(3+|2^{j_2}x-k_2|)^{3}}. 
\end{align*}
But, for all $x \leq s$ and $k_2 \in \Z_{j_2}[t,s]$, $|2^{j_2}x-k_2| = k_2-2^{j_2}x$ and then, using the same method as in \eqref{eqn:intavec2}, we get
\begin{align} \label{eqn:boundfor21}
 \int_{- \infty}^{s}  \sum_{k_2 \in \Z_{j_2}[t,s]} \frac{dx}{(3+k_2-2^{j_2}x)^{3}} \leq  C 2^{-j_2} 
\end{align}
which finally leads to
\begin{align}
\nonumber &  \sum_{k_1 \in \Z} \sum_{k_2 \in \Z_{j_2}[t,s]} L_{j_1,j_2}^{k_1,k_2} \left|\int_{- \infty}^{s}  \psi_{H_1}(2^{j_1} x-k_1) \psi_{H_2}(2^{j_2} x-k_2) \, dx \right| \\
& \quad \leq C \sqrt{\log(3+|j_1|+2^{j_2})}\sqrt{\log(3+j_2+2^{j_2})} 2^{-j_2}.\label{eqn:boundfor201}
\end{align}
We get in the same way,
\begin{align*}
&  \sum_{k_1 \in \Z} \sum_{k_2 \in \Z_{j_2}[t,s]} L_{j_1,j_2}^{k_1,k_2} \left|\int_{t}^{+ \infty}  \psi_{H_1}(2^{j_1} x-k_1) \psi_{H_2}(2^{j_2} x-k_2) \, dx \right| 
 \\
& \quad \leq C \sqrt{\log(3+|j_1|+2^{j_2})}\sqrt{\log(3+j_2+2^{j_2})} 2^{-j_2}.
\end{align*}
\end{proof}

Next Lemma will be used to bound the subsums of \eqref{eqn:maindiff} with $j_1<n \leq j_2$ or \linebreak $j_2 <n \leq j_1$.

\begin{Lemma}\label{lem:rapid4}
There exists a deterministic constant $C>0$ such that, for all $t,s \in (0,1)$, the quantities
\[
R^{<\geq n}[t,s] :=\sum_{j_1<n} \sum_{j_2 \geq n} \sum_{(k_1,k_2) \in \Z^2} 2^{j_1(1-H_1)} 2^{j_2(1-H_2)} L_{j_1,j_2}^{k_1,k_2} |I_{j_1,j_2}^{k_1,k_2}[t,s] |
\]
\[ R^{\geq< n}[t,s] :=\sum_{j_1 \geq n} \sum_{j_2 < n} \sum_{(k_1,k_2) \in \Z^2} 2^{j_1(1-H_1)} 2^{j_2(1-H_2)} L_{j_1,j_2}^{k_1,k_2} |I_{j_1,j_2}^{k_1,k_2}[t,s]| \] 
are bounded from above by
\[
C |t-s|^{H_1+H_2-1} \log |t-s|^{-1}. 
\]
\end{Lemma}
\begin{proof}
As $R^{<\geq n}[t,s]$ and $R^{\geq< n}[t,s]$ can clearly be treated symmetrically, we restrict our attention to $R^{<\geq n}[t,s]$. One sees that
\begin{align}
\nonumber &\sum_{j_1<n} \sum_{j_2 \geq n}  \sum_{(k_1,k_2) \in \Z^2} 2^{j_1(1-H_1)} 2^{j_2(1-H_2)} L_{j_1,j_2}^{k_1,k_2} I_{j_1,j_2}^{k_1,k_2}[t,s] \\
\quad &= \sum_{j_1<n} \sum_{j_2 \geq n} \sum_{k_1 \in \Z} \sum_{k_2 \in \Z_{j_2}^<(t,s)}2^{j_1(1-H_1)} 2^{j_2(1-H_2)} L_{j_1,j_2}^{k_1,k_2} I_{j_1,j_2}^{k_1,k_2}[t,s]\label{eqn:lemma3121} \\ 
& + \sum_{j_1<n} \sum_{j_2 \geq n} \sum_{k_1 \in \Z} \sum_{k_2 \in \Z_{j_2}^>(t,s)}2^{j_1(1-H_1)} 2^{j_2(1-H_2)} L_{j_1,j_2}^{k_1,k_2} I_{j_1,j_2}^{k_1,k_2}[t,s] \label{eqn:lemma3122}\\ 
&  + \sum_{j_1<n} \sum_{j_2 \geq n} \sum_{k_1 \in \Z} \sum_{k_2 \in \Z_{j_2}[t,s]} 2^{j_1(1-H_1)} 2^{j_2(1-H_2)} L_{j_1,j_2}^{k_1,k_2} I_{j_1,j_2}^{k_1,k_2}[t,s]. \label{eqn:lemma3123}
\end{align}
For \eqref{eqn:lemma3121}, we use Lemma \ref{lem:rapid2} to get
\begin{align*}
|\eqref{eqn:lemma3121}|& \leq C \sum_{j_1<n} 2^{j_1(1-H_1)} \sqrt{\log(3+|j_1|+2^{j_1})}  \sum_{j_2 \geq n} 2^{-j_2 H_2} \sqrt{\log(3+|j_2|+2^{j_2})}.
\end{align*}
The sum over over $j_1$ is bounded just as in \eqref{eqn:ineserie1} while, for the sum over $j_2$, we have
\begin{align}
\nonumber \sum_{j_2 \geq n} 2^{-j_2 H_2} \sqrt{\log(3+|j_2|+2^{j_2})} & \leq \sum_{j_2 \geq n} 2^{-j_2 H_2} \sqrt{\log(3+2^{j_2 +1})} \\
& \leq C 2^{-n H_2} \sqrt{n}.\label{eqn:ineserie2}
\end{align}
We bound \eqref{eqn:lemma3122} in exactly the same way.

For \eqref{eqn:lemma3123}, let us again assume $s \leq t$, then we write
\begin{align} 
 \nonumber I_{j_1,j_2}^{k_1,k_2}[t,s] &= \int_{\R} \psi_{H_1}(2^{j_1} x-k_1) \psi_{H_2}(2^{j_2} x-k_2) \, dx \\
\nonumber & - \int_{- \infty}^{s} \psi_{H_1}(2^{j_1} x-k_1) \psi_{H_2}(2^{j_2} x-k_2) \, dx \\
&-\int_{t}^{+ \infty} \psi_{H_1}(2^{j_1} x-k_1) \psi_{H_2}(2^{j_2} x-k_2) \, dx.\label{eqn:split}
\end{align}

Since $j_1<n$ and $j_2 \geq n$, recalling Lemma \ref{Lemma:intR}, the sum
\[\sum_{k_1 \in \Z} \sum_{k_2 \in \Z_{j_2}[t,s]}L_{j_1,j_2}^{k_1,k_2} \int_\R \psi_{H_1}(2^{j_1} x-k_1) \psi_{H_2}(2^{j_2} x-k_2) \, dx,\]
vanishes except maybe when $(j_1,j_2)=(n-1,n)$. In this case, note that $\# \Z_{n}[t,s] \leq 2$ and, for all $k_2 \in\Z_{n}[t,s]$, $|k_2| \leq 2^n$.Then, by Lemma \ref{Lemma:intR} and inequality \eqref{eqn:sumlog}, we get
\begin{align*}
& \sum_{k_1 \in \Z}  \sum_{k_2 \in \Z_{n}[t,s]} L_{n-1,n}^{k_1,k_2} \left|I_{n-1,n}^{k_1,k_2} \right| \\
& \quad \leq C 2^{-n}  \sum_{k_1 \in \Z}  \sum_{k_2 \in \Z_{n}[t,s]} \frac{\sqrt{\log(3+n-1+|k_1|)} \sqrt{\log(3+n+|k_2|}}{(3+|2k_1-k_2|)^4} \\
& \quad \leq   C 2^{-n}  \sum_{k_2 \in \Z_{n}[t,s]} \sqrt{\log(3+n-1+|\frac{k_2}{2}|)} \sqrt{\log(3+n+|k_2|)} \\
& \quad \leq C 2^{-n} n
\end{align*}

Now, using Lemma \ref{lem:rapid3}, we also get
\begin{align}
\nonumber& \left| \sum_{j_1<n} \sum_{j_2 \geq n} \sum_{k_1 \in \Z} \sum_{k_2 \in \Z_{j_2}[t,s]} 2^{j_1(1-H_1)} 2^{j_2(1-H_2)} L_{j_1,j_2}^{k_1,k_2} \int_{- \infty}^{s} \psi_{H_1}(2^{j_1} x-k_1) \psi_{H_2}(2^{j_2} x-k_2) \, dx \right| \\
\nonumber \quad & \leq \sum_{j_1<n} \sum_{j_2 \geq n} 2^{j_1(1-H_1)}2^{-j_2 H_2} \sqrt{\log(3+|j_1|+2^{j_2})} \sqrt{\log(3+|j_2|+2^{j_2})} \\
\nonumber & \leq  \sum_{j_1<0}  \sum_{j_2 \geq n} 2^{j_1(1-H_1)}2^{-j_2 H_2}  \sqrt{\log(3+|j_1|)}\log(3+2^{j_2+1})
 \\ \nonumber & \quad  + \sum_{j_1=0}^{n-1}  \sum_{j_2 \geq n} 2^{j_1(1-H_1)}2^{-j_2 H_2} \sqrt{\log(3+2^{j_2+1})} \sqrt{\log(3+2^{j_2+1})} \\
 & \leq C 2^{n(1-H_1-H_2)}n \label{eqn:ineserie3}
\end{align}
The series
\[ \left| \sum_{j_1<n} \sum_{j_2 \geq n} \sum_{k_1 \in \Z} \sum_{k_2 \in \Z_{j_2}[t,s]} 2^{j_1(1-H_1)} 2^{j_2(1-H_2)} L_{j_1,j_2}^{k_1,k_2} \int_{- \infty}^{s} \psi_{H_1}(2^{j_1} x-k_1) \psi_{H_2}(2^{j_2} x-k_2) \, dx \right|\]
is bounded in exactly the same way and the conclusion follows.
\end{proof}

It remains us to bound the subsums of \eqref{eqn:maindiff} with $j_1 \geq n$ and $j_2 \geq n$. For this, let us define some random variables associated with dyadic intervals.
\begin{Def}
If $\lambda$ is a dyadic interval of scale $n$, we define, for all $j \geq n$, the indexation sets
\begin{align*}
 S_j^0(\lambda) &:= \{ (k^{(1)},K^{(1)},k^{(2)},K^{(2)}) \in \Z^4 \, : \, \frac{k^{(1)}}{2^{j}},\frac{K^{(1)}}{2^{j}},\frac{k^{(2)}}{2^{j}}, \frac{K^{(2)}}{2^j} \in \lambda \}, \\
S_j^1(\lambda)&: = \{ (k^{(1)},K^{(1)},k^{(2)},K^{(2)}) \in \Z^4 \, : \, \frac{k^{(1)}}{2^{j+1}},\frac{K^{(1)}}{2^{j+1}},\frac{k^{(2)}}{2^{j}}, \frac{K^{(2)}}{2^j} \in \lambda \}, \\
 S_j^2(\lambda) &:= \{ (k^{(1)},K^{(1)},k^{(2)},K^{(2)}) \in \Z^4 \, : \, \frac{k^{(1)}}{2^{j}},\frac{K^{(1)}}{2^{j}},\frac{k^{(2)}}{2^{j+1}}, \frac{K^{(2)}}{2^{j+1}} \in \lambda \}
\end{align*}
and consider the random variables, for $(k^{(1)},K^{(1)},k^{(2)},K^{(2)}) \in S_j^0(\lambda)$,
\begin{equation}\label{eqn:rvint2}
\sideset{_j^0}{_{k^{(1)},K^{(1)}} ^{k^{(2)},K^{(2)}}}\sum := \sum_{k^{(1)} \leq k_1 \leq K^{(1)}} \sum_{k^{(2)} \leq k_2 \leq K^{(2)}} \varepsilon_{j,j}^{k_1,k_2} I_{j,j}^{k_1,k_2}
\end{equation}, for $(k^{(1)},K^{(1)},k^{(2)},K^{(2)}) \in S_j^1(\lambda)$,
\begin{equation}\label{eqn:rvint1}
\sideset{_j^1}{_{k^{(1)},K^{(1)}} ^{k^{(2)},K^{(2)}}}\sum := \sum_{k^{(1)} \leq k_1 \leq K^{(1)}} \sum_{k^{(2)} \leq k_2 \leq K^{(2)}} \varepsilon_{j+1,j}^{k_1,k_2} I_{j+1,j}^{k_1,k_2}
\end{equation}
and, for $(k^{(1)},K^{(1)},k^{(2)},K^{(2)}) \in S_j^2(\lambda)$,
\begin{equation}\label{eqn:rvint3}
\sideset{_j^2}{_{k^{(1)},K^{(1)}} ^{k^{(2)},K^{(2)}}}\sum := \sum_{k^{(1)} \leq k_1 \leq K^{(1)}} \sum_{k^{(2)} \leq k_2 \leq K^{(2)}} \varepsilon_{j,j+1}^{k_1,k_2} I_{j,j+1}^{k_1,k_2}. 
\end{equation}
\end{Def}
The idea behind the definition of these random variables is, as $|t-s| \leq 2^{-n}$, $s \in 3\lambda_n(t)$ and thus any sum of the form
\begin{equation}\label{eqn:explirv}
\sum_{k_1 \in \Z_j [t,s]} \sum_{k_2 \in \Z_{\ell}[t,s]} \varepsilon_{j,\ell}^{k_1,k_2} I_{j,\ell}^{k_1,k_2}
\end{equation}
for $\ell \in \{j,j+1\}$ can be written as the sum of random variables \eqref{eqn:rvint2}, \eqref{eqn:rvint1} or \eqref{eqn:rvint3} for some $(k^{(1)},K^{(1)},k^{(2)},K^{(2)})$ belonging to at most two $S_j^\ell(\lambda)$ $(\ell \in \{0,1,2 \}$) with $\lambda \in \lambda_n(t)$. Indeed,
\begin{itemize}
\item if $t$ and $s$ both belong to $\lambda_n(t)$ then we only need to rewrite \eqref{eqn:explirv} in the form \eqref{eqn:rvint2}, \eqref{eqn:rvint1} or \eqref{eqn:rvint3} for $(k^{(1)},K^{(1)},k^{(2)},K^{(2)}) \in S_j^\ell(\lambda_n(t))$;
\item if $s \in \lambda$ with $\lambda \in 3 \lambda_n(t) \setminus \lambda_n(t)$ then we need to consider a first sum indexed by a quadruple of $S_j^\ell(\lambda_n(t))$ and a second indexed by a quadruple of $S_j^\ell(\lambda)$.
\end{itemize}

The reason why we decide to put $\lambda$ instead of $3 \lambda$ in the definition of the sets $S_j^\ell(\lambda)$ is that if, for all $n \in \N$ and for all $\lambda \in \Lambda_n$ and $j \geq n$, we define the random variable
\begin{equation}
\Xi_j (\lambda) =  \max_{\ell \in \{0,1,2\}} \sup_{(k^{(1)},K^{(1)},k^{(2)},K^{(2)}) \in S_j^\ell(\lambda))} \frac{\left| \sideset{_j^\ell}{_{k^{(1)},K^{(1)}} ^{k^{(2)},K^{(2)}}}\sum \right|}{\left \| \sideset{_j^\ell}{_{k^{(1)},K^{(1)}} ^{k^{(2)},K^{(2)}}}\sum \right \|_{L^2(\Omega)}},
\end{equation}
we want $\Xi_j (\lambda)$ to be independent of $\Xi_j (\lambda')$ as soon as $\lambda \cap \lambda' = \cmemptyset$. Moreover, from the definitions of the random variables \eqref{eqn:rvint1}, \eqref{eqn:rvint2} and \eqref{eqn:rvint3}, the remarks below Theorem \ref{thm:decompo} and the explicit expressions \eqref{eqn:intR1}, \eqref{eqn:intR2} and \eqref{eqn:intR3}, the law of $\Xi_j (\lambda)$ does not depend on $\lambda \in \Lambda_n$but only on $j-n$.

The key results to estimate the random variables $\Xi_j$ are \cite[Theorem 6.7 and Theorem 6.12]{janson97} that we recall here.

\begin{Thm}\label{thm:bound}
There exists a strictly positive universal deterministic constant $\overset{\star}{C}$ such that, for every random variable $X$ belonging to the second order Wiener chaos and for each real number $y \geq 2$, one has
\[ \mathbb{P}(|X| \geq y \|X\|_{L^2(\Omega)}) \leq \exp(-\overset{\star}{C} y).\]
\end{Thm}

\begin{Thm}\label{thm:bound2}
If $X$ is a random variable belonging to the second order Wiener chaos, there exist $a,b,y_0>0$ such that, for all $y\geq y_0$,
\[ \exp(-a y) \leq \mathbb{P}(|X| \geq y ) \leq \exp(-b y). \] 
\end{Thm}

\begin{Rmk}
As stated in \cite{janson97}, the constants $a,b$ in Theorem \ref{thm:bound2} are not universal and depend on the law of $X$. Note that $b$ can be recovered from Theorem \ref{thm:bound} and thus is universal on the unit sphere in $L^2(\Omega)$.
\end{Rmk}

\begin{Lemma}
There exists a deterministic constant $C>0$ such that, for all $n \in \N$, $\lambda \in \Lambda_n$, $j \geq n$, $\ell \in \{0,1,2\}$ and $(k^{(1)},K^{(1)},k^{(2)},K^{(2)}) \in S_j^\ell(\lambda)$, we have
\[ \left \| \sideset{_j^\ell}{_{k^{(1)},K^{(1)}} ^{k^{(2)},K^{(2)}}}\sum \right \|_{L^2(\Omega)} \leq C 2^{\frac{-j-n}{2}}.\]
\end{Lemma}

\begin{proof}
Following an idea from \cite[Lemma 2.21]{ayacheesmili}, we write
\[ \left \| \sideset{_j^\ell}{_{k^{(1)},K^{(1)}} ^{k^{(2)},K^{(2)}}}\sum \right \|_{L^2(\Omega)} \leq \sum_{\mathcal{R} \in \{ <,>,=\} }\left \| \sideset{_j^\ell}{_{k^{(1)},K^{(1)}} ^{k^{(2)},K^{(2)}}}\sum_\mathcal{R} \right \|_{L^2(\Omega)} \]
where 
\[ \sideset{_j^\ell}{_{k^{(1)},K^{(1)}} ^{k^{(2)},K^{(2)}}}\sum_\mathcal{R}\]
is the subsum of \eqref{eqn:rvint1}, \eqref{eqn:rvint2} or \eqref{eqn:rvint3} in which $k_1 \mathcal{R}k_2$. By doing so, we make sure that two random variables $\varepsilon_{j_1,j_2}^{k_1,k_2}$ and $\varepsilon_{j_1',j_2'}^{k_1',k_2'}$ appearing in this subsum are uncorrelated except when $(k_1,k_2)=(k_1',k_2')$. Then from Lemma \ref{Lemma:intR}, we have for $\ell=0$ (the argument being the same for $\ell=1$ or $\ell=2$), for all $\mathcal{R} \in \{<,>,=\}$,
\begin{align*}
\left \| \sideset{_j^\ell}{_{k^{(1)},K^{(1)}} ^{k^{(2)},K^{(2)}}}\sum_\mathcal{R} \right \|_{L^2(\Omega)}^2 & = \sum_{k^{(1)} \leq k_1 \leq K^{(1)}} \sum_{k^{(2)} \leq k_2 \leq K^{(2)}, k_1 \mathcal{R}k_2} \mathbb{E}[(\varepsilon_{j,j}^{k_1,k_2})^2] (I_{j,j}^{k_1,k_2} )^2 \\
& \leq  \sum_{k^{(1)} \leq k_1 \leq K^{(1)}} \sum_{k_2 \in \Z} \frac{2^{-2j}}{(3+|k_1-k_2|)^{8}}. 
\end{align*}
Since $\# \{k_1 \in \Z \, : \, k^{(1)} \leq k_1 \leq K^{(1)} \} \leq 2^{j-n}$, we conclude that
\begin{equation}\label{eqn:boundl2}
\left \| \sideset{_j^\ell}{_{k^{(1)},K^{(1)}} ^{k^{(2)},K^{(2)}}}\sum \right \|_{L^2(\Omega)} \leq C 2^{\frac{-j-n}{2}}.
\end{equation}
\end{proof}

\begin{Lemma}\label{lemma:borelcantellitotal}
There exist an event $\widetilde{\Omega}$ of probability $1$ and a positive random variable $C_2$ with finite moment of any order such that, on $\widetilde{\Omega}$
\begin{equation}\label{eqn:borelcantellitotal}
\forall n \in \N, ~ \forall \, \lambda \subseteq  [0,1],~\lambda \in \Lambda_n,~ \forall j \geq n, ~\Xi_j(\lambda) \leq C_2 \, (j-n+1)n.
\end{equation}
\end{Lemma}

\begin{proof}
Let us take $\theta > 0$ and consider, for all $n \in \N$ the event
\[ A_n := \left \{ \forall \, \lambda \subseteq  [0,1], \lambda \in \Lambda_n, \, \forall j \geq n, \, \Xi_j(\lambda) \leq \theta (j-n+1)n \right \}.\]
If $A_n^c$ stands for the complementary set of $A_n$ in $\Omega$, we have, of course,
\begin{align*}
\mathbb{P}(A_n^c)= \mathbb{P}(\exists \lambda \subseteq  [0,1], \lambda \in \Lambda_n \, : \, \exists j \geq n \text{ s. t. }  \Xi_j(\lambda) \geq \theta (j-n+1)n).
\end{align*}
But, for all $\lambda \subseteq  [0,1], \lambda \in \Lambda_n$, $j \geq n$, $\ell \in \{0,1,2\}$ and $(k^{(1)},K^{(1)},k^{(2)},K^{(2)}) \in S_j^\ell(\lambda)$ we have, by Theorem \ref{thm:bound}, if $\theta \geq 2$,
\begin{align*}
\mathbb{P} \left(\frac{\left| \sideset{_j^\ell}{_{k^{(1)},K^{(1)}} ^{k^{(2)},K^{(2)}}}\sum \right|}{\left \| \sideset{_j^\ell}{_{k^{(1)},K^{(1)}} ^{k^{(2)},K^{(2)}}}\sum \right \|_{L^2(\Omega)}} \geq \theta (j-n+1)n \right)\leq \exp(-\overset{\star}{C} \theta (j-n+1)n).
\end{align*}
As, for all $j \geq n$, $\#S_j^\ell(\lambda) \leq 2^{4(j-n)}$ and $\# \{\lambda \subseteq  [0,1] \, : \,  \lambda \in \Lambda_n \}=2^n$, we get
\begin{align*}
\mathbb{P}(A_n^c) &\leq C 2^n \sum_{j \geq n} 2^{4(j-n)} \exp(-\overset{\star}{C} \theta (j-n+1)n) \\
& \leq C 2^{n} \exp(-\overset{\star}{C} \theta n) \sum_{j \geq n} 2^{4(j-n)} \exp(-\overset{\star}{C} \theta (j-n))
\end{align*}
for a deterministic constant $C>0$. Therefore, if we take $\theta > 4 \log(2) / \overset{\star}{C}$, the conclusion follows from Borel-Cantelli Lemma.  
\end{proof}

\begin{Lemma}\label{lem:rapid5}
Let $\Omega^*$ and $\widetilde{\Omega}$ be the events of probability $1$ given by Lemmata \ref{Lemma:ayacheesmili} and \ref{lemma:borelcantellitotal} respectively. There exists a positive random variable  $C_3$ with finite moment of any order such that, on $\Omega^* \cap \widetilde{\Omega}$, for all $t,s \in (0,1)$ the random variable
\begin{equation}\label{eqn:lastrv}
\left| \sum_{j_1 \geq n} \sum_{j_2 > n} \sum_{(k_1,k_2) \in \Z^2} 2^{j_1(1-H_1)} 2^{j_2(1-H_2)} \varepsilon_{j_1,j_2}^{k_1,k_2} I_{j_1,j_2}^{k_1,k_2}[t,s] \right|\\
\end{equation}
is bounded from above by
\[
C_{3} |t-s|^{H_1+H_2-1} \log |t-s|^{-1}. 
\]
\end{Lemma}

\begin{proof}
We start by splitting the sums in \eqref{eqn:lastrv} in two parts:
\begin{align}\label{somme43}
\nonumber & \sum_{j_1 \geq n} \sum_{j_2 \geq j_1} \sum_{(k_1,k_2) \in \Z^2} 2^{j_1(1-H_1)} 2^{j_2(1-H_2)} \varepsilon_{j_1,j_2}^{k_1,k_2} I_{j_1,j_2}^{k_1,k_2}[t,s]\text{ and}   \\
&  \sum_{j_2 \geq n} \sum_{j_1 > j_2} \sum_{(k_1,k_2) \in \Z^2} 2^{j_1(1-H_1)} 2^{j_2(1-H_2)} \varepsilon_{j_1,j_2}^{k_1,k_2} I_{j_1,j_2}^{k_1,k_2}[t,s].
\end{align}
We only focus on the first sums, as the argument is symmetric in $j_1$ and $j_2$. As in Lemma \ref{lem:rapid4} we write
\begin{align}
\nonumber &\sum_{j_1 \geq n} \sum_{j_2 \geq j_1}  \sum_{(k_1,k_2) \in \Z^2} 2^{j_1(1-H_1)} 2^{j_2(1-H_2)} \varepsilon_{j_1,j_2}^{k_1,k_2} I_{j_1,j_2}^{k_1,k_2}[t,s] \\
\quad &= \sum_{j_1 \geq n} \sum_{j_2 \geq j_1} \sum_{k_1 \in \Z} \sum_{k_2 \in \Z_{j_2}^<(t,s)}2^{j_1(1-H_1)} 2^{j_2(1-H_2)} \varepsilon_{j_1,j_2}^{k_1,k_2} I_{j_1,j_2}^{k_1,k_2}[t,s]\label{eqn:lemma3191} \\ 
& + \sum_{j_1 \geq n} \sum_{j_2 \geq j_1} \sum_{k_1 \in \Z} \sum_{k_2 \in \Z_{j_2}^>(t,s)}2^{j_1(1-H_1)} 2^{j_2(1-H_2)} \varepsilon_{j_1,j_2}^{k_1,k_2} I_{j_1,j_2}^{k_1,k_2}[t,s] \label{eqn:lemma3192}\\ 
&  + \sum_{j_1 \geq n} \sum_{j_2 \geq j_1} \sum_{k_1 \in \Z} \sum_{k_2 \in \Z_{j_2}[t,s]} 2^{j_1(1-H_1)} 2^{j_2(1-H_2)} \varepsilon_{j_1,j_2}^{k_1,k_2} I_{j_1,j_2}^{k_1,k_2}[t,s]. \label{eqn:lemma3193}
\end{align}

To bound \eqref{eqn:lemma3191}, we use inequality \eqref{eqn:moustique} and Lemma \ref{lem:rapid2} to get
\begin{align*}
 | \eqref{eqn:lemma3191}| &\leq C C_1 \sum_{j_1 \geq n} \sum_{j_2 \geq j_1} 2^{j_1(1-H_1)} 2^{- j_2 H_2 }\sqrt{\log(3+|j_1|+2^{j_1})} \sqrt{\log(3+|j_2|+2^{j_2})} \\
& \leq C C_1 \sum_{j_1 \geq n} 2^{j_1(1-H_1-H_2)} \sqrt{j_1} \\
& \leq C C_1 2^{n(1-H_1-H_2)} n,
\end{align*}
by applying twice inequality \eqref{eqn:ineserie2}. The sum \eqref{eqn:lemma3192} is bounded in exactly the same way.

To bound \eqref{eqn:lemma3193}, we use once again the equality \eqref{eqn:split}. First we have, by inequality \eqref{eqn:moustique} and Lemma \ref{lem:rapid3},
\begin{align}
\nonumber &\left| \sum_{j_1 \geq n} \sum_{j_2 \geq j_1} \sum_{k_1 \in \Z} \sum_{k_2 \in \Z_{j_2}[t,s]} 2^{j_1(1-H_1)} 2^{j_2(1-H_2)} \varepsilon_{j_1,j_2}^{k_1,k_2} \int_{- \infty}^{s} \psi_{H_1}(2^{j_1} x-k_1) \psi_{H_2}(2^{j_2} x-k_2) \, dx \right| \\
\nonumber \quad & \leq C C_1 \sum_{j_1 \geq n} \sum_{j_2 \geq j_1} 2^{j_1(1-H_1)} 2^{- j_2 H_2 } \sqrt{\log(3+|j_1|+2^{j_2})} \sqrt{\log(3+|j_2|+2^{j_2})} \\
& \leq C C_1 2^{n(1-H_1-H_2)} n. \label{eqn:ineserie4}
\end{align}
We bound
\begin{align*}
\left| \sum_{j_1 \geq n} \sum_{j_2 \geq j_1} \sum_{k_1 \in \Z} \sum_{k_2 \in \Z_{j_2}[t,s]} 2^{j_1(1-H_1)} 2^{j_2(1-H_2)} \varepsilon_{j_1,j_2}^{k_1,k_2} \int_{t}^{+ \infty} \psi_{H_1}(2^{j_1} x-k_1) \psi_{H_2}(2^{j_2} x-k_2) \, dx \right| 
\end{align*}
in the same way.

It only remains us to find an estimate for 
\[ \sum_{j_1 \geq n} \sum_{j_2 \geq j_1} \sum_{k_1 \in \Z} \sum_{k_2 \in \Z_{j_2}[t,s]} 2^{j_1(1-H_1)} 2^{j_2(1-H_2)} \varepsilon_{j_1,j_2}^{k_1,k_2} I_{j_1,j_2}^{k_1,k_2}\]
and thus, recalling Lemma \ref{Lemma:intR}, we reduce the problem to first bound, for $j\geq n$ and $ \ell \in \{j,j+1\}$, the sums
\begin{equation}\label{eqn:r>>6}
\sum_{k_1 \in \Z_j^<(t,s)} \sum_{k_2 \in \Z_{\ell}[t,s]} \varepsilon_{j,\ell}^{k_1,k_2} I_{j,\ell}^{k_1,k_2},
\end{equation}
\begin{equation}\label{eqn:r>>5}
\sum_{k_1 \in \Z_j^>(t,s)} \sum_{k_2 \in \Z_{\ell}[t,s]} \varepsilon_{j,\ell}^{k_1,k_2} I_{j,\ell}^{k_1,k_2},
\end{equation}
\begin{equation} \label{eqn:r>>3}
\sum_{k_1 \in \Z_j [t,s]} \sum_{k_2 \in \Z_{\ell}[t,s]} \varepsilon_{j,\ell}^{k_1,k_2} I_{j,\ell}^{k_1,k_2}
\end{equation}
on $\Omega^* \cap \widetilde{\Omega}$. Let us consider \eqref{eqn:r>>6} with $\ell=j$, the argument for $\ell=j+1$ and \eqref{eqn:r>>5} being similar. Using again Lemmata \ref{Lemma:intR} and \ref{Lemma:ayacheesmili}, we have on $\Omega^* \cap \widetilde{\Omega}$, since for all $k_2 \in \Z_{j}[t,s]$, $|k_2| \leq 2^{j}$, for $j\geq n$,
\begin{align}
\nonumber|\eqref{eqn:r>>6}| &\leq C C_1\sum_{k_1 \in \Z_j^<(t,s)} \sum_{k_2 \in \Z_{j}[t,s]} \frac{2^{-j}}{(3+|k_1-k_2|)^4} \sqrt{\log(3+j+|k_1|)} \sqrt{\log(3+j+|k_2|)} \\
\nonumber & \leq CC_1\sum_{k_1 \in \Z_j^<(t,s)} \sum_{k_2 \in \Z_{j}[t,s]} \frac{2^{-j} \sqrt{j}}{(3+k_2-k_1)^4} \sqrt{\log(3+j+|k_1|)} \\
\nonumber & \leq C C_1\sum_{k_1 \in \Z_j^<(t,s)} \sum_{m=0}^{+ \infty} \frac{2^{-j} \sqrt{j}}{(3+2^j \min\{s,t\}+m-k_1)^4} \sqrt{\log(3+j+|k_1|)} \\
\nonumber &\leq C C_1\sum_{k_1 \in \Z_j^<(t,s)} 2^{-j} \sqrt{j} \int_{0}^{+ \infty} \frac{dy}{(2+2^j \min\{s,t\}+y-k_1)^4} \sqrt{\log(3+j+|k_1|)} \\
\nonumber & \leq C C_1 2^{-j} \sqrt{j} \sum_{k_1 \in \Z_j^<(t,s)} \frac{\sqrt{\log(3+j+|k_1|)}}{(2+2^j \min\{s,t\}-k_1)^{3}} \\
\nonumber & \leq C C_1 2^{-j} \sqrt{j} \sqrt{\log(3+j+2^j \min\{s,t\})} \\
& \leq C C_1 2^{-j} j.\label{eqn:boundfor26}
\end{align}
It follows that
\[ \left |  \sum_{j \geq n} 2^{j(2-H_1-H_2)}\sum_{\ell=j}^{j+1} \sum_{k_1 \in\Z_j^<(t,s)} \sum_{k_2 \in \Z_{\ell}[t,s]} \varepsilon_{j,\ell}^{k_1,k_2} I_{j,\ell}^{k_1,k_2}\right| \leq C C_1 2^{n(1-H_1-H_2)} n\]
and, similarly,
\[ \left |  \sum_{j \geq n} 2^{j(2-H_1-H_2)}\sum_{\ell=j}^{j+1} \sum_{k_1 \in\Z_j^>(t,s)} \sum_{k_2 \in \Z_{\ell}[t,s]} \varepsilon_{j,\ell}^{k_1,k_2} I_{j,\ell}^{k_1,k_2}\right| \leq C C_1 2^{n(1-H_1-H_2)} n.\]
The bound for \eqref{eqn:r>>3} is obtained using \eqref{eqn:borelcantellitotal} and \eqref{eqn:boundl2} which lead to
\begin{align*}
&\left |  \sum_{j \geq n} 2^{j(2-H_1-H_2)}\sum_{\ell=j}^{j+1} \sum_{k_1 \in \Z_j [t,s]} \sum_{k_2 \in \Z_{\ell}[t,s]} \varepsilon_{j,\ell}^{k_1,k_2} I_{j,\ell}^{k_1,k_2}\right| \\ \quad & \leq C C_2 \sum_{j \geq n} 2^{j(\frac{3}{2}-H_1-H_2)} 2^{-\frac{n}{2}} (j-n+1)n \\
& \leq C C_2 2^{n(\frac{3}{2}-H_1-H_2)} 2^{-\frac{n}{2}} n \\
& = C C_2 2^{n(1-H_1-H_2)} n,
\end{align*}
as $\frac{3}{2}<H_1+H_2$.

Putting all of these together we get that \eqref{eqn:lastrv} is bounded from above by
\[  C \max\{C_1,C_2\} |t-s|^{H_1+H_2-1} \log |t-s|^{-1}\]
on $\Omega^* \cap \widetilde{\Omega}$.
\end{proof}

We now prove the main result of this subsection.

\begin{proof}[Proof of Proposition \ref{prop:rapid}]
Let us consider $\omega$ in the event $\Omega^* \cap \widetilde{\Omega}$ of probability $1$, where $\Omega^*$ and $\widetilde{\Omega}$ are given by Lemmata \ref{Lemma:ayacheesmili} and \ref{lemma:borelcantellitotal} respectively.

If $t,s \in (0,1)$, we write
\begin{align*}
\begin{split}
|R_{H_1,H_2}&(t,\omega)-R_{H_1,H_2}(s,\omega)|\\
 & \leq \left| \sum_{j_1 < n} \sum_{j_2 < n} \sum_{(k_1,k_2) \in \Z^2} 2^{j_1(1-H_1)} 2^{j_2(1-H_2)}  \varepsilon_{j_1,j_2}^{k_1,k_2}(\omega) I_{j_1,j_2}^{k_1,k_2}[t,s]\right| \\
\quad & + \left| \sum_{j_1 < n} \sum_{j_2 \geq  n} \sum_{(k_1,k_2) \in \Z^2} 2^{j_1(1-H_1)} 2^{j_2(1-H_2)} \varepsilon_{j_1,j_2}^{k_1,k_2}(\omega)  I_{j_1,j_2}^{k_1,k_2}[t,s]\right| \\
& + \left| \sum_{j_1 \geq n} \sum_{j_2 < n} \sum_{(k_1,k_2) \in \Z^2} 2^{j_1(1-H_1)} 2^{j_2(1-H_2)}  \varepsilon_{j_1,j_2}^{k_1,k_2}(\omega)  I_{j_1,j_2}^{k_1,k_2}[t,s]\right| \\
& + \left| \sum_{j_1 \geq n} \sum_{j_2 \geq n} \sum_{(k_1,k_2) \in \Z^2} 2^{j_1(1-H_1)} 2^{j_2(1-H_2)}  \varepsilon_{j_1,j_2}^{k_1,k_2}(\omega)  I_{j_1,j_2}^{k_1,k_2}[t,s]\right|.
\end{split}
\end{align*}
The first sum is bounded from above by Lemmata \ref{Lemma:ayacheesmili} and \ref{lem:rapid1}, the second and the third one are bounded from above by Lemmata \ref{Lemma:ayacheesmili} and \ref{lem:rapid4} and the last one is bounded from above by Lemma \ref{lem:rapid5}.
\end{proof}

\begin{Rmk}\label{rmk:method}
Starting from now and until the end of this section, one can reduce our attention to the process
\[ \left \{R_{H_1,H_2}'(t)= \sum_{j_1=0}^{+ \infty} \sum_{j_2=0}^{+ \infty} \sum_{(k_1,k_2) \in \Z^2} 2^{j_1(1-H_1)} 2^{j_2(1-H_2)} \varepsilon_{j_1,j_2}^{k_1,k_2}  I_{j_1,j_2}^{k_1,k_2}[0,\cdot] \right \} \]
because almost surely, it is the most irregular part of $R_{H_1,H_2}$. Indeed, using different estimates obtained in this subsection, one can see that, almost surely, there exists a constant $C>0$ such that, for all $s,t \in (0,1)$,
\begin{align*}
\left| \sum_{j_1<0} \sum_{j_2<0} \sum_{(k_1,k_2) \in \Z^2} 2^{j_1(1-H_1)} 2^{j_2(1-H_2)}  \varepsilon_{j_1,j_2}^{k_1,k_2}  I_{j_1,j_2}^{k_1,k_2}[t,s]\right| &\leq C |t-s|, \\
\left| \sum_{j_1<0} \sum_{j_2>0} \sum_{(k_1,k_2) \in \Z^2} 2^{j_1(1-H_1)} 2^{j_2(1-H_2)}  \varepsilon_{j_1,j_2}^{k_1,k_2}  I_{j_1,j_2}^{k_1,k_2}[t,s] \right| &\leq C |t-s|^{-H_2} \log|t-s|^{-1} \\
\left| \sum_{j_1>0} \sum_{j_2<0} \sum_{(k_1,k_2) \in \Z^2} 2^{j_1(1-H_1)} 2^{j_2(1-H_2)}  \varepsilon_{j_1,j_2}^{k_1,k_2}  I_{j_1,j_2}^{k_1,k_2}[t,s] \right| &\leq C |t-s|^{-H_1} \log|t-s|^{-1}
\end{align*}
and we conclude because $H_1+H_2-1<\min \{H_1,H_2\}<1$.
\end{Rmk}

\subsection{Ordinary points}

Let us now go to the almost sure finiteness of the limit \eqref{eqn:mainthm2} for almost every point. The main idea behind our method is that wavelets which contribute the most in $|R_{H_1,H_2}(t,\cdot)-R_{H_1,H_2}(s,\cdot)|$ are the ones with associated dyadic intervals ``close'' to the interval $[t,s]$. Thus, we aim at proving the following Proposition.

\begin{Prop}\label{Prop:ordinary}
There exists an event $\Omega_{\text{ord}}$ of probability $1$ such that for all $\omega \in \Omega_{\text{ord}}$, for almost every $t \in (0,1)$,
\[ \limsup_{s \to t} \dfrac{|R_{H_1,H_2}(t,\omega)-R_{H_1,H_2}(s,\omega) |}{|t-s|^{H_1+H_2-1}  \log\log |t-s|^{-1}} < + \infty.\]
\end{Prop}

As in \cite{esserloosveldt}, for all $j \in \N$, we denote by $k_j(t)$ the unique integer such that $t \in [k_j(t)2^{-j},(k_j(t)+1)2^{-j})$. In other words, $k_j(t)=s(\lambda_j(t))$. If $t \in (0,1)$ is fixed, applying Lemma \ref{Lemma:ayacheesmili} to the sequence of random variables $(\xi_{j_1,j_2}^{k_1',k_2'})_{(j_1,j_2,k_1',k_2') \in \Z^4}$ defined by
\[\xi_{j_1,j_2}^{k_1',k_2'}=\varepsilon_{j_1,j_2}^{k_1'+k_{j_1(t)},k_2'+k_{j_2(t)}} \]
we deduce the existence of $\Omega_t^*$, an event of probability $1$, and $C_{t,1}$, a positive random variable with finite moment of any order, such that, for all $\omega \in \Omega_t^*$ and for each $(j_1,j_2,k_1,k_2) \in \Z^4$, one has
\begin{equation}\label{eqn:ordinary}
|\varepsilon_{j_1,j_2}^{k_1,k_2}(\omega)| \leq C_{t,1}(\omega) \sqrt{\log(3+|j_1|+|k_1-k_{j_1}(t)|)} \sqrt{\log(3+|j_2|+|k_2-k_{j_2}(t)|)}.
\end{equation}

In view of this fact, let us set, for $t \in (0,1)$ and $(j_1,j_2,k_1,k_2) \in \N^2 \times \Z^2$
\[L_{j_1,j_2}^{k_1,k_2}(t) =  \sqrt{\log(3+j_1+|k_1-k_{j_1}(t)|)} \sqrt{\log(3+j_2+|k_2-k_{j_2}(t)|)}.\]
In what follows, we show how to modify Lemmata \ref{lem:rapid1} to \ref{lem:rapid5} from the previous subsection, using $L_{j_1,j_2}^{k_1,k_2}(t)$ instead of $L_{j_1,j_2}^{k_1,k_2}$. Before all, we need the following Lemma  which is inspired by results from \cite{esserloosveldt} that can be extended in our case.

\begin{Lemma}\label{Lemma:ordinary}
For all $L>2$ there exists a constant $C_L>0$ such that, for all $n \in \N$ and $t,s \in (0,1)$ such that $2^{-n-1} < |t-s| \leq 2^{-n}$, for all $x \in [s,t]$ 
\begin{enumerate}
\item For all $0\leq j < n $
\[  \sum_{k \in \Z} \frac{\sqrt{\log(3+j+|k-k_j(t)|}}{(3+|2^{j}x-k|)^{L}} \leq C_L \sqrt{\log(3+j)}.\]
\item For all $j \geq n $
\[
\sum_{k \in \Z} \frac{\sqrt{\log(3+j+|k-k_j(t)|}}{(3+|2^{j}x-k|)^{L}} \leq C_L \sqrt{j-n+1}\sqrt{\log(3+j)}.
\]
\end{enumerate}
\end{Lemma}

\begin{proof}
For all $j \in \N$, $k \in \Z$ and $x \in [s,t]$, observe that
\begin{equation}\label{eqn:Lemma:ordinary}
|k-k_j(t)| \leq |k-2^jx|+|2^jx -2^j t| + |2^jt-k_j(t)| \leq |k-2^jx|+2^{j-n}+1.
\end{equation}

If $0\leq j < n $, then it follows from \eqref{eqn:Lemma:ordinary} that $|k-k_j(t)| \leq |2^jx-k|+2$ which allow us to write, thanks to inequality \eqref{eqn:inelog},
\begin{align*}
\frac{\sqrt{\log(3+j+|k-k_j(t)|}}{(3+|2^{j}x-k|)} &\leq \sqrt{\log(3+j)} \frac{\sqrt{\log(5+|2^jx-k|)}}{|2^jx-k|+3} \\
&\leq C \sqrt{\log(3+j)}.
\end{align*}
where $C := \sup_{x \geq 0} \left(\frac{\sqrt{\log(5+x)}}{x+3} \right)$ and we conclude using the boundedness of the function
\begin{equation}\label{eqn:fctbounded}
\xi \mapsto \sum_{k \in \Z} \frac{1}{(1+|\xi-k|)^{M}}
\end{equation}
for all $M>1$.

Now, if $j \geq n $, from \eqref{eqn:Lemma:ordinary} we get $|k-k_j(t)| \leq |2^jx-k|+2^{j-n+1}$ and thus, again by inequality \eqref{eqn:inelog},
\begin{align*}
\frac{\sqrt{\log(3+j+|k-k_j(t)|}}{(3+|2^{j}x-k|)} &\leq  \sqrt{\log(3+2^{j-n+1})} \sqrt{\log(3+j)} \frac{\sqrt{\log(3+|2^jx-k|)}}{|2^jx-k|+3} \\
&\leq C'  \sqrt{j-n+1}\sqrt{\log(3+j)}.
\end{align*}
where $C' := \sqrt{3} \sup_{x \geq 0} \left(\frac{\sqrt{\log(3+x)}}{x+3} \right)$ and the conclusion comes again from the boundedness of the function in \eqref{eqn:fctbounded} for all $M>1$
\end{proof}

\begin{Lemma}\label{lem:ordinary1}
There exists a deterministic constant $C>0$ such that, for all $t,s \in (0,1)$ we have
\begin{align*}
\sum_{0 \leq j_1<n} &\sum_{0 \leq j_2<n} \sum_{(k_1,k_2) \in \Z^2} 2^{j_1(1-H_1)} 2^{j_2(1-H_2)}  L_{j_1,j_2}^{k_1,k_2}(t) \left| I_{j_1,j_2}^{k_1,k_2}[t,s] \right| \\ & \leq  C |t-s|^{H_1+H_2-1} \log \log |t-s|^{-1} .
\end{align*}
\end{Lemma}

\begin{proof}
If $\xi \in [s,t]$, we get from the fast decay of the fractional antiderivatives of $\psi$ \eqref{eqn:fastdecay} and inequality \eqref{eqn:ordinary}, for $0\leq j_1,j_2<n$,
\begin{align*}
\sum_{(k_1,k_2) \in \Z^2} & L_{j_1,j_2}^{k_1,k_2}(t)   |\psi_{H_1}(2^{j_1} \xi-k_1) \psi_{H_2}(2^{j_2} \xi-k_2)| \\ & \leq C C_1  \left(\sum_{k_1 \in \Z} \frac{\sqrt{\log(3+j_1+|k_1-k_{j_1}(t)|)}}{(3+|2^{j_1}\xi-k_1|)^4} \right) \\ & \qquad \qquad \left( \sum_{k_2 \in \Z} \frac{\sqrt{\log(3+j_2+|k_2-k_{j_2}(t)|)}}{(3+|2^{j_2}\xi-k_2|)^4} \right).
\end{align*}
These last two sums are bounded by the first point of Lemma \ref{Lemma:ordinary}. Using
\begin{equation}\label{eqn:ineserie6}
\sum_{j_1=0}^{n-1} 2^{j_1(1-H_1)} \sqrt{\log(3+j_1)} \leq C 2^{n(1-H_1) } \sqrt{\log(n)}
\end{equation}
instead of \eqref{eqn:ineserie1}, we conclude, just as in Lemma \ref{lem:rapid1}, that the desired inequality holds.
\end{proof}

\begin{Lemma}\label{lem:ordinary2}
There exists a deterministic constant $C>0$ such that, for all $t,s \in (0,1)$ and $0 \leq j_1<n \leq j_2$,
the quantities
\begin{equation}\label{eqn:r<>3t}
\sum_{k_1 \in \Z} \sum_{k_2 \in \Z_{j_2}^<(t,s)}L_{j_1,j_2}^{k_1,k_2}(t) \left|I_{j_1,j_2}^{k_1,k_2}[t,s]  \right|
\end{equation}
\begin{equation}\label{eqn:r<>2t}
\sum_{k_1 \in \Z} \sum_{k_2 \in \Z_{j_2}^>(t,s)}L_{j_1,j_2}^{k_1,k_2}(t) \left| I_{j_1,j_2}^{k_1,k_2}[t,s] \right|
\end{equation}
are bounded from above by
\[
C \sqrt{j_2-n+1} \sqrt{\log(3+j_1)}\sqrt{\log(3+j_2)}2^{-j_2}.
\]
\end{Lemma}

\begin{proof}
Let us prove the bound for \eqref{eqn:r<>3t}, the argument for \eqref{eqn:r<>2t} being similar. We have, by  the first part of Lemma \ref{Lemma:ordinary}, for $0\leq j_1<n \leq j_2$,
\begin{align*}
\eqref{eqn:r<>3t} \leq C \sqrt{\log(3+j_1)} \int_{[s,t]} \sum_{k_2 \in \Z_{j_2}^{<}(t,s)} \frac{\sqrt{\log(3+j_2+|k_2-k_{j_2}(t)|)}}{(3+|2^{j_2}x-k_2|)^4} \, dx
\end{align*}
and, as for all $k_2 \in \Z_{j_2}^{<}(t,s)$ and  $x \in [s,t]$ we have
\[ |k_2-k_{j_2}(t)| \leq |2^{j_2}x-k_2| + |k_{j_2}(t)-2^{j_2}x| \leq |2^{j_2}x-k_2|+2^{j_2-n}+1\]
and, by inequality \eqref{eqn:inelog},
\[ \sqrt{\log(3+j_2+|k_2-k_{j_2}(t)|)} \leq C \sqrt{j_2-n+1}\sqrt{\log(3+j_2)} \sqrt{\log(3+|2^{j_2}x-k_2|)}\]
it just remains us to use the bound \eqref{eqn:intavec2} to write
\begin{equation}\label{eqn:boundfor19ord}
\eqref{eqn:r<>3t} \leq C \sqrt{j_2-n+1} \sqrt{\log(3+j_1)}\sqrt{\log(3+j_2)}2^{-j_2}.
\end{equation}
\end{proof}

\begin{Lemma}\label{lem:ordinary3}
There exists a deterministic constant $C>0$ such that, for all $t,s \in (0,1)$ and $0 \leq j_1 <n \leq j_2$, the quantities
\begin{equation}\label{eqn:ordinary31}
\sum_{k_1 \in \Z} \sum_{k_2 \in \Z_{j_2}[t,s]} L_{j_1,j_2}^{k_1,k_2}(t) \left|\int_{- \infty}^{\min\{s,t\}} \psi_{H_1}(2^{j_1} x-k_1) \psi_{H_2}(2^{j_2} x-k_2) \, dx \right|
\end{equation}
\begin{equation}\label{eqn:ordinary32}
 \sum_{k_1 \in \Z} \sum_{k_2 \in \Z_{j_2}[t,s]}L_{j_1,j_2}^{k_1,k_2}(t)\left| \int_{\max\{s,t\}}^{+\infty} \psi_{H_1}(2^{j_1} x-k_1) \psi_{H_2}(2^{j_2} x-k_2) \, dx\right| 
\end{equation}
are bounded from above by
\[
C \sqrt{\log(3+j_1)} \sqrt{\log(3+j_2)} \sqrt{j_2-n+1} 2^{-j_2}.
\]
\end{Lemma}

\begin{proof}
Again we assume $s \leq t$. First, using the fast decay of the fractional antiderivatives of $\psi$ \eqref{eqn:fastdecay}, \eqref{eqn:ordinary31} is bounded from above by
\begin{align}\label{eqn:demoordistep2}
\int_{- \infty}^{s} \sum_{k_1 \in \Z} \sum_{k_2 \in \Z_{j_2}[t,s]}  \frac{\sqrt{\log(3+j_1+|k_1-k_{j_1}(t)|)}}{(3+|2^{j_1}x-k_1|)^4}  \frac{\sqrt{\log(3+j_2+|k_2-k_{j_2}(t)|)}}{(3+|2^{j_2}x-k_2|)^4} \, dx.
\end{align}
Observe that, for all $k_1 \in \Z$, $k_2 \in \Z_{j_2}[t,s]$ and $x \in (- \infty,s]$, we have, as $j_1<n \leq j_2$,
\begin{align*}
|2^{j_1}x-k_{j_1}(t)| & \leq |2^{j_1}x-2^{j_1-j_2}k_2|+|2^{j_1-j_2}k_2-2^{j_1}t| + |2^{j_1}t-k_{j_1}(t)| \\
& \leq |2^{j_2}x-k_2| +2
\end{align*}
and therefore
\[ |k_1-k_{j_1}(t)| \leq |2^{j_1}x-k_1| + |2^{j_2}x-k_2| +2\]
while
\[|k_2-k_{j_2}(t)| \leq |k_2-2^{j_2}t| + |2^{j_2}t -k_{j_2}(t)| \leq 2^{j_2-n}+1.\]
it allows to write, thanks to inequality \eqref{eqn:inelog}, the boundedness of the function \eqref{eqn:fctbounded} and inequality \eqref{eqn:boundfor21}
\begin{align*}
| \eqref{eqn:demoordistep2}| & \leq C\sqrt{\log(3+j_1)} \sqrt{\log(3+j_2)} \sqrt{j_2-n+1} \int_{- \infty}^{s} \sum_{k_2 \in \Z_{j_2}[t,s]}   \frac{dx}{(3+|2^{j_2}x-k_2|)^{3}} \\
& \leq C \sqrt{\log(3+j_1)} \sqrt{\log(3+j_2)} \sqrt{j_2-n+1} 2^{-j_2}.
\end{align*}
We bound the second sums in the same way.
\end{proof}

\begin{Lemma}\label{lem:ordinary4}
There exists a deterministic constant $C>0$ such that, for all $t,s \in (0,1)$, the quantities
\[
\sum_{0\leq j_1<n} \sum_{j_2 \geq n} \sum_{(k_1,k_2) \in \Z^2} 2^{j_1(1-H_1)} 2^{j_2(1-H_2)} L_{j_1,j_2}^{k_1,k_2}(t) |I_{j_1,j_2}^{k_1,k_2}[t,s] |
\]
\[ \sum_{j_1 \geq n} \sum_{0 \leq j_2 < n} \sum_{(k_1,k_2) \in \Z^2} 2^{j_1(1-H_1)} 2^{j_2(1-H_2)} L_{j_1,j_2}^{k_1,k_2}(t) |I_{j_1,j_2}^{k_1,k_2}[t,s]| \] 
are bounded from above by
\[
C |t-s|^{H_1+H_2-1} \log \log |t-s|^{-1} .
\]
\end{Lemma}

\begin{proof}
The proof is exactly the same as the one of Lemma \ref{lem:rapid4} excepted that we use Lemmata \ref{lem:ordinary2} and \ref{lem:ordinary3} instead of Lemmata \ref{lem:rapid2} and \ref{lem:rapid3} respectively and that we conclude using again \eqref{eqn:ineserie6} instead of \eqref{eqn:ineserie1} and
\begin{equation}\label{eqn:ineserie7}
\sum_{j_2=n}^{+ \infty} 2^{ -j_2 H_2} \sqrt{j_2-n+1} \sqrt{\log(3+j_2)} \leq C' 2^{- n H_2} \sqrt{\log(n)}.
\end{equation}
instead of \eqref{eqn:ineserie2}.
\end{proof}

\begin{Lemma}\label{lem:ordinary5}
There exists a deterministic constant $C>0$ such that, for all $t,s \in (0,1)$ and $n \leq j_1 \leq j_2$,
the quantities \eqref{eqn:r<>3t} and \eqref{eqn:r<>2t} are bounded from above by
\[
C \sqrt{j_2-n+1} \sqrt{j_1-n+1} \sqrt{\log(3+j_1)}\sqrt{\log(3+j_2)}2^{-j_2}.
\]
\end{Lemma}

\begin{proof}
The proof is exactly the same as for Lemma \ref{lem:ordinary2} except that, here, we use the second part of Lemma \ref{Lemma:ordinary} instead of the first one.
\end{proof}

\begin{Lemma}\label{lem:ordinary6}
There exists a deterministic constant $C>0$ such that, for all $t,s \in (0,1)$ and $n \leq j_1 \leq j_2$
the quantities \eqref{eqn:ordinary31} and \eqref{eqn:ordinary32} are bounded from above by
\[
C \sqrt{j_1-n+1} \sqrt{j_2-n+1} \sqrt{\log(3+j_1)}\sqrt{\log(3+j_2)}2^{-j_2}.
\]
\end{Lemma}

\begin{proof}
The proof is exactly the same as for Lemma \ref{lem:ordinary3} except that, here, we use the second part of Lemma \ref{Lemma:ordinary} instead of the first one.
\end{proof}

Just as we did for the rapid points, it remains us to bound the random variables $\Xi_j(\lambda)$. Here, we don't want anymore to show the existence of an uniform modulus but only a pointwise modulus of continuity at a fixed point of interest $t$. Therefore, we just have to bound, for all $n \in \N$ the random variables $\Xi_j(\lambda)$ for $j \geq n$ and $\lambda \in  3 \lambda_n(t)$. We thus have the following result.

\begin{Lemma}\label{lemma:borelcantellilocal}
For all $t \in (0,1)$, there exist an event $\widetilde{\Omega_t}$ of probability $1$ and a positive random variable $C_{t,2}$ with finite moment of any order such that, on $\widetilde{\Omega_t}$,
\begin{equation}\label{eqn:borelcantellilocal}
\forall  n \in \N, ~ \forall  \lambda  \in 3\lambda_n(t),~ \forall j \geq n,~ \Xi_j(\lambda) \leq C_{t,2} \, (j-n+1)\log(n).
\end{equation}
\end{Lemma}
\begin{proof}
If $t \in (0,1)$ is fixed and $\theta>0$, let us define the event
\[ A_n(t) = \{ \forall \lambda \in 3\lambda_n(t) ~ \forall j \geq n, ~ \Xi_j(\lambda) \leq \theta (j-n+1) \log(n)\}.\]
Similarly to Lemma \ref{lemma:borelcantellitotal}, we get
\begin{align*}
\mathbb{P}(A_n(t)^c) & \leq C \sum_{j \geq n} 2^{4(j-n)} \exp(- \overset{\star}{C}\theta(j-n+1)\log(n)) \\
&  \leq C \exp(-\overset{\star}{C} \theta \log(n)) \sum_{j \geq n} 2^{4(j-n)} \exp(- \overset{\star}{C}\theta(j-n)),
\end{align*}
for a determistic constant $C>0$. Therefore, if we take again $\theta > 4 \log(2)/\overset{\star}{C}$ then Borel-Cantelli Lemma implies the existence of an event $\widetilde{\Omega_t}$ of probability $1$ and $C_{t,2}$ a positive random variable of finite moment of any order such that, on $\widetilde{\Omega_t}$, assertion \eqref{eqn:borelcantellilocal} holds.
\end{proof}

\begin{Lemma}\label{lem:ordinary7}
If $t \in (0,1)$, let $\Omega_t^*$ be the event of probability $1$ where inequality \eqref{eqn:ordinary} holds and $\widetilde{\Omega_t}$ be the event of probability $1$ given by Lemma \ref{lemma:borelcantellilocal}. There exists a positive random variable $C_{t,3}$  with finite moment of any order such that, on $\Omega_t^* \cap \widetilde{\Omega_t}$, for all $s \in (0,1)$ the random variable
\begin{equation}\label{eqn:lastrvord}
\left| \sum_{j_1 \geq n} \sum_{j_2 > n} \sum_{(k_1,k_2) \in \Z^2} 2^{j_1(1-H_1)} 2^{j_2(1-H_2)} \varepsilon_{j_1,j_2}^{k_1,k_2} I_{j_1,j_2}^{k_1,k_2}[t,s] \right|
\end{equation}
is bounded from above by
\[
C_{t,3} |t-s|^{H_1+H_2-1} \log |t-s|^{-1}.
\]
\end{Lemma}

\begin{proof}
Again, we use the split \eqref{somme43} and we only do the details for the first sum. We deal with the series \eqref{eqn:lemma3191} and \eqref{eqn:lemma3192} in the same way that in Lemma \ref{lem:rapid5} but using inequality \eqref{eqn:ordinary} and Lemmata \ref{lem:ordinary5} and \ref{lem:ordinary6} and finally inequality \eqref{eqn:ineserie7}. 

For \eqref{eqn:lemma3193}, first, by Lemma \ref{lem:ordinary6} and inequality \eqref{eqn:ordinary}, we have, on $\Omega_t^* \cap \widetilde{\Omega_t}$
\begin{align}
\nonumber &\left| \sum_{j_1 \geq n} \sum_{j_2 \geq j_1} \sum_{k_1 \in \Z} \sum_{k_2 \in \Z_{j_2}[t,s]} 2^{j_1(1-H_1)} 2^{j_2(1-H_2)} \varepsilon_{j_1,j_2}^{k_1,k_2} \int_{- \infty}^{s} \psi_{H_1}(2^{j_1} x-k_1) \psi_{H_2}(2^{j_2} x-k_2) \, dx \right| \\
\nonumber \quad & \leq C C_{t,1} \sum_{j_1 \geq n} \sum_{j_2 \geq j_1} 2^{j_1(1-H_1)} 2^{- j_2 H_2 } \sqrt{j_1-n+1} \sqrt{j_2-n+1} \sqrt{\log(3+j_1)}\sqrt{\log(3+j_2)} \\
& \leq C C_1 2^{n(1-H_1-H_2)} \log(n). \label{eqn:ineserie8}
\end{align}
We bound 
\begin{align*}
\left| \sum_{j_1 \geq n} \sum_{j_2 \geq j_1} \sum_{k_1 \in \Z} \sum_{k_2 \in \Z_{j_2}[t,s]} 2^{j_1(1-H_1)} 2^{j_2(1-H_2)} \varepsilon_{j_1,j_2}^{k_1,k_2} \int_{- \infty}^{s} \psi_{H_1}(2^{j_1} x-k_1) \psi_{H_2}(2^{j_2} x-k_2) \, dx \right|
\end{align*}
on $\Omega_t^* \cap \widetilde{\Omega_t}$ exactly in the same way.

To finish the proof, again, we have to bound \eqref{eqn:r>>6}, \eqref{eqn:r>>5} and \eqref{eqn:r>>3} for $\ell \in \{j,j+1 \}$ (with $j \geq n$) on $\Omega_t^* \cap \widetilde{\Omega_t}$. For \eqref{eqn:r>>6}, in the case $\ell=j$, one can note that, for all $k_2 \in \Z_{j}[t,s]$, $|k_2-k_j(t)| \leq 2^{j-n} +1$ and, for all $k_1 \in \Z_j^<(t,s)$, $|k_1-k_j(t)| \leq |2^j\min\{t,s\}-k_1|+2^{j-n}+1$. Using the same tricks as in \eqref{eqn:boundfor26}, we get, on $\Omega_t^* \cap \widetilde{\Omega_t}$
\begin{align*}
|\eqref{eqn:r>>6}| &\leq C C_{t,1} (j-n+1) \log(3+j) 2^{-j} \sum_{k_1 \in \Z_j^<(t,s)} \frac{1}{(2+|2^j\min\{s,t\}-k_1|)^{3}} \\
& \leq C (j-n+1) \log(3+j) 2^{-j}.
\end{align*}
The bounds for \eqref{eqn:r>>5} and in the case $\ell=j+1$  are obtained in the same way. Finally to bound \eqref{eqn:r>>3}, we use \eqref{eqn:borelcantellilocal} and \eqref{eqn:boundl2} and get on $\Omega_t^* \cap \widetilde{\Omega_t}$
\begin{align*}
&\left |  \sum_{j \geq n} 2^{j(2-H_1-H_2)}\sum_{\ell=j}^{j+1} \sum_{k_1 \in \Z_j [t,s]} \sum_{k_2 \in \Z_{\ell}[t,s]} \varepsilon_{j,\ell}^{k_1,k_2} I_{j,\ell}^{k_1,k_2}\right| \\ \quad & \leq C C_{t,2} \sum_{j \geq n} 2^{j(\frac{3}{2}-H_1-H_2)} 2^{-\frac{n}{2}} (j-n+1) \log(n) \\
& \leq  C C_{t,2} 2^{n(1-H_1-H_2)}  \log(n).
\end{align*}
We conclude that \eqref{eqn:lastrvord} is bounded from above by
\[  C \max\{C_{t,1},C_{t,2}\} |t-s|^{H_1+H_2-1} \log |t-s|^{-1}\]
on $\Omega_t^* \cap \widetilde{\Omega_t}$.
\end{proof}

We can now prove Proposition \ref{Prop:ordinary}.

\begin{proof}[Proof of Proposition \ref{Prop:ordinary}]
Let us fix $t \in (0,1)$ and consider $ \omega \in \Omega_t^* \cap \widetilde{\Omega_t}$. For all $s \in (0,1)$, we write\footnote{We recall that $R_{H_1,H_2}'$ is defined in Remark \ref{rmk:method}.} 
\begin{align*}
\begin{split}
|R_{H_1,H_2}'&(t,\omega)-R_{H_1,H_2}'(s,\omega)| \\ & \leq \left| \sum_{0 \leq j_1 < n} \sum_{0 \leq j_2 < n} \sum_{(k_1,k_2) \in \Z^2} 2^{j_1(1-H_1)} 2^{j_2(1-H_2)}  \varepsilon_{j_1,j_2}^{k_1,k_2}(\omega) I_{j_1,j_2}^{k_1,k_2}[t,s]\right| \\
\quad & + \left| \sum_{0 \leq j_1 < n} \sum_{j_2 \geq  n} \sum_{(k_1,k_2) \in \Z^2} 2^{j_1(1-H_1)} 2^{j_2(1-H_2)} \varepsilon_{j_1,j_2}^{k_1,k_2} (\omega) I_{j_1,j_2}^{k_1,k_2}[t,s]\right| \\
& + \left| \sum_{j_1 \geq n} \sum_{0 \leq j_2 < n} \sum_{(k_1,k_2) \in \Z^2} 2^{j_1(1-H_1)} 2^{j_2(1-H_2)}  \varepsilon_{j_1,j_2}^{k_1,k_2}(\omega)  I_{j_1,j_2}^{k_1,k_2}[t,s]\right| \\
& + \left| \sum_{j_1 \geq n} \sum_{j_2 \geq n} \sum_{(k_1,k_2) \in \Z^2} 2^{j_1(1-H_1)} 2^{j_2(1-H_2)}  \varepsilon_{j_1,j_2}^{k_1,k_2}(\omega)  I_{j_1,j_2}^{k_1,k_2}[t,s]\right|.
\end{split}
\end{align*}
We bound from above the first sum by inequality \eqref{eqn:ordinary} and Lemma \ref{lem:ordinary1}, the second and the third sums by inequality \eqref{eqn:ordinary} and Lemma \ref{lem:ordinary4} and the last sum by Lemma \ref{lem:ordinary7}.

Using inequalities \eqref{eqn:ineserie7} and \eqref{eqn:ineserie8} and Remark \ref{rmk:method}, one can finally write that for all $t\in (0,1)$, for all $\omega$ in the event of probability $1$ $\Omega_t^* \cap \widetilde{\Omega_t}$
\[ \limsup_{s \to t} \dfrac{|R_{H_1,H_2}(t,\omega)-R_{H_1,H_2}(s,\omega) |}{|t-s|^{H_1+H_2-1}  \log\log |t-s|^{-1}} < + \infty\]
and we conclude by Fubini Theorem.
\end{proof}

\subsection{Slow points}

In this section, we aim at showing that the generalized Rosenblatt process admits slow points: we prove the following Proposition.

\begin{Prop}\label{prop:slow}
There exists an event $\Omega_{\text{slo}}$ of probability $1$ such that for all $\omega \in \Omega_{\text{slo}}$ there exist $t \in (0,1)$ such that
\begin{equation}\label{eqn:Prop:slow}
\limsup_{s \to t} \dfrac{|R_{H_1,H_2}(t,\omega)-R_{H_1,H_2}(s,\omega) |}{|t-s|^{H_1+H_2-1} } < + \infty.
\end{equation}
\end{Prop}

In \cite{kahane85}, Kahane described a procedure to insure the existence of slow points for the Brownian motion. This procedure was then generalized in \cite{esserloosveldt} to fit for any arbitrary fractional Brownian motion. It consists in showing that for any $m>0$, almost surely, there exist $\mu >0$ and $t \in (0,1)$ such that, if one sets
\begin{equation}\label{eqn:slownota1}
\Lambda_{j}^0(t) = \{ \lambda \in \Lambda_j : \, |s(\lambda(t))-s(\lambda)| \leq 1 \}
\end{equation}
and, for all $1 \leq l$
\begin{equation}\label{eqn:slownota2}
\Lambda_{j}^l(t) = \{ \lambda \in \Lambda_j,  : \, 2^{m(l-1)}<|s(\lambda(t))-s(\lambda)| \leq 2^{ml} \},
\end{equation}
then, for all $\lambda \in \Lambda_{j}^l(t)$ we have
\begin{equation}\label{eqn:slownesti}
|\varepsilon_\lambda| \leq 2^l \mu,
\end{equation}
where $\varepsilon_\lambda$ is the random variable
\[ 2^\frac{j}{2} \int_\R \psi_\lambda(x) \, dB(x).\]
In this procedure,  if $\mu \in \N$,for all $j,l \in \N_0$ and $\lambda \in \Lambda_j$, $\lambda \subseteq [0,1]$, we define
\[ \Lambda_{j,l}(\lambda) = \{\lambda' \in \Lambda_j,  \, : \, |s(\lambda)-s(\lambda')| \leq 2^{ml} \}\]
and the random set
\[ S_{j,l}^\mu = \{ \lambda' \in \Lambda_j,  \, : \, 2^l
  \mu < |\varepsilon_{\lambda'}| \leq 2^{l+1} \mu \}.\] 
Finally we consider the random set
\[ I_j^\mu=\{\lambda \in \Lambda_j, \lambda \subseteq [0,1] : \, \forall l \in \N_0 , \, \Lambda_{j,l}(\lambda) \cap S_{j,l}^\mu = \cmemptyset \},\]
and show that almost surely, there exists $\mu \in \N$ such that
\[ S_{\text{low}}^\mu = \bigcap_{j \in \N_{0}} \bigcup_{\lambda \in I_j^\mu } \overline{\lambda} \neq \cmemptyset\]
which is equivalent to the fact that, for any $J$
\[ S_{\text{low},J}^\mu= \bigcap_{j \leq J} \bigcup_{\lambda \in I_j^\mu } \overline{\lambda} \neq \cmemptyset \]
as $(S_{\text{low},J}^\mu)_J$ is a decreasing sequence of compact sets. To do so, let us denote by $2 S_{\text{low},J}^\mu$ the sets of dyadic intervals of scale $J+1$ obtained by cutting in two the remaining intervals\footnote{The interval $[k2^{-j},(k+1)2^{-j}]$ is cut into $[(2k)2 ^{-(j+1)},(2k+1)2 ^{-(j+1)}]$ and $[(2k+1)2 ^{-(j+1)},(2k+2)2 ^{-(j+1)}]$.} in $S_{\text{low},J}^\mu$ and remark that $S_{\text{low},J+1}^\mu$ is obtained from  $2 S_{\text{low},J}^\mu$ by removing the dyadic intervals $\lambda$  such that $\Lambda_{J+1,l}(\lambda) \cap S_{J+1,l}^\mu \neq \cmemptyset$ for a $l \in \N_0$. But now, if $\xi \sim \mathcal{N}(0,1)$, we set, for all such a $l$
\[p_l(\mu) = \mathbb{P}(2^l \mu < |\xi| \leq 2^{l+1} \mu).\]
and note that, if $N$ is the number of intervals of $S_{\text{low},J}^\mu$, counting the number of intervals in  $2S_{\text{low},J}^\mu \cap S_{J+1,l}^\mu$ is a binomial random variable of parameter $2N$ and $p_l(\mu)$ and this number is thus bounded by
\[ 2N (p_l(\mu) + (l+1) \sqrt{p_l(\mu)(1-p_l(\mu))} )\]
on an event of probability $1-(l+1)^{-2} N^{-1}$. Therefore, to pass from $S_{\text{low},J}^\mu$ to $S_{\text{low},J+1}^\mu$
we remove at most
\[ 2N \sum_{l=0}^{+ \infty} (2^{ml+1}+1) (p_l(\mu) + (l+1) \sqrt{p_l(\mu)(1-p_l(\mu))} ) \]
intervals with probability greater than $1-N^{-1}$. But if $\mu$ is large enough, as $p_l(\mu)$ is of order $\frac{e^{-(2^l \mu)^2}}{2^l \mu}$, one can make sure that this last term is bounded by $\frac{N}{2}$. So, if $N_J^\mu$ is the random variable counting the number of subintervals of $S_{\text{low},J}^\mu$, we have
\[ \mathbb{P}(N_{J+1}^\mu \geq \frac{3}{2}N_J^\mu | N_J^\mu =N) \geq 1-N^{-1}\]
which leads to the recursive formula
\[\mathbb{P}(N_{J+1}^\mu \geq (\frac{3}{2})^{J+1}) \geq
  (1-(\frac{2}{3})^J) \mathbb{P}(N_{J}^\mu \geq (\frac{3}{2})^{J}) ,
  \quad \forall J \in \N_{0},  \]
  see \cite[Lemma 3.6 and Theorem 3.7.]{esserloosveldt}. Finally, we deduce
\begin{equation}\label{eqn:proba1pourslow}
\mathbb{P}\big( \bigcup_\mu \bigcap_{J \in\N_{0} } (N_J^\mu \geq 1)\big) =1.
\end{equation}
Moreover, we can show that, in this case, $S_{\text{low}}^\mu \cap (0,1) \neq \cmemptyset$. If $\alpha>0$, applying this procedure with $\frac{1}{m}<\alpha$ gives us that any point $t \in S_{\text{low}}^\mu \cap (0,1)$ is a slow point of the fractional Brownian motion of exponent $\alpha$.

From formulas \eqref{eqn:randomvar1} and \eqref{eqn:randomvar2}, we see that this procedure is also useful to bound the random variables appearing in the expansion \eqref{eqn:waveletserie} of the generalized Rosenblatt process. But, from the proofs of Propositions \ref{prop:rapid} and \ref{Prop:ordinary} we know that this is not sufficient and we also need to give a bound for the random variables $\Xi_j(\lambda)$, for $\lambda \in 3 \lambda_n(t), n \in \N$ and $j \geq n$. Such dyadic intervals are precisely the ones in the set $\Lambda_{n,0}(\lambda_n(t))$ and this fact forces us to consider the following modification of the procedure. For all $j \in \N$, if $l \neq 0$, the sets $S_{j,l}^\mu$ remain untouched as well as its associated probability $p_l(\mu)$ while for $l=0$ we set
\[ S_{j,0}^\mu = \{ \lambda' \in \Lambda_j, \lambda' \subseteq [0,1] \, : \,\exists j' \geq j \, \Xi_{j'}(\lambda') > (j'-j+1) \mu \},\] 
with associated probability (which only depends on $\mu$)
\[ p_0(\mu) = \mathbb{P}(\exists j' \geq j \, \Xi_{j'}(\lambda) > (j'-j+1) \mu).\]
As $\Xi_{j'}(\lambda_1)$ is independent of $\Xi_{j'}(\lambda_2)$ as soon as $\lambda_1 \cap \lambda_2 = \cmemptyset$, for all $J \in \N$, if $N$ is again the number of dyadic intervals of $S_{\text{low},J}^\mu$, the number of such intervals in $2S_{\text{low},J}^\mu \cap S_{J+1,0}^\mu$ is still a binomial random variable of parameter $2N$ and $p_0(\mu)$. Therefore
If $\mu$ is large enough, using Theorems \ref{thm:bound} and \ref{thm:bound2}, one can still affirm
\[ 2N \sum_{l=0}^{+ \infty} (2^{ml+1}+1) (p_l(\mu) + (l+1) \sqrt{p_l(\mu)(1-p_l(\mu))} ) \leq \frac{N}{2} \]
and the end of the procedure is saved: equality \eqref{eqn:proba1pourslow} still holds. Now, if \linebreak $t \in S_{\text{low}}^\mu \cap (0,1)$ we know that
\begin{equation}\label{eqn:concluprocedure}
\forall \, n \in \N, \, \forall \, \lambda  \in 3\lambda_n(t) \,, \forall j \geq n \,, \Xi_j(\lambda) \leq  (j-n+1)\mu.
\end{equation}
Let us remark that, as for all $\lambda \in \Lambda_n$, $|\varepsilon_\lambda^2| \leq 2 \Xi_n(\lambda)+1$, we still have, in this case, for all $\lambda \in 3 \lambda_n(t)$, $|\varepsilon_\lambda| \leq C \mu$, for a deterministic constant $C>0$.
.

Starting from now we take $m$ such that $1/m < \min\{H_1,H_2\}$ and  \linebreak $2/m<1-H_1-H_2$

In order to use notations \eqref{eqn:slownota1} and \eqref{eqn:slownota2}, here after $\lambda_1$ (resp. $\lambda_2$) will always stand for the dyadic interval $[k_1 2^{-j_1}, (k_1+1)2^{-j_1})$ (resp. $[k_2 2^{-j_2}, (k_2+1)2^{-j_2})$) and $\psi_{\lambda _1}$ (resp. $\psi_{\lambda _2}$) will be the associated antiderivative of wavelet $\psi_{H_1}(2^{j_1} \cdot - k_1)$ (resp.  $\psi_{H_2}(2^{j_2} \cdot - k_2)$) and $I_{\lambda_1,\lambda_2}[t,s]$ will stand for $I_{j_1,j_2}^{k_1,k_2}[t,s]$. Finally, $\varepsilon_{\lambda_1,\lambda_2}$ will stand for $\varepsilon_{j_1,j_2}^{k_1,k_2}$. If $t \in (0,1)$, let $(y_{\lambda}(t))_{\lambda \Lambda}$ be the sequence defined by
\[ y_{\lambda}(t) = 2^{l} \text{ if } \lambda \in \Lambda_{j}^{l}(t).\]

Note that, if we apply the preceding procedure, we find $\Omega_\text{slo}$ an event of probability $1$ such that, for all $\omega \in \Omega_\text{slo}$, there exists $\mu$ for which $S_{\text{low}}^\mu \cap (0,1)\neq \cmemptyset$. Then, if $t$ belong to this set, we have, thanks to inequality \eqref{eqn:slownesti} and equalities \eqref{eqn:randomvar1} and \eqref{eqn:randomvar2}
\begin{equation}\label{eqn:slowrvesti}
|\varepsilon_{\lambda_1,\lambda_2}(\omega) | \leq  C \mu^2 y_{\lambda_1}(t) y_{\lambda_2}(t),
\end{equation}
for a deterministic constant $C>0$. Again, we need to adapt the Lemmata from previous sections with this alternative upper bound.

\begin{Lemma}\label{lem:slow1}
There exists a deterministic constant $C>0$ such that, for all $t,s \in (0,1)$ we have
\begin{align*}
\sum_{0 \leq j_1<n} &\sum_{0 \leq j_2<n} \sum_{\lambda_1 \in \Lambda_{j_1}, \lambda_2 \in \Lambda_{j_2}} 2^{j_1(1-H_1)} 2^{j_2(1-H_2)} y_{\lambda_1}(t) y_{\lambda_2}(t) \left| I_{\lambda_1,\lambda_2}[t,s] \right| \\ & \leq  C |t-s|^{H_1+H_2-1}. 
\end{align*}
\end{Lemma}

\begin{proof}
If $\xi \in [s,t]$ and $\lambda \in \lambda_j^l(t)$, for $0 \leq j < n$ and $l \geq 1$, \[|2^j \xi - s(\lambda)| \geq |s(\lambda(t))-s(\lambda)|-2 > 2^{m(l-1)}-2\] and so, using the fast decay of the fractional antiderivatives of $\psi$ \eqref{eqn:fastdecay} and the definition of $(y_{\lambda})_{\lambda \in \Lambda}$ , we get for $0 \leq j_1,j_2 <n$
\begin{align}
\nonumber &\sum_{\lambda_1 \in \Lambda_{j_1}, \lambda_2 \in \Lambda_{j_2} } y_{\lambda_1}(t) y_{\lambda_2}(t) |\psi_{\lambda_1}(\xi) \psi_{\lambda_2}(\xi)|  \\
\nonumber & \quad = \sum_{(l_1,l_2) \in \N_0^2} \sum_{\lambda_1 \Lambda_{j_1}^{l_1}(t)} \sum_{\lambda_2 \Lambda_{j_2}^{l_2}(t)} y_{\lambda_1}(t) y_{\lambda_2}(t) |\psi_{\lambda_1}(\xi) \psi_{\lambda_2}(\xi)| \\
\nonumber & \quad \leq C \sum_{(l_1,l_2) \in \N_0} \sum_{\lambda_1 \Lambda_{j_1}^{l_1}(t)} \sum_{\lambda_2 \Lambda_{j_2}^{l_2}(t)}  \frac{2^{l_1+l_2} }{(3+|2^{j_1}\xi-k_1|)^{4} (3+|2^{j_2}\xi-k_2|)^{4}} \\
\nonumber & \quad \leq C \sum_{(l_1,l_2) \in \N_0} \sum_{\lambda_1 \Lambda_{j_1}^{l_1}(t)} \sum_{\lambda_2 \Lambda_{j_2}^{l_2}(t)}   \frac{2^{l_1+l_2} 2^{-m(l_1+l_2)}}{(3+|2^{j_1}\xi-k_1|)^{3} (3+|2^{j_2}\xi-k_2|)^{3}} \\
\nonumber & \leq C \sum_{k_1 \in \Z} \frac{1}{(3+|2^{j_1}\xi-k_1|)^{3}} \sum_{k_2 \in \Z} \frac{1}{(3+|2^{j_2}\xi-k_2|)^{3}} \\
& \leq C. \label{slow:step1}
\end{align}
It leads, just as in Lemmata \ref{lem:rapid1} and \ref{lem:ordinary1}, to the desired estimate.
\end{proof}

In what follows, we use these notations instead of the one given in Definition \ref{def:partitionofZ}:
\[ \Lambda_{j_2}^{<}(t,s) = \{ \lambda_2 \in \Lambda_{j_2} \, : \, s(\lambda_2) \in \Z_{j_2}^<(t,s) \},\]
\[ \Lambda_{j_2}^{>}(t,s) = \{ \lambda_2 \in \Lambda_{j_2} \, : \, s(\lambda_2) \in \Z_{j_2}^>(t,s) \},\]
\[ \Lambda_{j_2}[t,s] = \{ \lambda_2 \in \Lambda_{j_2} \, : \, s(\lambda_2) \in \Z_{j_2}[t,s] \}.\]

\begin{Lemma}\label{lem:slow2}
There exists a deterministic constant $C>0$ such that, for all $t,s \in (0,1)$ and $0 \leq j_1<n \leq j_2$,
the quantities
\begin{equation}\label{eqn:r<>3s}
\sum_{\lambda_1  \in \Lambda_{j_1}} \sum_{\lambda_2 \in \Lambda_{j_2}^{<}(t,s)}y_{\lambda_1}(t) y_{\lambda_2}(t) \left|I_{\lambda_1,\lambda_2}[t,s]  \right|
\end{equation}
\begin{equation}\label{eqn:r<>2s}
\sum_{\lambda_1 \in \Lambda_{j_1}} \sum_{\lambda_2 \in \Lambda_{j_2}^{>}(t,s)} y_{\lambda_1}(t) y_{\lambda_2}(t) \left| I_{\lambda_1,\lambda_2}[t,s] \right|
\end{equation}
are bounded by
\[
C 2^{\frac{1}{m}(j_2-n)} 2^{-j_2}.
\]
\end{Lemma}

\begin{proof}
Again, we prove the bound for \eqref{eqn:r<>3s}, the reasoning for \eqref{eqn:r<>2s} being similar. Let us remark that, if $j_2 \geq n$ $x \in [s,t]$ and $\lambda_{j_2}(x) \in  \Lambda_{j_2}^{l}(t)$ then, the construction and the definition of $(y_{\lambda}(t))_{\lambda \in \Lambda}$ gives that
\begin{itemize}
\item $l \leq \frac{1}{m}(j_2-n)$, as $|s-t| \leq 2^{-n}$,
\item if $\lambda \in \Lambda_{j_2}^{l_2}(x)$ then $|y_\lambda| \leq 2^{l_2} 2^{l+1} \mu$ while, by definition, if $l_2 \geq 1$
\[ 3+|2^{j_2}x -s(\lambda)| \geq 2+2^{m(l_2-1)}.\]
\end{itemize}
Therefore, if we set
\[ D_{j_2}^l (t) = \bigcup_{\lambda \in  \Lambda_{j_2}^{l}(t)} \lambda,\]
we have
\begin{align}\label{eqn:demoslowsplit}
\begin{split}
\eqref{eqn:r<>3s} & \leq  \sum_{\lambda_1 \in \Lambda_{j_1}}\sum_{ \lambda_2 \in \Lambda_{j_2}^{<}(t,s) } y_{\lambda_1}(t) y_{\lambda_2}(t) \int_{[s,t]}|\psi_{\lambda_1}(x) \psi_{\lambda_2}(x)| \, dx \\
& \leq  \sum_{0 \leq l \leq  \frac{1}{m}(j_2-n)} \sum_{\lambda_1 \in \Lambda_{j_1}} \sum_{\lambda_2 \in \Lambda_{j_2}^{<}(t,s) } y_{\lambda_1}(t) y_{\lambda_2}(t)  \int_{D_{j_2}^l(t)} |\psi_{\lambda_1}(x) \psi_{\lambda_2}(x)| \, dx.
\end{split}
\end{align}
But, for all $x \in D_{j_2}^l$, using the same method as in \eqref{slow:step1}, but splitting the sums according to the set $\Lambda_{j_1}^{l_1}(x)$ and $\Lambda_{j_2}^{l_2}(x)$ on which $y_{\lambda_1}(t) y_{\lambda_2}(t) \leq  2^{l+l_1+l_2+1}$ we get
\begin{align}
\begin{split}
\sum_{\lambda_1 \in \Lambda_{j_1}} & \sum_{\lambda_2 \in \Lambda_{j_2}^{<}(t,s) }  |\varepsilon_{\lambda_1,\lambda_2} | |\psi_{\lambda_1}(x) \psi_{\lambda_2}(x)|\\
&\leq C 2^{l+1}  \sum_{\lambda_1 \in \Lambda_{j_1}} \frac{1}{(3+|2^{j_1}x-k_1|)^{3}} \sum_{\lambda_2 \in \Lambda_{j_2}^{<}(t,s)} \frac{1}{(3+|2^{j_2}x-k_2|)^{3}} \\
& \leq C 2^{l+1}  \sum_{\lambda_2 \in \Lambda_{j_2}^{<}(t,s)} \frac{1}{(3+|2^{j_2}x-k_2|)^{3}}.
\end{split}
\end{align}
Finally, using the techniques in \eqref{eqn:intavec2}, we get
\[ \eqref{eqn:r<>3s} \leq C 2^{\frac{1}{m}(j_2-n)} 2^{-j_2}. \]
\end{proof}

\begin{Lemma}\label{lem:slow3}
There exists a deterministic constant $C>0$ such that, for all $t,s \in (0,1)$ and $0 \leq j_1 <n \leq j_2$, the quantities
\begin{equation}\label{eqn:slow31}
\sum_{\lambda_1 \in \Lambda_{j_1}} \sum_{\lambda_2 \in \Lambda_{j_2}[t,s]} y_{\lambda_1}(t) y_{\lambda_2}(t) \left|\int_{- \infty}^{\min\{s,t\}} \psi_{H_1}(2^{j_1} x-k_1) \psi_{H_2}(2^{j_2} x-k_2) \, dx \right|
\end{equation}
\begin{equation}\label{eqn:slow32}
 \sum_{\lambda_1 \in \Lambda_{j_1}} \sum_{\lambda_2 \in \Lambda_{j_2}[t,s]}y_{\lambda_1}(t) y_{\lambda_2}(t)\left| \int_{\max\{s,t\}}^{+\infty} \psi_{H_1}(2^{j_1} x-k_1) \psi_{H_2}(2^{j_2} x-k_2) \, dx\right| 
\end{equation}
are bounded by
\[
C 2^{\frac{1}{m}(j_2-n)} 2^{-j_2}.
\]
\end{Lemma}

\begin{proof}
Again, we assume $s \leq t$. If $x \in (-\infty,s]$ is such that $\lambda_{j_1}(x) \in \Lambda_{j_1}^{l}(s)$, we have, for all $\lambda_1 \in \Lambda_{j_1}^{l_1}(x)$ and $\lambda_2 \in \Lambda_{j_2}[t,s] \cap \Lambda_{j_2}^{l_2}(s)$ (with $j_1 <n \leq j_2)$,
\begin{align}\label{eqn:demoslowsplitbound}
\begin{split}
\frac{y_{\lambda_1}(t) y_{\lambda_2}(t)}{(3+|2^{j_1}x-k_1|)^4(3+|2^{j_2}x-k_2|)^4} & \leq C  \frac{2^{\frac{1}{m}(j_2-n)+l+l_1+l_2+1} \mu^2}{(3+|2^{j_1}x-k_1|)^4(3+|2^{j_2}x-k_2|)^5} \\
& \leq C \frac{2^{\frac{1}{m}(j_2-n)+l+1}}{(3+|2^{j_1}x-k_1|)^{3}(3+|2^{j_2}x-k_2|)^{4}} \\
& \leq C \frac{2^{\frac{1}{m}(j_2-n)}}{(3+|2^{j_1}x-k_1|)^{3}(3+|2^{j_2}x-k_2|)^{3}}
\end{split}
\end{align}
because
\[ 3+|2^{j_2}x-k_2| = 3+k_2-2^{j_2}x \geq 2+2^{j_1}(s-x) \geq 2^{m(l-1)}. \]
Thus we get, using the fast decay of the fractional antiderivatives of the wavelet before splitting the integral over $(- \infty,s]$ into the integral over the sets \linebreak $(-\infty,s] \cap D_{j_1}^l(s)$, in the same way as in \eqref{eqn:demoslowsplit}, using \eqref{eqn:demoslowsplitbound} and finally the boundedness of the function \eqref{eqn:fctbounded} for $M=3$ and inequality \eqref{eqn:boundfor21}
\begin{align*}
\sum_{\lambda_1 \in \Lambda_{j_1}} & \sum_{\lambda_2 \in \Lambda_{j_2} [t,s]} \int_{-\infty}^s  |\psi_{\lambda_1}(x) \psi_{\lambda_2}(x)| \, dx \\ &\leq C 2^{\frac{1}{m}(j_2-n)}  \int_{- \infty}^s \sum_{\lambda_1 \in \Lambda_{j_1}} \sum_{\lambda_2 \in \Lambda_{j_2} [t,s]}\frac{dx }{(3+|2^{j_1}x-k_1|)^{3}(3+|2^{j_2}x-k_2|)^{3}} \\
& \leq C 2^{\frac{1}{m}(j_2-n)}  \int_{- \infty}^s  \sum_{\lambda_2 \in \Lambda_{j_2} [t,s]}\frac{dx }{(3+|2^{j_2}x-k_2|)^{3}} \\
& \leq C 2^{\frac{1}{m}(j_2-n)} 2^{-j_2}  .
\end{align*}
In the same way we get
\begin{equation}\label{eqn:demoslowholds}
\sum_{\lambda_1 \in \Lambda_{j_1}} \sum_{\lambda_2 \in \Lambda_{j_2}[t,s]}y_{\lambda_1}(t) y_{\lambda_2}(t)\left| \int_{\max\{t,s\}}^{+\infty} \psi_{H_1}(2^{j_1} x-k_1) \psi_{H_2}(2^{j_2} x-k_2) \, dx\right| \leq C  2^{-j_2}.
\end{equation}
\end{proof}

\begin{Lemma}\label{lem:slow4}
There exists a deterministic constant $C>0$ such that, for all $t,s \in (0,1)$, the quantities
\[
\sum_{0\leq j_1<n} \sum_{j_2 \geq n} \sum_{\lambda_1 \in \Lambda_{j_1}, \lambda_2 \in \Lambda_{j_2} } 2^{j_1(1-H_1)} 2^{j_2(1-H_2)}y_{\lambda_1}(t) y_{\lambda_2}(t) |I_{\lambda_1,\lambda_2}[t,s] |
\]
\[ \sum_{j_1 \geq n} \sum_{0 \leq j_2 < n} \sum_{\lambda_1 \in \Lambda_{j_1}, \lambda_2 \in \Lambda_{j_2}} 2^{j_1(1-H_1)} 2^{j_2(1-H_2)} y_{\lambda_1}(t) y_{\lambda_2}(t)|I_{\lambda_1,\lambda_2}[t,s]| \] 
are bounded by
\[
C |t-s|^{H_1+H_2-1}.
\]
\end{Lemma}

\begin{proof}
The proof is exactly the same as the one of Lemma \ref{lem:rapid4} excepted that we use Lemmata \ref{lem:slow2} and \ref{lem:slow3} instead of Lemmata \ref{lem:rapid2} and \ref{lem:rapid3} respectively. It leads on one side us to consider the sums
\[ \left( \left(\sum_{j_1 =0}^{n-1} 2^{j_1(1-H_1)}\sum_{j_2=n}^{+ \infty} 2^{\frac{1}{m}(j_2-n)} 2^{-j_2H_2 } \right) + 2^{n(1-H_1-H_2)} \right) \]
which are bounded by
\[ C 2^{n(1-H_1-H_2)} \leq C |t-s|^{H_1+H_2-1} \]
because $\frac{1}{m}<H_2$.
On the other side, if we write $I_{\lambda_1,\lambda_2}$ for $I_{j_1,j_2}^{k_1,k_2}$ in Lemma \ref{Lemma:intR}, we have, from it,
\begin{align}\label{eqn:demoslow:intR}
\begin{split}
&\left | \sum_{\lambda_1 \in \Lambda_{n-1}} \sum_{\lambda_2 \in \Lambda_{n}[t,s]} y_{\lambda_1}(t) y_{\lambda_2}(t) I_{\lambda_1,\lambda_2} \right|  \\
& \quad \leq C  2^{-n} \left| \sum_{l_1=0}^{+ \infty} \sum_{\lambda_1 \in \Lambda_{n-1}^{l_1}(t)}\sum_{\lambda_2 \in \Lambda_{n}[t,s]} \frac{2^{l_1}}{(3+|2k_1-k_2|)^4} \right| \\
& \quad \leq C 2^{-n}  \left|  \sum_{\lambda_1 \in \Lambda_{n-1}}\sum_{\lambda_2 \in \Lambda_{n}[t,s]} \frac{1}{(3+|2k_1-k_2|)^{3}} \right| \\
& \quad \leq C 2^{-n} .
\end{split}
\end{align}
\end{proof}

\begin{Lemma}\label{lem:slow5}
There exists a deterministic constant $C>0$ such that, for all $t,s \in (0,1)$ and $n \leq j_1 \leq j_2$
the quantities \eqref{eqn:r<>3s} and \eqref{eqn:r<>2s} are bounded by
\[
C 2^{\frac{1}{m}(j_1-n)} 2^{\frac{1}{m}(j_2-n)} 2^{-j_2}.
\]
\end{Lemma}

\begin{proof}
The proof is essentially the same as for Lemma \ref{lem:slow2} excepted that, now, as $n \leq j_1 \leq j_2$, we remark that if $x \in D_{j_2}^l(t)$ for a $0 \leq l \leq  \frac{1}{m}(j_2-n)$ then $x \in D_{j_1}^{l'}(t)$ for a $0\leq l' \leq \frac{1}{m}(j_1-n)$.
\end{proof}

\begin{Lemma}\label{lem:slow6}
There exists a deterministic constant $C>0$ such that, for all $t,s \in (0,1)$ and $n \leq j_1 \leq j_2$
the quantities \eqref{eqn:slow31} and \eqref{eqn:slow32} are bounded by
\[
C 2^{\frac{1}{m}(j_1-n)} 2^{\frac{1}{m}(j_2-n)} 2^{-j_2}.
\]
\end{Lemma}

\begin{proof}
The proof is essentially the same as for Lemma \ref{lem:slow3} and the only modification is the same as in the proof of Lemma \ref{lem:slow5}.
\end{proof}

This time, the bound for the random variables $\Xi_j(\lambda)$ are already considered in the construction and we can directly go to the proof of the main Proposition of this subsection.

\begin{proof}[Proof of Proposition \ref{prop:slow}]
If we apply the procedure with $m$ such that $1/m < \min\{H_1,H_2\}$ and $2/m<1-H_1-H_2$, we find an event $\Omega_\text{slo}$ of probability $1$ such that, for all $\omega \in \Omega_\text{slo}$, there is $\mu \in \N$ for which $S_{\text{low}}^\mu \cap (0,1)\neq \cmemptyset$. Then, if $\omega \in \Omega_\text{slo}$ and $t \in S_{\text{low}}^\mu(\omega) \cap (0,1)$ and $s \in (0,1)$, we write
\begin{align}\label{demoslow:decom}
\begin{split}
|R_{H_1,H_2}'(&t,\omega)-R_{H_1,H_2}'(s,\omega)| \\ & \leq \left| \sum_{0 \leq j_1 < n} \sum_{0 \leq j_2 < n} \sum_{\lambda_1 \in \Lambda_{j_1}, \lambda_2 \in \Lambda_{j_2}} 2^{j_1(1-H_1)} 2^{j_2(1-H_2)}  \varepsilon_{\lambda_1,\lambda_2}(\omega) I_{\lambda_1,\lambda_2}[t,s]\right| \\
\quad & + \left| \sum_{0 \leq j_1 < n} \sum_{j_2 \geq  n} \sum_{\lambda_1 \in \Lambda_{j_1}, \lambda_2 \in \Lambda_{j_2}} 2^{j_1(1-H_1)} 2^{j_2(1-H_2)} \varepsilon_{\lambda_1,\lambda_2}(\omega) I_{\lambda_1,\lambda_2}[t,s]\right| \\
& + \left| \sum_{j_1 \geq n} \sum_{0 \leq j_2 < n} \sum_{\lambda_1 \in \Lambda_{j_1}, \lambda_2 \in \Lambda_{j_2}}2^{j_1(1-H_1)} 2^{j_2(1-H_2)}  \varepsilon_{j_1,j_2}^{k_1,k_2}(\omega)  I_{\lambda_1,\lambda_2}[t,s]\right| \\
& + \left| \sum_{j_1 \geq n} \sum_{j_2 \geq n} \sum_{\lambda_1 \in \Lambda_{j_1}, \lambda_2 \in \Lambda_{j_2}} 2^{j_1(1-H_1)} 2^{j_2(1-H_2)}  \varepsilon_{\lambda_1,\lambda_2}(\omega) I_{\lambda_1,\lambda_2}[t,s]\right|.
\end{split}
\end{align}

As inequality \eqref{eqn:slowrvesti} holds, we use Lemma \ref{lem:slow1} to bound the first sum, and Lemma \ref{lem:slow4} to bound the second and the third one. For the last sum, from inequality \eqref{eqn:slowrvesti} and Lemmata \ref{lem:slow5} and \ref{lem:slow6}, it just remains us to find bound for the random variables \eqref{eqn:r>>6}, \eqref{eqn:r>>5} and \eqref{eqn:r>>3}with $\ell \in \{j,j+1\}$ on $\Omega_{\text{slo}}$. For \eqref{eqn:r>>6} with $\ell=j$ , we have, as in \eqref{eqn:demoslow:intR} and then \eqref{eqn:boundfor26}
\begin{align*}
&\left | \sum_{\lambda_1 \in \Lambda_{j}^<(t,s)} \sum_{\lambda_2 \in \Lambda_{j}[t,s]} \varepsilon_{\lambda_1,\lambda_2}(\omega) I_{\lambda_1,\lambda_2} \right|  \\
& \quad \leq C 2^{-j} 2^{\frac{2}{m}(j-n)}   \mu^2 \left| \sum_{\lambda_1 \in \Lambda_{j}^<(t,s)} \sum_{\lambda_2 \in \Lambda_{j}[t,s]} \frac{1}{(3+|2k_1-k_2|)^{3}} \right| \\
& \quad \leq C 2^{-j} 2^{\frac{2}{m}(j-n)} \mu^2 .
\end{align*}
The same bound holds when we consider the sums over $\lambda_1 \in \Lambda_{j}^>(t,s)$ or $\lambda_2 \in \Lambda_{j+1}[t,s]$, i.e. for\eqref{eqn:r>>6} and \eqref{eqn:r>>5}. Finally the construction and especially \eqref{eqn:concluprocedure} insures us that
\[ \left | \sum_{\lambda_1 \in \Lambda_{j}[
t,s]} \sum_{\lambda_2 \in \Lambda_{j}[t,s]} \varepsilon_{\lambda_1,\lambda_2}(\omega) I_{\lambda_1,\lambda_2} \right|  \leq C (j-n+1) 2^{\frac{-j-n}{2}} \mu.\]
Therefore, the last term in \eqref{demoslow:decom} is bounded from above by
\begin{align*}
& C \mu^2 \left(\sum_{j_1 \geq n} 2^{j_1(1-H_1)} 2^{\frac{1}{m}(j_1-n)} \sum_{j_2 \geq j_1} 2^{-j_2H_2} 2^{ \frac{1}{m}(j_2-n)}  +\sum_{j \geq n} 2^{j(\frac{3}{2}-H_1-H_2)} (j-n+1) 2^{- \frac{n}{2}} \right) \\
& \leq C \mu^2 \left(\sum_{j_1 \geq n} 2^{j_1(1-H_1-H_2)} 2^{\frac{2}{m}(j_1-n)} + 2^{n(\frac{3}{2}-H_1-H_2)}2^{- \frac{n}{2}} \right) \\
& \leq C \mu^2 2^{n(1-H_1-H_2)} \\
& \leq C \mu^2 |t-s|^{H_1+H_2-1}
\end{align*}
and thus inequality \eqref{eqn:Prop:slow} holds.
\end{proof}

\section{Lower bounds for wavelet leaders} \label{sec:lower}

In this section, we show that the limits \eqref{eqn:mainthm1} and \eqref{eqn:mainthm2} are strictly positive. In \cite{ayacheesserkleyntssens}, the authors used the independence of the increments of the Brownian motion to bound from below its wavelet leaders. But, for the (generalized) Rosenblatt process this nice feature is not met anymore. Nevertheless, following an idea by Ayache in a close but different context\footnote{In \cite{ayache20}, Ayache does not consider wavelets at all but directly work on Wiener-It\^o integrals} \cite{ayache20}, we decompose the wavelet coefficients of the generalized Rosenblatt process in two parts. We gain some independence properties in the first part while the second is, in some sense, negligible compared to the first, see Proposition \ref{prop:normel2} below. All along this section, $C$ stands for a deterministic constant whose value may change from a line to another but does not depend on any relevant quantities, and in order to ease notations we set
\[ C_{H_1,H_2} := \frac{1}{\Gamma \left(H_1-\frac{1}{2}\right) \Gamma \left(H_2-\frac{1}{2}\right)}  \]
and for $s,x_1,x_2 \in \R$
\[ f_{H_1,H_2}(s,x_1,x_2) = \lp s-x_1 \rp_+^{H_1-3/2} \lp s-x_2 \rp_+^{H_2-3/2} \]

 Let $\Psi$ be a wavelet with compact support included in $[-N,N]$. Using formula \eqref{eqn:coef} at $t=k/2^j$, the wavelet coefficient $c_{j,k}$ of the generalized Rosenblatt process is given by 
\begin{align*}
    c_{j,k} =&  \, \int_{-N}^N \lc\ros{\frac{x+k}{2^j}} - \ros{\frac{k}{2^j}} \rc\Psi(x) dx \\ =
    & \, c_{H_1,H_2} \int_{-N}^N \Psi(x) \int_{\R^2}'\int_{\frac{k}{2^j}}^{\frac{x+k}{2^j}}  f_{H_1,H_2}(s,x_1,x_2) \, ds \, dB(x_1)\, dB(x_2) \, dx \\= 
    & \,  c_{H_1,H_2} \int_{\R^2}'\int_{-N}^N \Psi(x) \int_{\frac{k}{2^j}}^{\frac{x+k}{2^j}} f_{H_1,H_2}(s,x_1,x_2)  \,ds \, dx\, dB(x_1)\, dB(x_2)  \\= 
    & \, \,  c_{H_1,H_2} \int_{A}'\int_{-N}^N \Psi(x) \int_{\frac{k}{2^j}}^{\frac{x+k}{2^j}} f_{H_1,H_2}(s,x_1,x_2)  \, ds \, dx\, dB(x_1)\, dB(x_2)
\end{align*}
where $A := \left] - \infty, \frac{k+N}{2^j} \right]^2$, because, as soon as $x \in [-N,N]$ and $s \in [k2^{-J},(k+N)2^{-j}]$, $f(s,x_1,x_2)$ vanishes for all $x_1,x_2$ outside of $A$.
\begin{Def}
Given an integer $M \geq 0$, $c_{j,k}$ can be written as following 
\[
    c_{j,k} = \wt{c_{j,k}}^M + \wc{c_{j,k}}^M
\]
where
\begin{align}\label{eqn:coefind}
    \wt{c_{j,k}}^M  = c_{H_1,H_2} \int_{\lambda_{j,k}^{M}}'\int_{-N}^N \Psi(x) \int_{\frac{k}{2^j}}^{\frac{x+k}{2^j}} f_{H_1,H_2}(s,x_1,x_2) \, ds \, dx\, dB(x_1)\, dB(x_2)
\end{align}
with \[\lambda_{j,k}^{M} :=\left] \frac{k-NM}{2^j}, \frac{k+N}{2^j} \right]^2\] and 
\[
    \wc{c_{j,k}}^M  = c_{H_1,H_2} \int_{A \setminus \lambda_{j,k}^{M}}'\int_{-N}^N \Psi(x) \int_{\frac{k}{2^j}}^{\frac{x+k}{2^j}}f_{H_1,H_2}(s,x_1,x_2) \,  ds \, dx\, dB(x_1)\, dB(x_2).
\]
\end{Def}
\begin{Rmk}\label{rmk:iid}
Let us highlight the fact that using time change of variable for Wiener-It\^o integrals \cite[Theorem 8.5.7]{oksendal2013stochastic}, for all $j,k$, we have $\wt{c_{j,k}}^M$ is equal in law to the random variable
\[ c_{H_1,H_2}2^{-j (H_1+H_2-1)} \int_{I_{M}}' \int_{-N}^{N} \psi(x) \int_0^x f_{H_1,H_2}(s,x_1,x_2) \, ds dx \, dB(x_1)dB(x_2)\]
with $I_{M}=(-MN,N]^2$, while $\wc{c_{j,k}}^M$ is equal in law to the random variable
\[ c_{H_1,H_2}2^{-j (H_1+H_2-1)} \int_{I_{M}'}' \int_{-N}^{N} \psi(x) \int_0^x f_{H_1,H_2}(s,x_1,x_2) \, ds dx \, dB(x_1)dB(x_2)\]
with $I'_{M}=(-\infty,N]^2 \setminus (-MN,N]^2$.
\end{Rmk}
\begin{Def}
For all $(j,k) \in \N \times \Z$ and $M \in \N$ we define the random variables  
\begin{align*}
\wt{\ep_{j,k}}^M := \dfrac{\wt{c_{j,k}}^M}{2^{-j(H_1+H_2-1)}}\text{ and } \wc{\ep_{j,k}}^M := \dfrac{\wc{c_{j,k}}^M}{2^{-j(H_1+H_2-1)}}.
\end{align*}
\end{Def}
\begin{Rmk}
Note that $\wt{\ep_{j,k}}^M $ and $\wt{\ep_{j',k'}}^M $ are independent when 
\begin{align} \label{eqn:condindep}
    \lambda_{j,k}^{M} \cap \lambda_{j',k'}^{M}= \cmemptyset.
\end{align}
Indeed, if $(f_j)_j$ is a sequence of real-valued step functions on $\R^2 \setminus \{(x,x) \, : \, x \in \R \}$ which converge to the integrand with respect to $dB(x_1)dB(x_2)$ in \eqref{eqn:coefind} then $\int_{\R^2}' f_j(x_1,x_2) \, dB(x_1)dB(x_2)$ is a polynomial function of a finite number of increments $B(t_2)-B(t_1)$ of the Brownian motion for some $t_1,t_2 \in \lambda_{j,k}^{M}$. Thus $\wt{\ep_{j,k}}^M$ is measurable with respect to the $\sigma$-algebra generated by these increments
\[ \sigma_{j,k}^{M} :=\sigma \left(\{ B(t_2)-B(t_1) \, : \,  t_1,t_2 \in \lambda_{j,k}^{M} \} \right). \]
Using the independence of the increments of the Brownian motion, one concludes that $\sigma_{j,k}^{M}$ and $\sigma_{j',k'}^{M}$ are independent as soon as condition \eqref{eqn:condindep} is met and so the same holds for $\wt{\ep_{j,k}}^M $ and $\wt{\ep_{j,k'}}^M $.
Moreover, $\wt{\ep_{j_1,k_1}}^M, \dots, \wt{\ep_{j_n,k_n}}^M$ are independent when the following condition is satisfied 
\begin{align}
\label{cond:indep}
     \ \lambda_{j_i,k_i}^{M} \cap  \lambda_{j_l,k_l}^{M} = \cmemptyset \mbox{ for all } 1 \leq i < l \leq n.
    \end{align}
\end{Rmk}
This leads to defining the following condition.    
\begin{Def}\label{def:conditionind}
Let $n \geq 2$. We say $\lambda_{j_1,k_1}, \dots, \lambda_{j_n,k_n}$ satisfy condition $(C_M)$ if \eqref{cond:indep} is satisfied. 
\end{Def}
From Remark \ref{rmk:iid}, we know that $(\wt{\ep_{j,k}}^M)_{\lambda \in \Lambda}$ is a family of identically distributed second order Wiener chaos random variables. Moreover, $\wt{\ep_{j_1,k_1}}^M, \dots, \wt{\ep_{j_n,k_n}}^M$ are independent as soon as $\lambda_{j_1,k_1}, \dots, \lambda_{j_n,k_n}$ satisfies $(C_M)$. The following proposition provides a lower bound (independent of $M$) for the tail behavior of the random variable $\wt{\ep_{j,k}}^M$. 
\begin{Prop}
Let $M \in \N$ and $y \in \R^+$. If $M$ and $y$ are large enough, then the exists a deterministic constant $c_2>0$ (independent of $M$) such that 
\begin{equation}\label{eqn:constantM}
\mathbb{P} \lp |\wt{\ep_{j,k}}^M| > y \rp \geq \exp{\lp- c_2 y\rp}
\end{equation}
for all $(j,k) \in \N \times \Z$
\end{Prop}
\begin{proof}
Fix $y \in \R^+$ (large enough). Our aim is to prove the existence of lower bound for $\mathbb{P} \lp |\wt{\ep_{\lambda}}^M| > y \rp$ which is independent of $M$. To this end, we start by proving the following lemma
\begin{Lemma} \label{prop:normel2}
There exist three strictly positive deterministic constants $C_{\Psi,H_1,H_2}$, $C'_{\Psi,H_1,H_2}$ and $C^*_{\Psi,H_1,H_2}$ such that for all $(j,k) \in \N \times \Z$ and $M \geq 2$ one has 
\begin{align*}
 C_{\Psi,H_1,H_2} 2^{-j(H_1+H_2-1)} \leq & \norm{\wt{c_{j,k}}^M}\leq C'_{\Psi,H_1,H_2} 2^{-j(H_1+H_2-1)}  \\
    & \norm{\wc{c_{j,k}}^M} \leq C^*_{\Psi,H_1,H_2} 2^{-j(H_1+H_2-1)} M^{\max\{H_1,H_2\}-1} 
\end{align*}
\end{Lemma}
\begin{proof}
Let us assume, w.l.o.g. that $H_1 \geq H_2$. We define the functions
\[ \Phi_1 \, : \, (x_1,x_2) \mapsto  \int_{-N}^N \Psi(x) \int_{0}^{x}  f_{H_1,H_2}(s,x_1,x_2)ds \, dx, \] 
\[ \Phi_2 \, : \, (x_1,x_2) \mapsto  \int_{-N}^N \Psi(x) \int_{0}^{x}  f_{H_1,H_2}(s,x_2,x_1) ds \, dx, \] 
and the symmetric function\footnote{The function $\Phi$ is in the fact the symmetrization of $\Phi_1$.}
\[ \Phi = \frac{1}{2} \left( \Phi_1 + \Phi_2 \right).\]
By Remark \ref{rmk:iid} we have, using the ``Wiener isometry'' \footnote{For $f$ a symmetric function in $L^2(\R^2)$ , and $I_2(f)$ the second order Wiener-It\^o integral of $f$. One has $\E(I_m(f))^2= 2! ||f||_{L^2(\R^2)}$.} \cite[Section 5]{taqqu-2011},
\begin{align*}
\norm{\wt{c_{j,k}}^M} =   \sqrt{2} c_{H_1,H_2} 2^{-j(H_1+H_2-1)} \left \| \Phi \right \|_{L^2(I_M)}
\end{align*}
and thus it suffices to take  
\begin{align*}
    & C_{\Psi,H_1,H_2} := \sqrt{2} c_{H_1,H_2}  \left \| \Phi \right \|_{L^2([-N,N]^2)} \\ & C'_{\Psi,H_1,H_2} := \sqrt{2} c_{H_1,H_2}  \left \| \Phi \right \|_{L^2((-\infty,N]^2)}
\end{align*}
Now, still using Remark \ref{rmk:iid} and ``Wiener isometry'' we have
\begin{align*}
    \norm{\wc{c_{j,k}}^M} &=  \sqrt{2} c_{H_1,H_2} 2^{-j(H_1+H_2-1)} \left \| \Phi \right \|_{L^2(I_M')} \\
    & \leq \sqrt{2} c_{H_1,H_2} 2^{-j(H_1+H_2-1)} \left \| \Phi_1 \right \|_{L^2(I_M')}. 
\end{align*}
Also as 
\begin{align*}
    I'_{M}=(-\infty,N]^2 \setminus (-MN,N]^2 \subset \R \times \left( -\infty,-MN\right] \bigcup \left( -\infty,-MN\right] \times \R,
\end{align*}
we write
\begin{align*}
&\left \| \Phi_1 \right \|_{L^2(I_M')}^2  \\&= \int_{I_M'} \left| \int_{-N}^N \Psi(x) \int_0^x (s-x_1)_+^{H_1-\frac{3}{2}} (s-x_2)_+^{H_2-\frac{3}{2}} ds \, dx \right |^2 dx_{1} \, dx_2 \\
& \leq \int_{I_M'} \left( \int_{-N}^N |\Psi(x)| \int_{[0,x]} (s-x_1)_+^{H_1-\frac{3}{2}} (s-x_2)_+^{H_2-\frac{3}{2}} ds \, dx \right )^2 dx_{1} \, dx_2 \\
& \leq \int_\R \int_{- \infty}^{-MN}  \left( \int_{-N}^N |\Psi(x)| \int_{[0,x]} (s-x_1)_+^{H_1-\frac{3}{2}} (s-x_2)_+^{H_2-\frac{3}{2}} ds \, dx \right )^2 dx_{1} \, dx_2 \\
& +  \int_\R \int_{- \infty}^{-MN}  \left( \int_{-N}^N |\Psi(x)| \int_{[0,x]} (s-x_1)_+^{H_1-\frac{3}{2}} (s-x_2)_+^{H_2-\frac{3}{2}} ds \, dx \right )^2 dx_{2} \, dx_1.
\end{align*}
Let us deal with the first term in the last sum, the second one can be treated similarly by permuting the roles of $H_1$ and $H_2$ as well as $x_1$ and $x_2$. As the function $y \mapsto  y^{H_1-3/2}$ is decreasing, one gets 
\begin{align*}
&\int_\R  \int_{- \infty}^{-MN}  \left( \int_{-N}^N |\Psi(x)| \int_{[0,x]} (s-x_1)_+^{H_1-\frac{3}{2}} (s-x_2)_+^{H_2-\frac{3}{2}} ds \, dx \right )^2 dx_{1} \, dx_2 \\
& \leq \left(  \int_{- \infty}^{-MN} (-N-x_1)^{2H_1-3} \, dx_1 \right) \times \int_\R \lp \int_{-N}^N | \Psi(x)| \lln \int_{[0,x]} \lp s-x_2 \rp_+^{H_2-3/2} ds \rrn dx \rp^2 dx_2.
\end{align*}
Concerning the first integral, we have, as $M \geq 2$
\begin{align*}
    \int_{-\infty}^{ -NM} \lp -N-x_1 \rp^{2H_1-3} dx_1 &= \dfrac{1}{2-2H_1}  \lp NM-N \rp_+^{2H_1-2} \\
    & = \dfrac{1}{2-2H_1} N^{2H_1-2}  (M-1) ^{2H_1-2} \\
    & \leq  c \cdot  M^{2H_1-2} 
\end{align*}
while, using again the ``Wiener isometry'', 
\begin{align*}
    \int_\R &\lp \int_{-N}^N | \Psi(x)| \lln \int_{[0,x]} \lp s-x_2 \rp_+^{H_2-3/2} \rrn dx \rp^2 dx_2  \\ \leq 
    & \, 2N \norminf{\Psi} \sup_{x \in [-N,N]} \int_\R  \lln \int_{[0,x]} \lp s-x_2 \rp_+^{H_2-3/2} ds \rrn^2 dx_2 \\ =
    & \, 2N \norminf{\Psi} \sup_{x \in [-N,N]} \E \lc \lln B_{H_2} \lp x \rp - B_{H_2} \lp 0 \rp \rrn^2 \rc \\ \leq 
    & \, 2N \norminf{\Psi} \sup_{x \in [-N,N]} C_{H_2} \lp |x|\rp^{2H_2} \leq c, 
\end{align*}
where $B_{H_2}$ denotes the fractional Brownian motion with parameter $H_2$. As a result, there exists a positive constant $C^*_{\Psi,H_1,H_2}$ such that, as we suppose $H_1 \geq H_2$, one has 
\begin{align*}
    \norm{\wc{c_{j,k}}^M} \leq C^*_{\Psi,H_1,H_2} 2^{-j(H_1+H_2-1)} M^{H_1-1}.
\end{align*}
\end{proof} 
By Lemma \ref{prop:normel2}, one can remark that as $M \rightarrow +\infty$, $(\wt{\ep_{j,k}}^M )_M$ converges in $L^2(\Omega)$ to the random variable
   \[ \ep_{j,k} := \dfrac{c_{j,k}}{2^{-j(H_1+H_2-1)}} \]
   with, for all $M \in \N$,
   \[  \ep_{j,k}-\wt{\ep_{j,k}}^M =  \wc{\ep_{j,k}}^M . \]
By Theorem \ref{thm:bound2}, there exists a constant $c_1>0$ such that, for all $\lambda \in \Lambda$ and $y$ sufficiently large
\begin{align*}
 \mathbb{P} \lp \lvert \ep_{j,k} \rvert \geq y \rp \geq \exp{\lp- c_1 y\rp}.
\end{align*}
Then, for all $M \in \N$, we have, for all such $\lambda$ and $y$ 
\begin{align*}
 \mathbb{P} \lp \lvert \wt{\ep_{j,k}}^M\rvert \geq y \rp & \geq  \mathbb{P} \lp \{\lvert \wt{\ep_{j,k}}^M\rvert \geq y \} \cap \{|\wc{\ep_{j,k}}^M | \leq y \} \rp \\
 & \geq \mathbb{P} \lp \{ \lvert \ep_{j,k}\rvert- |\wc{\ep_{j,k}}^M |\geq y \} \cap \{ |\wc{\ep_{j,k}}^M| \leq y \} \rp \\
 & \geq \mathbb{P} \lp \{ \lvert \ep_{j,k}\rvert\geq 2 y \} \cap \{ |\wc{\ep_{j,k}}^M| \leq y \} \rp \\
 & \geq \mathbb{P} \lp \lvert \ep_{j,k}\rvert\geq 2 y  \rp -\mathbb{P} \lp (|\wc{\ep_{j,k}}^M| > y  \rp.
\end{align*}
Using Lemma \ref{prop:normel2} and Theorem \ref{thm:bound} one has
\begin{align*}
\mathbb{P} \lp |\wc{\ep_{j,k}}^M| > y \rp &\leq \mathbb{P} \lp (|\wc{c_{j,k}}^M| > y \norm{\wc{c_{j,k}}^M} (C^*_{\Psi,H_1,H_2})^{-1} M^{1-\max\{H_1,H_2\}}  \rp \\
& \leq \exp(- \overset{\star}{C} (C^*_{\Psi,H_1,H_2})^{-1} M^{1-\max\{H_1,H_2\}} y).
\end{align*}
Thus, if $M$ is large enough, one has, as $1-\max\{H_1,H_2\}>0$,
\[ \exp(- \overset{\star}{C} (C^*_{\Psi,H_1,H_2})^{-1} M^{1-\max\{H_1,H_2\}} y) \leq \frac{1}{2}  \exp{\lp- 2c_1 y\rp}\]
which gives that, for all large enough $y$, one gets
\begin{equation}\label{eqn:constantM}
\mathbb{P} \lp |\wt{\ep_{j,k}}^M| > y \rp \geq \exp{\lp- c_2 y\rp}
\end{equation}
with $c_2:= 2 c_1$. In the sequel, we will implicitly always consider such large enough $M$.
\end{proof}

In the following two subsections, Lemmata \ref{Lemma:leaderord1} and \ref{Lemma:rapid1} follow the lines of Lemmata 3.6 and 3.8 in \cite{ayacheesserkleyntssens} respectively, with some subtle modifications as the authors in \cite{ayacheesserkleyntssens} deal with $\mathcal{N}(0,1)$ random variables while, here, we focus on random variables in the second order Wiener chaos that depend on the parameter $M$. For the sake of completeness and clarity, we write the proofs in full details.

\subsection{Ordinary Points}
In this section our aim is to prove the following proposition.
\begin{Prop}
\label{Prop:ordinleader}
There exists $\Omega^*_1 \subset \Omega$ with probability 1 such that for all $\omega \in \Omega^*_1$ and Lebesgue almost every $t \in (0,1)$ one has 
\begin{align}
\label{eq:Propslow}
    \limsup_{j \rightarrow +\infty}  \dfrac{d_j(t,\omega)}{2^{-j(H_1+H_2-1)} \log j} >0.
\end{align}
\end{Prop}
To this end, as  a first step, let us state the following lemma concerning the random variable $\wt{\ep_{\lambda}}^M$. If $\lambda=\lambda_{j,k}$ is a dyadic interval and $m \in \N$, $S_{\lambda,m}=S_{j,k,m}$ stands for the finite set of cardinality $2^m$ whose elements are the dyadic intervals of scale $j + m$ included in $\lambda_{j,k}$, formally speaking 
\begin{align*}
S_{j,k,m} := \{ \lambda \in \Lambda_{j+m} \, : \, \lambda \subset \lambda_{j,k} \}
\end{align*}

\begin{Lemma}
\label{Lemma:leaderord1}
There is a deterministic constant $C>0$ such that the following holds: for all $M \in \N$ and for all $t \in (0,1)$, there exists $\Omega_{t,1} \subset \Omega$ with probability 1 such that for all $\omega \in \Omega_{t,1}$ there are infinitely many $j \in \N$ such that
\begin{align*}
    \max_{\begin{matrix}
\lambda' \in S_{\lambda,\lfloor \log_2 (NM) \rfloor +2}\\
\lambda \in 3\lambda_j(t)
\end{matrix}} \lln \wt{\ep_{\lambda'}}^M(\omega) \rrn \geq C \log j.
\end{align*}
\end{Lemma}

\begin{proof}
Let us fix $t \in (0,1)$ and $j \in \N $. For any $\lambda \in S_{j,k_j(t),m}$, there exists a unique decreasing finite sequence $(I_n)_{0 \leq n \leq m}$ of decreasing dyadic intervals in the sense of inclusion such that $I_0= \lambda_{j,k_j(t)}$, $I_m=\lambda$ and $I_n \in S_{j,k_j(t),n}$. Then, define the sequence  $(T_n)_{1 \leq n \leq m}$ of unique dyadic intervals such that $I_{n-1} = I_n \cup T_n$. Note that for all $1 \leq n \leq m$, $T_n \in 3I_n$. Moreover, as $(I_n)_{0 \leq n \leq m}$ is decreasing, $(T_n)_{1 \leq n \leq m}$ are pairwisely disjoint. Furthermore, for every $n \in \{1,...,m\}$, there exist $T'_n = \lambda_{j_n,k_n} \in  S_{T_n,\lfloor \log_2 NM \rfloor +2}$ such that 
\begin{align*}
    \lp\dfrac{k_n-NM}{2^{j_n}}, \dfrac{k_n+N}{2^{j_n}} \rp \subset T_n.
\end{align*}
As a consequence, the associated random variables $\lp \wt{\ep_{T'_n}}^M \rp_{1 \leq n \leq m}$ are independent as the dyadic intervals $(T'_n)_{1 \leq n \leq m}$ satisfies condition $(C_M)$ in Definition \ref{def:conditionind}. Next, for a constant $C >0$ to be chosen later, we set 
\begin{align*}
    \mathcal{E}_{j,m}(t)= \lcl \omega \in \Omega \, : \, \max_{1 \leq n \leq m} \lln \wt{\ep_{T'_n}}^M\rrn \geq C \log(2m) \rcl.
\end{align*}
Note that, as the random variables $\lp \wt{\ep_{T'_n}}^M \rp_{1 \leq n \leq m}$ are independent, 
\begin{align*}
    \mathbb{P} \lp \mathcal{E}_{j,m}(t)\rp = 1 - \prod_{n=1}^m \mathbb{P} \lp \lln \wt{\ep_{T'_n}}^M \rrn < C \log(2m) \rp 
\end{align*}
Recalling \eqref{eqn:constantM}, and the fact that $\log(1-x) \leq -x$ if $x \in (0,1)$, one gets, for $m$ is large enough, 
\begin{align*}
    \mathbb{P} \lp \mathcal{E}_{j,m}(t)\rp \geq & \, 1 - (1- \exp(-Cc_2 \log(2m))^m\\ = 
     & \, 1 - \lp 1 - \lp \dfrac{1}{2m} \rp ^{Cc_2} \rp^m \\ \geq 
     & \, 1 - \exp \lp \dfrac{m}{(2m)^{Cc_2}} \rp \\ = 
     & \, 1 - \exp \lp \dfrac{m^{1-Cc_2}}{2^{Cc_2}} \rp.
     \end{align*}

Finally, choosing $C$ such that $0<Cc_2<1$, one obtain that
\begin{align*}
    \sum_{p \in \N} \mathbb{P} \lp \mathcal{E}_{2^p,2^p}(t)\rp = + \infty.
\end{align*}
Knowing that the events $\mathcal{E}_{2^p,2^p}(t)$ are independent for all $p \in \N$, one concludes using Borel-Cantelli Lemma that 
\begin{align*}
    \mathbb{P} \lp  \limsup_{m \rightarrow +\infty} \mathcal{E}_{2^m,2^m}(t)\rp =1
\end{align*}
It follows that for a fixed $t \in \R$, almost surely, there are infinitely many $j \in \N$ such that
\begin{align*}
    \max_{\begin{matrix}
\lambda' \in S_{\lambda,\lfloor \log_2 NM \rfloor +2}\\
\lambda \in 3\lambda_j(t)
\end{matrix}} \lln \wt{\ep_{\lambda'}}^M(\omega) \rrn \geq C \log j.
\end{align*}
\end{proof}

Concerning the ``non-independent part'' of the wavelet coefficients, one can state the following Lemma.

\begin{Lemma}\label{Lemma:leaderord2}
There is a deterministic constant $C'>0$ such that, for all $M \in \N$ and for all $t \in (0,1)$,  there exists $\Omega_{t,2} \subset \Omega$ with probability 1 such that for all $\omega \in \Omega_{t,2}$ there exists $J \in \N$ such that, for all $j \geq J$,
\begin{align*}
    \max_{\begin{matrix}
\lambda' \in S_{\lambda,\lfloor \log_2( NM) \rfloor +2}\\
\lambda \in 3\lambda_j(t)
\end{matrix}} \lln \wc{\ep_{\lambda'}}^M(\omega) \rrn \leq C'  M^{\max\{H_1,H_2\}-1} \log j.
\end{align*}
\end{Lemma}

\begin{proof}
Let us fix $t \in (0,1)$. For any $C'>0$, for all $j$ sufficiently large and $\lambda \in 3\lambda_j(t)$, we have, by Theorem \ref{thm:bound},
\begin{align*}
\mathbb{P}&\left(\exists \lambda' \in  S_{\lambda,\lfloor \log_2 NM \rfloor +2} \, : \,  \lln \wc{\ep_{\lambda'}}^M \rrn \geq  C'  M^{\max\{H_1,H_2\}-1}\log j \right) \\
& \leq \sum_{\lambda' \in  S_{\lambda,\lfloor \log_2 NM \rfloor +2}} \mathbb{P} \left(\lln \wc{\ep_{\lambda'}}^M) \rrn \geq  C'  M^{\max\{H_1,H_2\}-1}\log j  \right) \\& \leq \sum_{\lambda' \in  S_{\lambda,\lfloor \log_2 NM \rfloor +2}} \mathbb{P} \left(\lln \wc{\varepsilon_{\lambda'}}^M \rrn \geq  C' (C^*_{\Psi,H_1,H_2})^{-1} \|\wc{\varepsilon_{\lambda'}}^M  \|_{L^2(\Omega)} \log j  \right) \\
& \leq 4NM \exp(-\overset{\star}{C} C' (C^*_{\Psi,H_1,H_2})^{-1}  \log j)
\end{align*}
Thus, for $C' > C^*_{\Psi,H_1,H_2}/\overset{\star}{C}$, the conclusion follows by Borel-Cantelli Lemma.
\end{proof}

\begin{proof}[Proof of Proposition \ref{Prop:ordinleader}]
The constant $C$ and $C'$ of Lemmata \ref{Lemma:leaderord1} and \ref{Lemma:leaderord2} being deterministic and independent of $M$, on can choose $M$ large enough such that 
\[ C-C' M^{\max\{H_1,H_2\}-1} >0.\]
Let us fix $t \in (0,1)$ and consider $\omega \in \Omega_{t,1} \cap \Omega_{t,2}$, where the events, of probability $1$, $\Omega_{t,1}$ and $\Omega_{t,2}$ are given by the same Lemmata. For all $J \in \N$, by Lemma \ref{Lemma:leaderord1}, there exist $ j \geq J$ and $\lambda'(j) \subseteq 3\lambda_j(t)$ of scale $j'= j + \lfloor \log NM \rfloor +2$ such that 
\begin{align*}
    \lln \wt{c_{\lambda'(j)}}^M(\omega)\rrn \geq C 2^{-j'(H_1+H_2-1)}\log j.
\end{align*}
If $J$ is large enough, we also have, for all such $j \geq J$, by Lemma \ref{Lemma:leaderord2},
\begin{align*}
    \lln \wc{c_{\lambda'(j)}}^M(\omega)\rrn \leq C' M^{\max\{H_1,H_2\}-1} 2^{-j'(H_1+H_2-1)} \log j .
\end{align*}
From this we deduce that
\begin{align*}
     d_j(t,\omega)  \geq &
     \lln c_{\lambda'(j)}(\omega) \rrn \\ \geq 
    & \lln \wt{c_{\lambda'(j)}}^M(\omega)\rrn  - \lln \wc{c_{\lambda'(j)}}^M(\omega)\rrn  \\ \geq 
    & 2^{-j'(H_1+H_2-1)} \log j \lp C - C'  M^{\max\{H_1,H_2\}-1}  \rp \\
    \geq & 2^{-j(H_1+H_2-1)} (4NM)^{1-H_1-H_2} \log j \lp C - C'  M^{\max\{H_1,H_2\}-1}  \rp
\end{align*}
Therefore, \eqref{eq:Propslow} holds true for all $t \in (0,1)$ and $\omega \in \Omega_{t,1} \cap \Omega_{t,2}$. The conclusion follows then from Fubini Theorem.
\end{proof}

\subsection{Rapid Points}
In this section our aim is to prove the following proposition.
\begin{Prop}
\label{Prop:rapid}
There exists $\Omega^*_2 \subset \Omega$ with probability 1 such that, for all $\omega \in \Omega^*_2$, there exist $t \in (0,1)$ such that 
\begin{align}
\label{eq:Proprapid}
    \limsup_{j \rightarrow +\infty}  \dfrac{d_j(t,\omega)}{2^{-j(H_1+H_2-1)} j} >0.
\end{align}
\end{Prop}
 As in the previous subsection, we start by working with the random variables $\wt{\ep_{\lambda}}^M$.
\begin{Lemma}
\label{Lemma:rapid1}
There exists a deterministic constant $C>0$ such that for all $M$ there is $\Omega_2 \subset \Omega$ with probability 1 such that for all $\omega \in \Omega_2$ there exist $t \in (0,1)$ such that 
\begin{align}
\label{eq:rapidpoints}
    \limsup_{j \rightarrow +\infty} \dfrac{\lln \wt{\ep_{\lambda_j(t)}}^M(\omega) \rrn}{j} \geq C. 
\end{align}
\end{Lemma}
\begin{proof}
Let us fix $a \in (0, 1) $ and $C > 0$ to be chosen later on. For every $(j, l) \in \N \times \left\{ 0,\dots, \lfloor 2^{j(1-a)} \rfloor -1\right\}$, we set 
\[ S_{j,l}^M = \lcl l \lfloor 2^{aj}/(2NM) \rfloor , \dots, (l+1) \lfloor 2^{aj}/(2NM) \rfloor-1 \rcl\]
and consider the event
\begin{align*}
    \mathcal{E}_{j,l}^M = \lcl \omega \in \Omega \, : \, \max_{k \in S_{j,l}^M } \lln \wt{\ep_{j,2kNM}}^M(\omega) \rrn \geq Cj\rcl
\end{align*}
Let $j_0$ be the smallest integer such that $\lfloor 2^{aj}/(2NM) \rfloor \geq 1$. If we assume that 
\begin{align}
\label{eq:omega*2}
    \Omega^*_2 = \bigcup_{J \geq j_0} \bigcap_{j \geq J} \bigcap_{l \in \left\{ 0,\dots, \lfloor 2^{j(1-a)} \rfloor -1\right\} } \mathcal{E}_{j,l}^M
\end{align}
is an event of probability 1 and we consider $\omega \in \Omega^*_2$. For every $j \geq j_0$, denote by 
\begin{align}
\label{eq:Gj(omega)}
    G_j^M(\omega) := \lp k \in \{0, \dots, 2^j -1\} \, : \, \lln \wt{\ep_{j,k}}^M(\omega) \rrn \geq Cj \rp.
\end{align}
Moreover, for every $n \geq j_0$, one considers
\begin{align}
\label{eq:On(omega)}
    O_n^M(\omega) := \bigcup_{j \geq n} U_j^M(\omega), \quad \mbox{where } U_j^M(\omega) := \bigcup_{k \in G_j^M(\omega)} \lp \frac{k}{2^j}, \frac{k+1}{2^j} \rp.
\end{align}
If one proves that $O_n^M(\omega)$ is dense in $(0,1)$, then by Baire's theorem the set $\cap_{n \geq j_0} O_n^M(\omega)$ is non-empty and let t be an element of this set. Then for every $n \geq j_0$, there is $j \geq n$ such that $\lln \wt{\ep_{\lambda_j(t)}}^M(\omega) \rrn \geq Cj$, and so desired statement \eqref{eq:rapidpoints} is true.
\\ We still have to prove two points: 
\begin{enumerate}
    \item $O_n^M(\omega)$ is dense in $(0, 1)$.
    \item $\Omega_2^*$ is an event of probability 1.
\end{enumerate}
Indeed, starting with statement 1, consider $t \in (0, 1)$, $j \geq j_0$ and $k$ such that $\lambda_j(t) = \lambda_{j,k}$. Then, we have two cases: 
\begin{enumerate}
    \item[] Case 1 : There is $l \in \{0, \dots, \lfloor 2^{j(1-a)} \rfloor -1\}$ such that 
    \begin{align*}
        k \in \lcl l \lfloor 2^{aj} \rfloor , \dots, (l+1) \lfloor 2^{aj} \rfloor-1 \rcl
    \end{align*}
     Using \eqref{eq:omega*2} and \eqref{eq:Gj(omega)}, there is $k' \in \{ l \lfloor 2^{aj}/(2NM) \rfloor , \dots, (l+1) \lfloor 2^{aj}/(2NM) \rfloor-1 \}$ such that $2k'NM \in G_j(\omega)$. Then, by \eqref{eq:On(omega)}, 
     \begin{align*}
         \lp \dfrac{2NMk'}{2^j}, \dfrac{2NMk'+1}{2^j} \rp \subset O_n^M(\omega).
     \end{align*}
     which is at is at most $2^{-j} \lp\lfloor 2^{aj} \rfloor +2NM\lfloor 2^{aj}/(2NM) \rfloor \rp$ from $t$. Finally, we get that $t$ is at a distance at most $22^{j(a-1)}$ of $U_j^M(\omega)$.
    \item[] Case 2 : $k \in \{ \lfloor 2^{j(1-a)} \rfloor \lfloor 2^{ja} \rfloor,\dots, 2^j-1\}$. Again by \eqref{eq:omega*2} and \eqref{eq:Gj(omega)}, there is $k' \in S_{j,l}^M$ such that $2k'NM \in G_j^M(\omega)$, and similarly, we get that $t$ is at a distance at most $c 2^{j(a-1)}$ of $U_j^M(\omega)$, for some constant $c > 0$ depending only on $N$, $M$ and $a$. 
\end{enumerate}
Finally, in both cases $t$ is at a distance at most $c 2^{j(a-1)}$, and so the density follows. 
\\ Now for statement 2, in order to prove that $\Omega^*_2$ has a probability 1, it is enough to prove that 
\begin{align}
    \label{eq:event}
    \PP \lp \mathcal{C} \lp \bigcap_{l \in \left\{ 0,\dots, \lfloor 2^{j(1-a)} \rfloor -1\right\} } \mathcal{E}_{j,l}^M\rp\rp
\end{align}
is the general term of a convergent series, then the result follows by Borel-Cantelli Lemma. Note that the variables $\wt{\ep_{j,2NMk}}^M(\omega)$, $k \in  S_{j,l}^M$ and $l \in \{0, \ldots , \lfloor2^{j(1-a)}\rfloor - 1\}$, are independent because for every $k \neq k'$, $|2NMk-2NMk'| \geq 2NM$ and so $\lambda_{j,2MNk}^M \cap \lambda_{j,2MNk'}^M = \cmemptyset$. Consequently, one has
\begin{align}\label{prop:utilepourconvergence}
\nonumber
   & \PP \lp \mathcal{C} \lp \bigcap_{l \in \left\{ 0,\dots, \lfloor 2^{j(1-a)} \rfloor -1\right\} } \mathcal{E}_{j,l}^M\rp\rp \\ = \nonumber
   & 1- \PP \lp \bigcap_{l \in \left\{ 0,\dots, \lfloor 2^{j(1-a)} \rfloor -1\right\} } \mathcal{C} \lp \mathcal{E}_{j,l}^M\rp\rp  \\ = \nonumber
    & 1 - \prod_{l \in \left\{ 0,\dots, \lfloor 2^{j(1-a)} \rfloor -1\right\}} \lp 1 - \prod_{k \in  S_{j,l}^M} \PP \lp \lln \wt{\ep_{j,2NMk}}^M(\omega) \rrn < Cj\rp \rp \\ = \nonumber
    & 1 - \lp 1 - \lp1- \PP \lp |\ep| \geq Cj\rp \rp^{ \lfloor 2^{aj}/(2NM) \rfloor}  \rp^{ \lfloor 2^{j(1-a)}\rfloor} \\ \leq 
    & 1- \exp{\lp 2^{j(1-a)} \log (1-p_j)\rp}
\end{align}
where $\ep$ is a random variable belonging to the Wiener chaos of order 2 distributed according to the $(\wt{\ep_{\lambda}})_{\lambda \in \Lambda}$  and $p_j = \lp1- \PP \lp |\ep| \geq Cj\rp \rp^{ \lfloor 2^{aj}/(2NM) \rfloor}$. Remark that $p_j$ is a positive term that tends to $0$ as $j \rightarrow +\infty$. Indeed, using the fact that $\log(1-x) \leq -x$ if $x \in (0,1)$ together with \eqref{eqn:constantM}, there exist $J \in \N$ such that for all $j \geq J$, 
\begin{align}
\label{eq:p_j}
    0 \leq p_j \leq \nonumber &\lp 1 -\exp{(-C\,c_2 \,j) }\rp^{ \left\lfloor 2^{aj}/(2NM) \right\rfloor} \\\leq \nonumber
    &\exp{\lp -\left\lfloor \frac{2^{aj}}{2NM} \right\rfloor  \exp{(-C\,c_2 \,j) }\rp} \\ \leq \nonumber
    & \exp{\lp -C'\exp{(\log 2^{aj})} \exp{(-C\,c_2 \,j) }\rp} \\ \leq 
    & \exp{\lp -C' \exp{j(a\log2-C\,c_2) }\rp}
\end{align}
where $C'$ depends only on $N$, $M$ and $a$ and $c_2$ is the constant given in \eqref{eqn:constantM}. It is enough to choose $C$ such that $a\log2-C\,c_2 >0$ to deduce that and so $p_j \rightarrow 0$ as $j \rightarrow +\infty$. Similarly, one can get for all $j \geq J$
\begin{align*}
    0 \leq 2^{j(1-a)}p_j \leq \exp{\lp -C' \exp{j(\log2-C\,c_2) }\rp}
\end{align*}
which indeed shows that $2^{j(1-a)}p_j $ tends to $0$ as $j \rightarrow +\infty$. Now, using the fact that $\log(1-x) = -x +o(x)$ and $\exp{(x)} = 1+x+ o(x)$ as $x \rightarrow 0$, together with \eqref{prop:utilepourconvergence} we obtain that for all $\delta >0$ 
\begin{align*}
    \PP \lp \mathcal{C} \lp \bigcap_{l \in \left\{ 0,\dots, \lfloor 2^{j(1-a)} \rfloor -1\right\} } \mathcal{E}_{j,l}^M\rp\rp \leq 2^{j(1-a)} \lp \delta (p_j+ \delta p_j) +p_j +\delta p_j\rp
\end{align*}
for $j$ large enough. Using the upper bound in \eqref{eq:p_j}, one can finally conclude that \eqref{eq:event} is indeed the general term of a convergent series.
\end{proof}

Concerning the random variable $\wc{\ep_{\lambda}}^M$, one can give an almost sure upper bound.

\begin{Lemma}\label{Lemma:rapid2}
There exists a deterministic constant $C'>0$ such that for all $M$ there is  $\Omega_2' \subset \Omega$ with probability 1 such that for all $\omega \in \Omega_2'$ there exist $J \in \N$ such that, for all $j \geq J$, for all $\lambda \in \Lambda_j$, $\lambda \subseteq [0,1]$,
\[ \lln \wc{\ep_{\lambda}}^M(\omega) \rrn \leq C'  M^{\max\{H_1,H_2\}-1} j\]
\end{Lemma}

\begin{proof}
If $C'>0$, for all $j$ sufficiently large, we have, by Theorem \ref{thm:bound}
\begin{align*}
\mathbb{P}&\left(\exists \lambda \in \Lambda_j, \, \lambda \subseteq [0,1] \, : \,  \lln \wc{\ep_{\lambda}}^M(\omega) \rrn \geq  C' M^{\max\{H_1,H_2\}-1}j \right) \\
& \leq \sum_{ \lambda \in \Lambda_j \, \lambda \subseteq [0,1]} \mathbb{P} \left(\lln \wc{\ep_{\lambda}}^M(\omega) \rrn \geq  C' M^{\max\{H_1,H_2\}-1} j  \right) \\
&\leq  2^j \exp(-\overset{\star}{C} C' (C^*_{\Psi,H_1,H_2})^{-1}  j)
\end{align*}
and thus, if $C' > \log(2)C^*_{\Psi,H_1,H_2}/\overset{\star}{C}$, the conclusion follows by Borel-Cantelli Lemma.
\end{proof}

\begin{proof}[Proof of Proposition \ref{Prop:rapid}]
Again, one can choose $M$ large enough such that
\[ C-C' M^{\max\{H_1,H_2\}-1} >0,\]
where $C$ and $C'$ are the constant given by Lemmata \ref{Lemma:rapid1} and \ref{Lemma:rapid2} respectively. Let us consider $ \omega \in \Omega^*_2:= \Omega_2 \cap \Omega_2'$ where the evnets, of probability $1$, $\Omega_2$ and $\Omega_2'$ are giving by the same Lemmata. We use the same notations as in them. First there exist $t \in (0,1)$ such that for all $ J \in \N$ there exist $j \geq n$ such that 
 \begin{align}
    \lln \wt{c_{\lambda_j(t)}}^M(\omega)\rrn \geq C \, j 2^{-j(H_1+H_2-1)}.
\end{align}
Moreover, if $J$ is large enough, for all such $j$ we also have
 \begin{align}
    \lln \wc{c_{\lambda_j(t)}}^M(\omega)\rrn \leq C' \, M^{\max\{H_1,H_2\}-1}  2^{-j(H_1+H_2-1)} j.
\end{align}
In this case, as in \ref{Prop:ordinleader} we have that for all $J$ great enough, there is $j \geq J$ such that
\begin{align*}
     d_j(t,\omega)  \geq  2^{-j(H_1+H_2-1)} j \lp C  - C'  M^{\max\{H_1,H_2\}-1}  \rp 
\end{align*}
and so one can conclude that \eqref{eq:Proprapid} holds true for all $\omega \in \Omega^*_2$.
\end{proof}

\section{Proof of the main Theorem}

Theorem \ref{thm:main} is then a straightforward consequence of Propositions \ref{prop:rapid}, \ref{Prop:ordinary}, \ref{prop:slow}, \ref{Prop:ordinleader} and \ref{Prop:rapid}.

\begin{proof}[Proof of Theorem \ref{thm:main}]
Let us denote by $\Omega_{R}$ the event obtained by taking the intersection of all the events of probability $1$ induced by Propositions \ref{prop:rapid}, \ref{Prop:ordinary}, \ref{prop:slow}, \ref{Prop:ordinleader} and \ref{Prop:rapid}.

If we consider $\omega$ belonging to this event of probability $1$, first, from Proposition \ref{prop:rapid} there exists $C_R>0$ such that, for all $t,s \in (0,1)$
\begin{equation}\label{eqn:lastdemo}
|R_{H_1,H_2}(t,\omega)-R_{H_1,H_2}(s,\omega)| \leq  C_R |t-s|^{H_1+H_2-1} \log |t-s|^{-1}
\end{equation}
while, for almost every $t_o \in (0,1)$, from Propositions \ref{Prop:ordinary} and \ref{Prop:ordinleader}
\[0< \limsup_{s \to t_o} \dfrac{|R_{H_1,H_2}(t_o,\omega)-R_{H_1,H_2}(s,\omega) |}{|t_o-s|^{H_1+H_2-1} \log \log |t_o-s|^{-1}} < + \infty. \]
Nevertheless, from Proposition \ref{Prop:rapid} we also know that there exists $t_r \in (0,1)$ such that
\[ 0< \limsup_{s \to t_r} \dfrac{|R_{H_1,H_2}(t_r,\omega)-R_{H_1,H_2}(s,\omega) |}{|t_r-s|^{H_1+H_2-1}  \log |t_r-s|^{-1}}\]
which, combined with \eqref{eqn:lastdemo}, gives that, for all such a $t_r$,
\[ 0< \limsup_{s \to t_r} \dfrac{|R_{H_1,H_2}(t_r,\omega)-R_{H_1,H_2}(s,\omega) |}{|t_r-s|^{H_1+H_2-1}  \log |t_r-s|^{-1}} < + \infty.\]
Moreover, from Proposition \ref{prop:slow}, we also know that one can find $t_\sigma \in (0,1)$ such that
\[ \limsup_{s \to t_\sigma} \dfrac{|R_{H_1,H_2}(t_\sigma,\omega)-R_{H_1,H_2}(s,\omega) |}{|t_\sigma-s|^{H_1+H_2-1} } < + \infty.\]
The conclusion follows by Remark \ref{rmk:mainthm}.
\end{proof}

\begin{Rmk}\label{rmk:last}
Unfortunately, our method does not allow us to affirm the positiveness of the limit \eqref{eqn:mainthm3}, at the opposite of limits \eqref{eqn:mainthm1} and \eqref{eqn:mainthm2}. Indeed, as for almost every $\omega \in \Omega$
\[ \limsup_{s \to t} \dfrac{|R_{H_1,H_2}(t,\omega)-R_{H_1,H_2}(s,\omega) |}{|t-s|^{H_1+H_2-1} }\]
is finite for \textit{some} $t$, we would need to show its positiveness for \textit{all} $t$ and thus the positiveness of the limit
\begin{equation}\label{eqn:lastremak}
\limsup_{j \to + \infty} \frac{d_j(t,\omega)}{2^{-j(H_1+H_2-1)}}
\end{equation}
for all $t$.

Concerning the random variables $(\wt{\ep_\lambda}^M)_\lambda$, one can obtain a positive result\footnote{This result is again a generalization of \cite[Lemma 3.3.]{ayacheesserkleyntssens} where most of the modifications comes from the fact that we are working in the Wiener chaos of order 2}. Indeed, from \cite[Theorem 6.9 and Remark 6.10]{janson97} we know that there exists an universal deterministic constant $\gamma \in [0,1)$ such that, for each random variable $X$ in the Wiener chaos of order $2$
\[ \mathbb{P} \left(|X| \leq \frac{1}{2} \norm{X} \right) \leq \gamma. \]
As $0 \leq \gamma <1$, of course, one can find $\ell_0 \in \N$ such that 
\begin{equation}\label{eqn:gamma}
\gamma^{\ell_0} < 2^{-1}.
\end{equation}
 Let us go back to the construction starting the proof of Lemma \ref{Lemma:leaderord1}. If the dyadic interval $\lambda_{j,k}$ and $m \in \N$ are fixed and $S \in \mathcal{S}_{j,k,m}$ we define the sequences of dyadic intervals $(I_n)_{0 \leq n \leq m}$ and $(T_n)_{1 \leq n \leq m}$ in the same way: $I_0=\lambda_{j,k}$, $I_m=S$ and, for all $1 \leq n \leq m$, $I_{n-1}=I_n \cup T_n$. Now, for any $1 \leq n \leq m$, there are $\ell_0$ dyadic intervals $(T_n^\ell=\lambda_{j_{n}^{(\ell)},k_{n}^{(\ell)}})_{1 \leq \ell \leq \ell_0}$ in $S_{T_n,\lfloor \log_2(\ell_0 NM) \rfloor +2}$ such that, for all $ 1 \leq \ell \leq \ell_0$
\[ \left(\frac{k_n^{(\ell)}-NM}{2^{j_n^{(\ell)}}},\frac{k_n^{(\ell)}+N}{2^{j_n^{(\ell)}}} \right) \subseteq T_n \]
and, if $\ell \neq \ell'$, $T_n^\ell \cap T_n^{\ell'} = \cmemptyset$. Therefore, the dyadic intervals $(T_n^\ell)_{1 \leq n \leq m, 1 \leq \ell \leq \ell_0}$ satisfy condition $(C_M)$ in Definition \ref{def:conditionind}. From this, for all $S \in \mathcal{S}_{j,k,m}$ we define the Bernouilli random variable
\[ \mathcal{B}_{j,k,m}(S) = \prod_{1 \leq n \leq m, 1 \leq \ell \leq \ell_0} \mathbf{1}_{\{|\wt{\varepsilon_{T_n^\ell}}^M| < 2^{-1} C_{\Psi,H_1,H_2}\}}\]
for which, by Proposition \ref{prop:normel2}, we have, using the independence of the random variables $(\wt{\varepsilon_{T_n^\ell}}^M)_{1 \leq n \leq m, 1 \leq \ell \leq \ell_0}$,  $\E [\mathcal{B}_{j,k,m}(S)] \leq \gamma^{m \ell_0}$. Therefore, if we define the random variable
\[ \mathcal{G}_{j,k,m} = \sum_{S \in \mathcal{S}_{j,k,m}} \mathcal{B}_{j,k,m}(S)  \]
then $\E[\mathcal{G}_{j,k,m}] \leq (2 \gamma^{\ell_0})^m$ and it follows from inequality \eqref{eqn:gamma} and Fatou Lemma that
\[ \E \lc \liminf_{m \to + \infty} \mathcal{G}_{j,k,m} \rc = 0. \]
As a consequence,
\[ \Omega_1 = \bigcap_{j \in \N, 0 \leq k < 2^{j}} \{ \omega \, : \, \liminf_{m \to + \infty} \mathcal{G}_{j,k,m}(\omega)=0 \} \]
is an event of probability $1$.

Now if $\omega \in \Omega_1$ and $t \in (0,1)$, we take $j \in \N$ and $k=k_j(t)$ and since, for all $m$, $\mathcal{G}_{j,k_j(t),m}$ has values in $\{0,\ldots,2^m \}$ we conclude that there are infinitely many $m$ for which, for every $S \in \mathcal{S}_{j,k_{j}(t),m}$, $\mathcal{B}_{j,k,m}(S)=0$. Considering such a $m$ and $S= \lambda_{j+m}(t)$ then we first remark that, for all $1 \leq n \leq m$, $I_n=\lambda_{j+n}(t)$ and thus $T_n \in 3\lambda_{j+n}(t)$. Now, as $\mathcal{B}_{j,k,m}(\lambda_{j+m}(t))=0$, one can find $1 \leq n \leq m$ and $1 \leq \ell \leq \ell_0$ such that
\[ |\wt{\varepsilon_{T_n^\ell}}^M(\omega)| \geq 2^{-1} C_{\Psi,H_1,H_2}.\]
Thus we have showed that, for all $\omega \in \Omega_1$ and $t \in (0,1)$ there exist infinitely many $j' \in \N$ such that
\[   \max_{\begin{matrix}
\lambda' \in S_{\lambda,\lfloor \log_2 (\ell_0NM) \rfloor +2}\\
\lambda \in 3\lambda_{j'}(t)
\end{matrix}} \lln \wt{\ep_{\lambda'}}^M(\omega) \rrn \geq 2^{-1} C_{\Psi,H_1,H_2}.\]
To pass to the wavelet leaders, in the spirit of Propositions \ref{Prop:ordinleader} and \ref{Prop:rapid}, we would need to get from Borel-Cantelli Lemma an upper bound of 
\begin{align*}
\max_{\begin{matrix}
\lambda' \in S_{\lambda,\lfloor \log_2 (\ell_0NM) \rfloor +2}\\
\lambda \in 3\lambda_{j'}(t)
\end{matrix}} \lln \wc{\ep_{\lambda'}}^M(\omega) \rrn
\end{align*}
for \textit{all} $j$ sufficiently large on an event of probability $1$ which does not depend on $t$. Then, as
\begin{align*}
\mathbb{P}&\lp \exists \lambda \in \Lambda_j, \, \lambda \subseteq [0,1] \, : \,  \max_{\begin{matrix}
 \lambda'' \in S_{\lambda',\lfloor \log_2( \ell_0 NM) \rfloor +2}\\
\lambda' \in 3\lambda
\end{matrix}} \lln \wc{\ep_{\lambda'}}^M(\omega) \rrn \geq  C'  M^{\max\{H_1,H_2\}-1} j \rp \\ &\leq 2^j 4\ell_0 N M \exp(-\overset{\star}{C} C' (C^*_{\Psi,H_1,H_2})^{-1}  j),
\end{align*}
 if $C' > \log(2)C^*_{\Psi,H_1,H_2}/\overset{\star}{C}$ this probability is the general term of some convergent series and in this case one can affirm the existence of an event $\Omega'_1$ of probability $1$ such that, for all $\omega \in  \Omega'_1$ there exist $J \in \N$ such that, for all $j \geq J$, for all $\lambda \in \Lambda_j$, $\lambda \subseteq [0,1]$,
\begin{align*} \max_{\begin{matrix}
 \lambda'' \in S_{\lambda',\lfloor \log_2( \ell_0 NM) \rfloor +2}\\
\lambda' \in 3\lambda
\end{matrix}} \lln \wc{\ep_{\lambda'}}^M(\omega) \rrn \leq C'  M^{\max\{H_1,H_2\}-1} j.
\end{align*}
It seems to be the sharper upper bound that we can hope to find with our constraints and the fact that we don't have any independence property to take advantage of when dealing with the random variables $\wc{\ep_{\lambda}}^M$. This is insufficient to consider properly limit \eqref{eqn:lastremak}. Nevertheless, if, instead of working with an uniform constant $M$ we make it depends on the scale $j$ by setting $M_j=(4C' C_{\Psi,H_1,H_2}^{-1} j)^{\frac{1}{1-\max \{H_1,H_2\}}}$, where $C'>\log(2)C^*_{\Psi,H_1,H_2}/\overset{\star}{C}$ is the same constant as in Lemma \ref{Lemma:rapid2},
\[\lambda_{j,k}^{M_j} :=\left] \frac{k-NM_j}{2^j}, \frac{k+N}{2^j} \right]^2,\]
\[  \wt{c_{j,k}}^{M_j}  = c_{H_1,H_2} \int_{\lambda_{j,k}^{M_j}}'\int_{-N}^N \Psi(x) \int_{\frac{k}{2^j}}^{\frac{x+k}{2^j}} f_{H_1,H_2}(s,x_1,x_2) \, ds \, dx\, dB(x_1)\, dB(x_2) \]
and
\[ \wc{c_{j,k}}^{M_j}=c_{j,k}-\wt{c_{j,k}}^{M_j}\]
then Proposition \ref{prop:normel2} stills holds if we replace $M$ by ${M_j}$ with $j$ sufficiently large and, by directly adapting what precedes one can find on event $\Omega_1^*$ of probability $1$ such that, for all $\omega \in \Omega_1^*$ and $t \in (0,1)$ there exist infinitely many $j \in \N$ such that\footnote{The random variables $\wt{\ep_{\lambda'}}^{M_j}$ and $\wc{\ep_{\lambda'}}^{M_j}$ are defined in an obvious way.}
\[   \max_{\begin{matrix}
\lambda' \in S_{\lambda,\lfloor \log_2 (\ell_0NM_{j}) \rfloor +2}\\
\lambda \in 3\lambda_{j}(t)
\end{matrix}} \lln \wt{\ep_{\lambda'}}^{M_j}(\omega) \rrn \geq 2^{-1} C_{\Psi,H_1,H_2}.\]
while there exist $J \in \N$ such that, for all $j \geq J$, for all $\lambda \in \Lambda_j$, $\lambda \subseteq [0,1]$,
\begin{align*} \max_{\begin{matrix}
 \lambda'' \in S_{\lambda',\lfloor \log_2( \ell_0 N{M_j}) \rfloor +2}\\
\lambda' \in 3\lambda
\end{matrix}} \lln \wc{\ep_{\lambda'}}^{M_j}(\omega) \rrn & \leq C'  (M_{j})^{\max\{H_1,H_2\}-1} j \\
& \leq 4^{-1} C_{\Psi,H_1,H_2} .
\end{align*}
As a consequence, as in Proposition \ref{Prop:ordinleader}, for all $J \in \N$ there exist $j \geq J$ with
\begin{align*}
 d_j(t,\omega) \geq 2^{-j(H_1+H_2-1)} (4C' C_{\Psi,H_1,H_2}^{-1} j)^{\frac{1-H_1-H_2}{1-\max\{H_1,H_2\}}} (4 \ell_0 N )^{1-H_1-H_2} 4^{-1} C_{\Psi,H_1,H_2}
\end{align*}
which allows to state that, for all $t \in (0,1)$ and $\omega \in \Omega_1$,
\[ \limsup_{j \to + \infty} \frac{ d_j(t,\omega)}{2^{-j(H_1+H_2-1)} j^{\frac{1-H_1-H_2}{1-\max\{H_1,H_2\}}}} > 0\]
and thus, for all $\omega \in \Omega_1$ and for all $t \in (0,1)$,
\[  \limsup_{s \to t} \dfrac{|R_{H_1,H_2}(t,\omega)-R_{H_1,H_2}(s,\omega) |}{|t-s|^{H_1+H_2-1} (\log |t-s|^{-1})^{\frac{1-H_1-H_2}{1-\max\{H_1,H_2\}}}} >0.\]
In particular, we find an almost sure uniform lower modulus of continuity for the generalized Rosenblatt process, similar to the one established in \cite{nourdin21} for the Rosenblatt process. However, we are not able to judge the optimality of this modulus, which seems to be a difficult problem, as already stated in \cite[Remark 1.2]{ayache20}. 
\end{Rmk}
An interesting corollary of Remark \ref{rmk:last} and Proposition \ref{prop:rapid} is the fact that, almost surely, the pointwise H\"older exponent of the generalized Rosenblatt process is everywhere $H_1+H_2-1$ and, in particular, it is nowhere differentiable.

Similarly, one can also take $(M_j= (4 C'C_{\Psi,H_1,H_2}^{-1}\log(j)^\frac{1}{1-\max \{H_1,H_2 \}})_j$, where $C'$ is this time the same constant that in Lemma \ref{Lemma:leaderord2} and show, precisely like in this Lemma, that there exists a deterministic constant $C'>0$ such that, for all $t \in (0,1)$  there exists $\Omega_{t,2} \subset \Omega$ with probability 1 such that for all $\omega \in \Omega_{t,2}$ there exist $J \in \N$ such that, for all $j \geq J$,
\begin{align*}
    \max_{\begin{matrix}
\lambda' \in S_{\lambda,\lfloor \log_2( \ell_0NM_j) \rfloor +2}\\
\lambda \in 3\lambda_j(t)
\end{matrix}} \lln \wc{\ep_{\lambda'}}^{M_j}(\omega) \rrn \leq 4^{-1} C_{\Psi,H_1,H_2}.
\end{align*}
and conclude in the same way that there exists an event of probability $1$ such that, for all $\omega$ in this event and for almost every $t \in (0,1)$
\[  \limsup_{s \to t} \dfrac{|R_{H_1,H_2}(t,\omega)-R_{H_1,H_2}(s,\omega) |}{|t-s|^{H_1+H_2-1} (\log \log |t-s|^{-1})^{\frac{1-H_1-H_2}{1-\max\{H_1,H_2\}}}} >0\]

\bigskip

\noindent \textbf{Acknowledgement.} Both authors are supported by the FNR OPEN grant APOGEe at University of Luxembourg. 

The authors thank C\'eline Esser from University of Li\`ege, Ivan Nourdin from University of Luxembourg and St\'ephane Seuret from University Paris Est Cr\'eteil for fruitful discussions and valuable advices.

\bibliography{biblio}{}

\begin{thebibliography}{10}

\bibitem{albin-1998}
J.~M.~P. Albin.
\newblock A note on {R}osenblatt distributions.
\newblock {\em Statist. Probab. Lett.}, 40(1):83--91, 1998.

\bibitem{arneodoadn}
{A}. Arneodo, {Y}. d'{A}ubenton Carafa, {E}. Bacry, {P}.{V}. Graves, {J}.{F}.
  Muzy, and {C}. Thermes.
\newblock Wavelet based fractal analysis of {D}{N}{A} sequences.
\newblock {\em Physica {D}.: {N}onlinear {P}henomena}, (96):291--320, 1996.

\bibitem{Arneodo:95}
A.~Arneodo, J.-F. Muzy, and E.~Bacry.
\newblock Wavelets and multifractal formalism for singular signals: application
  to turbulence data.
\newblock {\em Physical Review Letters}, (67):3515--3518, 1991.

\bibitem{ayache}
A.~Ayache.
\newblock {\em Multifractional stochastic fields}.
\newblock World Scientific Publishing Co. Pte. Ltd., Hackensack, NJ, 2019.
\newblock Wavelet strategies in multifractional frameworks.

\bibitem{ayache20}
A.~Ayache.
\newblock Lower bound for local oscillations of {H}ermite processes.
\newblock {\em Stoch. Process. Their Appl.}, 130:4593--4607, 2020.

\bibitem{ayacheesmili}
{A}. Ayache and {Y}. Esmili.
\newblock Wavelet-type expansion of the generalized {R}osenblatt process and
  its rate of convergence.
\newblock {\em J. {F}ourier {A}nal. {A}ppl.}, 26(3):Paper No. 51, 35, 2020.

\bibitem{ayacheesserkleyntssens}
A.~Ayache, {C}. Esser, and {T}. Kleyntssens.
\newblock Different possible behaviors of wavelet leaders of the {B}rownian
  motion.
\newblock {\em Statist. Probab. Lett.}, 150:54--60, 2019.

\bibitem{bardet-tudor-2010}
J.-M. Bardet and C.~A. Tudor.
\newblock A wavelet analysis of the {R}osenblatt process: chaos expansion and
  estimation of the self-similarity parameter.
\newblock {\em Stochastic Process. Appl.}, 120(12):2331--2362, 2010.

\bibitem{barralseuret}
{J}. Barral and {S}. Seuret.
\newblock From multifractal measures to multifractal wavelet series.
\newblock {\em J. {F}ourier {A}nal {A}ppl.}, (11):589--614, 2005.

\bibitem{Bastin:14}
F.~Bastin, C.~Esser, and S.~Jaffard.
\newblock Large deviation spectra based on wavelet leaders.
\newblock {\em Rev. Mat. Iberoam.}, 32(3):859--890, 2016.

\bibitem{brundekropp}
{A}. Bunde, {J}. Kropp, and {H}.{J}. Schellnhuber.
\newblock {\em The science of disasters: climate disruptions, heart attacks,
  and market crashes}, volume~2.
\newblock Springer-Verlag, 2002.

\bibitem{chaurasia2019performance}
Amit Chaurasia.
\newblock Performance of synthetic {R}osenblatt process under multicore
  architecture.
\newblock In {\em 2019 3rd International conference on Electronics,
  Communication and Aerospace Technology (ICECA)}, pages 377--381. IEEE, 2019.

\bibitem{Clausel:11}
M.~Clausel and S.~Nicolay.
\newblock Wavelet techniques for pointwise anti-{H}{\"o}lderian irregularity.
\newblock {\em Constr. Approx.}, 33:41--75, 2011.

\bibitem{Daubechies:88}
I.~Daubechies.
\newblock Orthonormal bases of compactly supported wavelets.
\newblock {\em Comm. Pure App. Math.}, 41:909--996, 1988.

\bibitem{Daubechies:92}
I.~Daubechies.
\newblock {\em Ten lectures on wavelets}.
\newblock CBMS-NSF Regional Conference Series in Applied Mathematics, 1992.

\bibitem{deklsn:17}
A.~{D}eli{\`e}ge, T.~{K}leyntssens, and S.~{N}icolay.
\newblock Mars topography investigated through the wavelet leaders method: a
  multidimensional study of its fractal structure.
\newblock {\em Planet. Space Sci.}, 136:46--58, 2017.

\bibitem{dobrushin}
{R}.{L}. Dobrushin and {P}. Major.
\newblock Non-central limit theorems for nonlinear functionals of gaussian
  fields.
\newblock {\em Z. {W}ahrsch. {V}erw. {G}ebiete}, (50):27--52, 1979.

\bibitem{esserloosveldt}
{C}. Esser and {L}. Loosveldt.
\newblock Slow, ordinary and rapid points for {G}aussian {W}avelets {S}eries
  and application to {F}ractional {B}rownian {M}otions.
\newblock {\em Submitted for publication}, 2021.

\bibitem{flandrin}
{P}. Flandrin.
\newblock {T}ime-frequency / {T}ime-scale analysis.
\newblock In {\em Wavelet Analysis and its Applications}. Academic {P}ress,
  1999.

\bibitem{jaff97siam}
{S}. Jaffard.
\newblock Multifractal formalism for functions part {I}: {R}esults valid for
  all functions.
\newblock {\em SIAM J. Math. Anal.}, (28):944--970, 1997.

\bibitem{Jaffard:04b}
S.~Jaffard.
\newblock Wavelet techniques in multifractal analysis, fractal geometry and
  applications: A jubilee of {B}enoit {M}andelbrot.
\newblock {\em Proceedings of Symposia in Pure Mathematics}, 72:91--151, 2004.

\bibitem{jaffardabryroux}
{S}. Jaffard, {P}. {Abry}, and {S}. {R}oux.
\newblock Function spaces vs. {S}caling functions: {S}ome issues inimage
  classification.
\newblock {\em Mathematical Image processing}, pages 1--40, 2011.

\bibitem{jaffardmeyer96}
S.~Jaffard and Y.~Meyer.
\newblock Wavelet methods for pointwise regularity and local oscillations of
  functions.
\newblock {\em Mem. Amer. Math. Soc.}, 123(587):x+110, 1996.

\bibitem{janson97}
{S}. Janson.
\newblock {\em Gaussian {H}ilbert {S}paces}.
\newblock Cambridge {T}racts in Mathematics. Cambridge {U}niversity {P}ress,
  1997.

\bibitem{kahane85}
J.-P. Kahane.
\newblock {\em Some random series of functions}, volume~5 of {\em Cambridge
  Studies in Advanced Mathematics}.
\newblock Cambridge University Press, Cambridge, second edition, 1985.

\bibitem{nourdin21}
G.~Kerchev, {I}. Nourdin, {E}. Saksman, and {L}. Viitasaari.
\newblock Local times and sample path properties of the {R}osenblatt process.
\newblock {\em Stoch. Process. Their Appl.}, 131:498--522, 2021.

\bibitem{kreit18}
D.~Kreit and S.~Nicolay.
\newblock Generalized pointwise {H}{\"o}lder spaces defined via admissible
  sequences.
\newblock {\em J. Funct. Spaces}, ID 8276258:11, 2018.

\bibitem{lakhel2019existence}
{E}.~{H}. Lakhel and {A}. Tlidi.
\newblock Existence, uniqueness and stability of impulsive stochastic neutral
  functional differential equations driven by rosenblatt process with
  varying-time delays.
\newblock {\em Random Operators and Stochastic Equations}, 27(4):213--223,
  2019.

\bibitem{lemeyer}
P.{G}. Lemari{\'e} and {Y}. {M}eyer.
\newblock Ondelettes et bases hilbertiennes.
\newblock {\em {R}ev. {M}at. {I}beroam.}, 2:1--18, 1986.

\bibitem{levy48}
{P}. L{\'e}vy.
\newblock {\em Processus stochastiques et mouvement brownien}.
\newblock Gauthier-Villars, Paris, 1948.

\bibitem{loonic21}
L.~Loosveldt and S.~Nicolay.
\newblock Generalized spaces of pointwise regularity: {T}oward a general
  framework for the {W}{L}{M}.
\newblock {\em Nonlinearity}, 34:6561--6586, 2021.

\bibitem{tudor12}
M.~Maejima and C.A. Tudor.
\newblock Selfsimilar processes with stationnary increments in the second
  {W}iener chaos.
\newblock {\em Probab. Math. Stat.}, 32(1):167--186, 2012.

\bibitem{maejima-tudor-2013}
{M}. Maejima and {C}.{A}. Tudor.
\newblock On the distribution of the {R}osenblatt process.
\newblock {\em Statist. Probab. Lett.}, 83(6):1490--1495, 2013.

\bibitem{Mallat:99}
S.~Mallat.
\newblock {\em A {W}avelet {T}our of {S}ignal {P}rocessing}.
\newblock Academic Press, 1999.

\bibitem{Meyer:95}
Y.~Meyer and D.~Salinger.
\newblock {\em Wavelets and operators}, volume~1.
\newblock Cambridge university press, 1995.

\bibitem{meyersellantaqqu99}
Y.~Meyer, F.~Sellan, and M.~S. Taqqu.
\newblock Wavelets, generalized white noise and fractional integration: the
  synthesis of fractional {B}rownian motion.
\newblock {\em J. Fourier Anal. Appl.}, 5(5):465--494, 1999.

\bibitem{Nicolay:07}
S.~Nicolay, M.~Touchon, B.~Audit, Y.~d{\rq}Aubenton Carafa, C.~Thermes,
  A.~Arneodo, et~al.
\newblock Bifractality of human {DNA} strand-asymmetry profiles results from
  transcription.
\newblock {\em Phys. Rev. E}, 75:032902, 2007.

\bibitem{oksendal2013stochastic}
Bernt Oksendal.
\newblock {\em Stochastic differential equations: an introduction with
  applications}.
\newblock Springer Science \& Business Media, 2013.

\bibitem{rosenblatt-1956}
M.~Rosenblatt.
\newblock A central limit theorem and a strong mixing condition.
\newblock {\em Proc. Nat. Acad. Sci. U.S.A.}, 42:43--47, 1956.

\bibitem{sakthivel2018retarded}
{R}. Sakthivel, {P}. Revathi, {Y}. Ren, and {G}. Shen.
\newblock Retarded stochastic differential equations with infinite delay driven
  by rosenblatt process.
\newblock {\em Stochastic analysis and applications}, 36(2):304--323, 2018.

\bibitem{schwartz:78}
L.~Schwartz.
\newblock {\em Th{\'e}orie des distributions}.
\newblock Hermann, 1978.

\bibitem{taqqu-2011}
M.S. Taqqu.
\newblock The rosenblatt process.
\newblock In Richard~A. Davis, Keh-Shin Lii, and Dimitris~N. Politis, editors,
  {\em Selected Works of Murray Rosenblatt}, pages 29--45. Springer New York,
  New York, NY, 2011.

\bibitem{tudor-viens-2009}
{C}.{A}. Tudor and {F}.{G}. Viens.
\newblock Variations and estimators for self-similarity parameters via
  {M}alliavin calculus.
\newblock {\em Ann. Probab.}, 37(6):2093--2134, 2009.

\bibitem{veillette-taqqu-2013}
{M}.{S}.. Veillette and {M}.~S{}. Taqqu.
\newblock Properties and numerical evaluation of the {R}osenblatt distribution.
\newblock {\em Bernoulli}, 19(3):982--1005, 2013.

\bibitem{wendtabryjaf}
{H}. Wendt, {P}. Abry, {S}. Jaffard, {H}. Ji, and {Z}. Shen.
\newblock Wavelet leader multifractal analysis for texture classification.
\newblock {\em {I}{C}{I}{P}}, pages 3829--3832, 2009.

\bibitem{wendtbis}
{H}. Wendt, {S}. Roux, {S}. Jaffard, and {P}. Abry.
\newblock Wavelet leaders and bootstrap for multifractal analysis of images.
\newblock {\em Signal Processing}, (89(6)):1100--1114, 2009.

\bibitem{wiener76}
{N}. Wiener.
\newblock {\em Collected works}, volume~1.
\newblock The {M}{I}{T} {p}ress, 1976.

\bibitem{paleywiener}
{N}. Wiener and {R}.{C}. {P}aley.
\newblock Fourier transforms in the complex domain.
\newblock {\em Amer. {M}ath. {S}oc. {C}olloq. {P}ub.}, 19, 1934.
\newblock 183pp.

\end{thebibliography}
\bibliographystyle{plain}

\end{document}